\documentclass[preprint,review,11pt,authoryear]{elsarticle}

\let\today\relax
\makeatletter
\def\ps@pprintTitle{%
    \let\@oddhead\@empty
    \let\@evenhead\@empty
    \def\@oddfoot{\footnotesize\itshape
         {Preprint submitted to arXiv} \hfill\today}%
    \let\@evenfoot\@oddfoot
    }
\makeatother

\usepackage{amsmath,amssymb}
 \usepackage{amsthm}
\usepackage{bbm}       
\usepackage{bm}        
 \usepackage{graphicx}
\usepackage{color}
\definecolor{cornell-red}{RGB}{179,27,27}
\usepackage{subcaption}
\usepackage{mathtools}
\usepackage{dsfont}
\usepackage{color}
\usepackage{natbib}
    \usepackage{multirow}
\usepackage{pgfplots}
\usepackage{caption,subcaption}
\usepackage[left=1in,right=1in,top=1in,bottom=1in]{geometry}
\usepackage{indentfirst}
\usepackage[lined,boxed,commentsnumbered, ruled]{algorithm2e}   

\usepackage{enumitem}  
\usepackage{xcolor}
\usepackage[colorlinks = true,linkcolor = blue,urlcolor  = blue,citecolor = blue,anchorcolor = blue]{hyperref}

\usepackage{makecell, hhline}  

\usepackage{natbib}

\graphicspath{ {Figures/} {TikzPlots/} }

\usepackage{titlesec} 
\titlespacing*{\section}
{0pt}{1.0ex}{1.0ex}
\titlespacing*{\subsection}
{0pt}{1.0ex}{1.0ex}

\newtheorem{theorem}{Theorem}
\newtheorem{lemma}{Lemma}
\newtheorem{proposition}{Proposition}
\newtheorem{corollary}{Corollary}

\theoremstyle{definition}
\newtheorem{definition}{Definition}[section]

\newtheorem{example}{Example}

\newtheorem{remark}{\textbf{Remark}}
\newtheorem{axiom}{Axiom}[section]

\theoremstyle{remark}

\makeatletter
\newcommand{\setaxiomtag}[1]{
  \let\oldtheaxiom\theaxiom
  \renewcommand{\theaxiom}{#1}
  \g@addto@macro\endaxiom{
    \addtocounter{axiom}{-1}
    \global\let\theaxiom\oldtheaxiom}
  }
\makeatother

\DeclareMathOperator*{\argmax}{arg\,max}
\DeclareMathOperator*{\argmin}{arg\,min}

\DeclareMathOperator{\dom}{dom}

\DeclareMathOperator{\conv}{conv} 

\DeclareMathOperator{\sgn}{sgn}

\mathchardef\mhyphen="2D 

\newcommand{\N}{\mathbb{N}}

\newcommand{\R}{\mathbb{R}}
\newcommand{\Z}{\mathbb{Z}}

\newcommand{\norm}[1]{\left\lVert#1\right\rVert}
\newcommand{\norms}[1]{\lVert#1\rVert}

\newcommand{\calE}{\mathcal{E}}

\newcommand{\calJ}{\mathcal{J}}

\newcommand{\calL}{\mathcal{L}}

\newcommand{\calP}{\mathcal{P}}

\newcommand{\calS}{\mathcal{S}}

\newcommand{\calU}{\mathcal{U}}
\newcommand{\calV}{\mathcal{V}}
\newcommand{\calW}{\mathcal{W}}
\newcommand{\calX}{\mathcal{X}}

\newcommand{\one}{\bm{1}}
\newcommand{\zero}{\bm{0}}
\newcommand{\tp}{\top}

\newcommand{\ubar}{\overline{u}}

\newcommand{\wt}{\widetilde{w}}
\newcommand{\wbt}{\widetilde{\bm{w}}}

\newcommand{\calWt}{\widetilde{\mathcal{W}}}
\newcommand{\Umax}{U_\textup{max}}
\newcommand{\Umin}{U_\textup{min}}

\newcommand{\rb}{\bm{r}}

\newcommand{\ub}{\bm{u}}

\newcommand{\wb}{\bm{w}}
\newcommand{\xb}{\bm{x}}
\newcommand{\yb}{\bm{y}}
\newcommand{\zb}{\bm{z}}
\newcommand{\bb}{\bm{b}}
\newcommand{\lambdab}{\bm{\lambda}}
\newcommand{\thetab}{\bm{\theta}}

\newcommand{\ubt}{\widetilde{\bm{u}}}

\newcommand{\wbar}{\overline{w}}
\newcommand{\lambdabar}{\overline{\lambda}}

\newcommand{\Rup}{\R_{\uparrow}}

\newcommand{\eb}{\bm{e}}
\newcommand{\ubbar}{\overline{\bm{u}}}

\newcommand{\gt}{\tilde{g}}
\newcommand{\ut}{\tilde{u}}

\usepackage{titlesec}

\begin{document}

\begin{frontmatter}

\title{A Unified Framework for Analyzing and Optimizing a Class of Convex Fairness Measures}

\author[mymainaddress1]{Man Yiu Tsang}
\cortext[cor1]{Corresponding author. }
\ead{mat420@lehigh.edu}
\author[mymainaddress1]{Karmel S.~Shehadeh\corref{cor1}}
\ead{kas720@lehigh.edu}

\address[mymainaddress1]{Department of Industrial and Systems Engineering, Lehigh University, Bethlehem, PA,  USA}

\begin{abstract}

\noindent We propose a new framework that unifies different fairness measures into a general, parameterized class of convex fairness measures suitable for optimization contexts. First, we propose a new class of order-based fairness measures, discuss their properties, and derive an axiomatic characterization for such measures. Then, we introduce the class of convex fairness measures, discuss their properties, and derive an equivalent dual representation of these measures as a robustified order-based fairness measure over their dual sets. Importantly, this dual representation renders a unified mathematical expression and an alternative geometric characterization for convex fairness measures through their dual sets. Moreover, it allows us to develop a unified framework for optimization problems with a convex fairness measure objective or constraint, including unified reformulations and solution methods. In addition, we provide stability results that quantify the impact of employing different convex fairness measures on the optimal value and solution of the resulting fairness-promoting optimization problem. Finally, we present numerical results demonstrating the computational efficiency of our unified framework over traditional ones and illustrating our stability results.


\end{abstract}
\end{frontmatter}

\section{Introduction}\label{sec:intro}
Fairness concerns arise naturally in various decision-making contexts and application domains. Examples of real-life settings where fairness is a concern include healthcare management, transportation, facility location, humanitarian operations, and resource allocation \citep{Bertsimas_et_al:2013, Breugem_et_al:2022, Filippi_et_al:2021, Gutjahr_et_al:2018, Li_et_al:2022, Rahmattalabi_et_al:2022, Sun_et_al:2022, Sun_et_al:2023}. Thus, it is necessary to incorporate fairness measures in optimization models to address these concerns and make fair decisions (i.e., decisions that have an equal impact on the individuals or groups affected by them). 

Like other factors in an optimization context, one can account for fairness by putting related measures in the objective function or constraints. To illustrate, let us consider the following general setting. Suppose we have $N$ subjects of interest, individuals, or groups that are grouped based on some criteria (e.g., location, income, or race). Let $u_i$ represent the impact (or outcome) of a given decision $\xb\in\R^n$ on subject $i \in [N]:=\{1,\dots,N\}$. For example, in a facility location problem, $\xb$ may represent facility location decisions, and $\ub$ may represent customers' unmet demand or travel time to the nearest open facility. Using this notation, we define the following general optimization problem:
 \begin{equation}\label{prob:eff_min}
    \min_{\xb,\,\ub} \Big\{ f(\ub) \,\Big|\, \ub = U(\xb),\, \xb\in\calX\Big\},
\end{equation}
where $\calX\subseteq\R^n$ is the set of constraints on $\xb$ and $U:\R^n\rightarrow\R^N$ is a function for computing the impact of a given decision on the considered subjects. Formulation \eqref{prob:eff_min} seeks to find optimal decisions $\xb \in \calX$ that optimize some classical fairness-neutral objective function $f$ (often called efficiency or inefficiency measure). For example, in the facility location problem, $f$ may represent the sum of the unmet demand or the total travel time to open facilities. Formulation \eqref{prob:eff_min} may produce unfair decisions because it has no measure to ensure that the optimal decisions have fair outcomes. Thus, to promote fairness, the literature typically includes a fairness measure in the objective or constraint of the optimization model. For example, let $\phi:\R^N\rightarrow\R$ represent a fairness measure that gauges the degree of unfairness or inequality.\footnote{\cite{Chen_Hooker:2023}, \cite{Karsu_Morton:2015}, and \cite{Shehadeh_Snyder:2021} have pointed out in their comprehensive surveys that terms like ``fairness and equity,'' ``unfairness and inequity,'' ``fairness and equality,'' and ``unfairness and inequality'' have been used interchangeably in the context of fairness-related optimization. Likewise, in this paper, we adopt the same approach and use these terms interchangeably but note that there can be differences between them when they are formally defined. Furthermore, for simplicity, we define ``unfairness'' as the disparity or inequality in impacts across the considered subjects or groups.} Then, one could incorporate  $\phi$ in the objective function of \eqref{prob:eff_min} as follows:
\begin{equation}\label{prob:FM_eff_min_obj}
    \min_{\xb,\,\ub} \Big\{ f(\ub) + \gamma \phi(\ub) \,\Big|\, \ub = U(\xb),\, \xb\in\calX\Big\},
\end{equation}
for some $\gamma>0$, which controls the trade-off between efficiency and fairness. Alternatively, one could incorporate $\phi$ in the constraint by imposing an upper bound $\eta>0$ on $\phi$ as follows:
\begin{equation}\label{prob:FM_eff_min_constraint}
    \min_{\xb,\,\ub} \Big\{ f(\ub) \,\Big|\, \phi(\ub)\leq \eta,\, \ub = U(\xb),\, \xb\in\calX\Big\}.
\end{equation}

Unfortunately, quantifying fairness is challenging because there is no \textit{best} definition or measure of fairness that is universally accepted. Instead, there is a wide variety of notions and measures to gauge the level of fairness in the economics, decision theory, and operations research literature \citep{Karsu_Morton:2015, Shehadeh_Snyder:2021}. Indeed, different fairness measures may produce remarkably different conclusions (i.e., could yield different solutions with varying impacts on fairness). In addition, different measures have different mathematical expressions and characteristics. Hence, the choice of fairness measure has important consequences for the decisions and tractability of the resulting fairness-promoting optimization models such as \eqref{prob:FM_eff_min_obj} and \eqref{prob:FM_eff_min_constraint}. As such, customized optimization frameworks have been proposed to analyze and incorporate fairness measures in different contexts. To see this, in Table \ref{table:existing_FM}, we present a set of eight widely employed deviation-based fairness measures in the literature.  

\begin{table}[t]  
\footnotesize
\center
\renewcommand{\arraystretch}{1.0}
\caption{Existing deviation-based fairness measures} 
\begin{tabular}{lll}
\Xhline{1.0pt}
Index & Measure & Name \\ \Xhline{0.5pt}
i.    & $\max_{i\in[N]}u_i - \min_{i\in[N]} u_i$ &  Range \\
ii.   & $\sum_{i=1}^N \sum_{j=1}^N |u_i-u_j|$ & Gini deviation  \\
iii.  & $\max_{i\in[N]} \max_{j\in[N]} |u_i - u_j|$ & Maximum pairwise deviation \\
iv.   & $\sum_{i=1}^N |u_i-\ubar|$ & Absolute deviation from mean \\
v.    & $\big[\sum_{i=1}^N (u_i-\ubar)^2\big]^{1/2}$ & Standard deviation\\
vi.   & $\max_{i\in[N]} |u_i-\ubar|$ & Maximum absolute deviation from mean  \\
vii.  & $\max_{i\in[N]} \sum_{j=1}^N |u_i-u_j|$ & Maximum sum of pairwise deviation\\
viii. & $\sum_{i=1}^N \max_{j\in[N]} |u_i-u_j|$ & Sum of maximum pairwise deviation\\
\Xhline{1.0pt}
\end{tabular}\label{table:existing_FM}
\end{table}

We make the following observations about the measures in Table \ref{table:existing_FM}. First, in \ref{appdx:eg_inequity_measures}, we prove that when $N>3$, only measures (i) and (iii) are equivalent, i.e., $\phi^\text{(i)}(\ub)=\phi^\text{(iii)}(\ub)$, and measures (vi) and (vii) are equivalent, i.e., $\phi^\text{(vii)}(\ub)=N\phi^\text{(vi)}(\ub)$. Thus, using either (i) or (iii) in models \eqref{prob:FM_eff_min_obj} and \eqref{prob:FM_eff_min_constraint}  as the fairness measure will yield optimal solutions with the same impact on fairness.  Similarly, replacing $\phi^\text{(vi)}$ by $\phi^\text{(vii)}=N\phi^\text{(vi)}$ in \eqref{prob:FM_eff_min_obj} (resp. in \eqref{prob:FM_eff_min_constraint}) scales $\gamma$ (resp. $\eta$) by a factor $N$. Furthermore, since the remaining pairs of measures are not equivalent, incorporating any of them in \eqref{prob:FM_eff_min_obj} and \eqref{prob:FM_eff_min_constraint} may produce different solutions than other measures with varying impacts on fairness. 
Second, these measures have different mathematical expressions. Hence,  different reformulations and solution techniques may be required to solve the resulting fairness-promoting model that incorporates each. For example, to incorporate measure (i), we can introduce auxiliary variables to linearize the maximum and minimum operators. In contrast, we can introduce conic constraints to reformulate measure (v) or solve the resulting optimization model via non-linear optimization techniques. 

Indeed, decision-makers could benefit greatly from a unified framework that enables them to incorporate fairness measures and assess the trade-offs between different measures within their decision-making contexts. However, despite the subject's importance, few studies have proposed unifying and computationally efficient frameworks for incorporating fairness measures in optimization problems. Moreover, there is a lack of rigorous approaches for analyzing the impact of using different fairness measures on optimal solutions and fairness.

The above challenges inspire this paper’s main contribution: introducing a new framework that unifies different fairness measures into a general, parameterized class of convex fairness measures suitable for optimization contexts. Our theoretical and methodological contributions are summarized as follows.

\begin{itemize}[leftmargin=8mm]
    \item \textit{New order-based fairness measures}. We propose a new class of order-based fairness measures, discuss their properties, and derive an axiomatic characterization for such measures. This class of order-based fairness measures is the building block of our proposed unified framework for analyzing and optimizing convex fairness measures.

    \item \textit{A unified representation of the class of convex fairness measures.} We propose a unified representation of the class of convex fairness measures suitable for optimization contexts. This class includes any convex fairness measure satisfying well-known axioms that fairness measures should satisfy (see Section~\ref{sec:fairness_measures} and Definition~\ref{def:absolute_convex_fairness_measures}). We derive an equivalent dual representation of these measures as a robustified order-based fairness measure over their dual sets. Importantly, this dual representation renders a unified mathematical expression and an alternative geometric characterization for convex fairness measures through their dual sets.  Consequently, it provides a mechanism for investigating the equivalence of convex fairness measures from a geometric perspective. In \ref{appdx:relative_convex_fairness_measures}, we introduce the relative counterpart of the convex fairness measures and discuss its properties.
    
    \item \textit{A unified optimization framework.} Using the dual representation of convex fairness measures, we derive equivalent reformulations of optimization models with a convex fairness measure objective or constraint and develop decomposition methods to solve the reformulations. This unified optimization framework can be employed to solve optimization models with any convex fairness measure. In \ref{appdx:sol_approach_CFM_constraint_rel}, we derive equivalent reformulations of optimization models with relative convex fairness measures.

    \item \textit{Stability analysis.} We conduct stability analyses on the choice of convex fairness measures in optimization models via their dual representations. Specifically, we provide mechanisms to quantify the differences in optimal value and solutions of fairness-promoting optimization problems under two different convex fairness measures. This allows for studying the trade-off between considering different fairness measures in such models.
\end{itemize}

\subsection{Structure of the paper}
The remainder of the paper is organized as follows. In Section~\ref{sec:literature_review}, we review the relevant literature. In Section~\ref{sec:fairness_measures}, we discuss popular and desirable mathematical properties commonly used to study fairness measures that we also adopt to characterize the class of convex fairness measures. In Sections~\ref{sec:order_based_fairness_measures} and \ref{sec:convex_fairness_measures}, we introduce and analyze the class of ordered-based fairness measures and convex fairness measures, respectively. In Section~\ref{sec:opt_w_fairness_measures}, we derive equivalent reformulations of optimization models with a convex fairness measure objective and constraints. In Section~\ref{sec:stability_analysis}, we conduct stability analyses on the choice of convex fairness measures in optimization models. Finally, in Section~\ref{sec:num_experiments}, we present numerical results demonstrating the computational efficiency of our proposed reformulations and solution methods over existing ones using a facility location problem. In addition, we illustrate our stability results using a resource allocation problem.

\subsection{Notation}
For two integers $a$ and $b$ with $a<b$, we define the sets $[a]:=\{1,\dots,a\}$ and $[a,b]_\Z:=\{a,a+1,\dots,b\}$. We use boldface letters to denote vectors. In particular, $\eb_j$ is the $j$th standard unit vector; $\zero$ and $\one$ are vectors of zeros and ones respectively, where the dimension will be clear in the context. For a given vector $\xb$, we denote $x_{(i)}$ as the $i$th smallest entry of $\xb$. For two vectors $\xb\in\R^N$ and $\yb\in\R^N$, we say $\xb$ is majorized by $\yb$, denoted as $\xb\preceq\yb$, if $\sum_{i=1}^N x_i = \sum_{i=1}^N y_i$ and $\sum_{i=k}^N x_{(i)} \leq \sum_{i=k}^N y_{(i)}$ for $k\in[2,N]_\Z$. We use $\Rup^N$ to denote the set of $\R^N$ vectors with entries in ascending order, i.e., $\Rup^N=\{\xb\in\R^N\mid x_1\leq\dots\leq x_N\}$. Finally, we use $\conv(C)$ to denote the convex hull of a set $C$.

\section{Relevant Literature} \label{sec:literature_review}

The importance of fairness has been recognized and well-studied in various settings and research communities for decades. Early efforts and a large body of research later focused on defining appropriate measures to capture fairness and axiomatically characterize these measures \citep{Chakravarty:2007, Cowell_Kuga:1981, Donaldson_Weymark:1980, Foster:1983, Kakwani:1980, Kolm:1976, Thon:1982}. This has led to various notions and measures of fairness, each satisfying a set of axiomatic properties that other measures may not satisfy. We refer to \cite{Andreoli_Zoli:2020}, \cite{Chakravarty:1999}, \cite{Cowell:2000}, and \cite{Jancewicz:2016} for comprehensive surveys on fairness measures and their axiomatic properties. We limit the scope of this review to studies relevant to our paper, specifically those that focus on measuring (un)fairness in optimization contexts.

Assessing how decisions may positively or negatively impact individuals or groups is crucial in many decision-making contexts and application domains \citep{Jagtenberg_Mason:2020, Shen_et_al:2019, Uhde_et_al:2020}. In particular, ensuring fairness in such decision-related impact is essential. Hence, to promote fairness, fairness measures should be incorporated into the objective and/or constraints of optimization problems \citep{Barbati_Piccolo:2016, Bateni_et_al:2022, Diakonikolas_et_al:2020, Filippi_et_al:2021, Lan_et_al:2010, Sun_et_al:2022}. 

While there is a wide range of fairness measures, not all are suitable for optimization contexts. For example, as pointed out by \cite{Chen_Hooker:2023}, some fairness measures are difficult to optimize and pose computational challenges, including the Theil index \citep{Theil:1967} and Atkinson index \citep{Atkinson:1975} proposed in the economics literature. Thus, studies in the operations research literature have developed and employed fairness measures with desired mathematical properties suitable for optimization contexts, including network resource allocation \citep{Lan_et_al:2010}, facility location \citep{Barbati_Piccolo:2016, Marsh_Schilling:1994, Ogryczak:2000}, humanitarian logistics \citep{Donmez_et_al:2022}, and transportation \citep{Lewis_et_al:2021}. In this paper, we adopt a set of widely used axioms in the literature to characterize the class of convex fairness measures (see Section~\ref{sec:fairness_measures}). However, unlike existing studies focusing on a specific context, our unified framework is application-agnostic and can be employed in various optimization contexts.

Next, we briefly review several categories of fairness measures that have been employed in the optimization literature; see \cite{Chen_Hooker:2023}, \cite{Karsu_Morton:2015}, and \cite{Shehadeh_Snyder:2021} for comprehensive surveys. The first category measures the degree of equality in the distribution of utilities (e.g., resources allocated to different entities). Several well-known fairness measures such as relative range, mean deviation, and the Gini index belong to this category \citep{Cowell:2011, Gini:1912}. Another category focuses on disadvantaged entities. This includes the famous Rawlsian principle \citep{Rawls:1999}, which focuses on improving the worst-off entity (e.g., the minimum outcome level). As a result, the Rawlsian approach ignores the outcomes of all other entities. To remedy this issue, the Rawlsian approach is extended to a maximum lexicographic approach which sequentially maximizes the welfare of the worst-off, then the second worst-off, then the third worst-off, and so on \citep{Kostreva_et_al:2004, Ogryczak_Sliwinski:2006}. 

Another stream of the literature proposes new objective functions for optimization models that combine the two competing objectives: efficiency and fairness. The parameter of these new objective functions typically controls the trade-off between efficiency and fairness \citep{Bertsimas_et_al:2012, Hooker_Williams:2012, Mo_Walrand:2000, Williams_Cookson:2000}. Note that these approaches have a specific form of inefficiency measure. Unlike this stream of literature, we propose a unified framework for fairness measures, and we do not attempt to address the trade-off between efficiency and fairness. Nevertheless, decision-makers could use our proposed framework to incorporate convex fairness measures into optimization models with any inefficiency measure.

While these prior investigations have led to numerous tools for the decision-making toolbox, few studies have proposed unifying frameworks for incorporating fairness measures in optimization problems. In this paper, we contribute to the literature with a new framework that unifies different fairness measures into a general, parameterized class of convex fairness measures. As mentioned earlier, this class includes any convex fairness measure satisfying well-known axioms that fairness measures should satisfy (see Definition~\ref{def:absolute_convex_fairness_measures}). For example, it includes the proposed class of order-based fairness measures, some deviation-based measures (e.g., those in Table~\ref{table:existing_FM}), and some envy-based measures (see \ref{appdx:envy_based}), among others. In addition, our framework can be used to address fairness-promoting optimization problems integrating the Rawlsian principle (see \ref{appdx:Rawlsian}). We derive a dual representation of convex fairness measures as a robustified order-based fairness measure over their dual sets. This unified representation allows us to develop a unified framework for optimization problems with a convex fairness measure objective or constraint.

Notably, \cite{Lan_et_al:2010} is the first and, as far as we are aware (see  \citealp{Chen_Hooker:2023} and \citealp{Karsu_Morton:2015}), the only study that also adopts an axiomatic approach to unify a class of \textit{relative} fairness measures in network resource allocation. Our proposed unifying framework of the class of convex fairness measures and \cite{Lan_et_al:2010}'s axiomatic approach characterize different classes of fairness measures; 
 see \ref{appdx:comparison_with_Lan_et_al} for a detailed comparison. In particular,   \cite{Lan_et_al:2010}'s approach is based on the context of resource allocation, considers fairness in a relative sense, and adopts a slightly different set of axioms to define their class of relative fairness measures.  In contrast, our proposed framework is suitable for various optimization contexts. Moreover, \cite{Lan_et_al:2010} did not propose unified reformulations or solution techniques to solve optimization problems incorporating their relative fairness measures.  Finally, \cite{Lan_et_al:2010} did not provide mechanisms to study the stability of fairness-promoting optimization problems.

We close this section by noting that one can measure fairness in an absolute or relative sense. Yet, there is no universal consensus regarding the preference for either absolute or relative fairness measures. It often depends on the context and is a matter of subjective evaluation \citep{Chakravarty:1999}. For example, absolute fairness measures could be more appropriate for fairness-promoting optimization models of the form \eqref{prob:FM_eff_min_obj}. On the other hand, if we use fairness-promoting optimization models of the form \eqref{prob:FM_eff_min_constraint}, both absolute and relative fairness measures could be employed.  In \ref{appdx:relative_convex_fairness_measures}, we introduce the relative counterpart of the convex fairness measures and discuss its properties. Nevertheless, our paper does not focus on choosing between absolute and relative fairness measures, or between convex and order-based fairness measures.

\section{Axioms for Fairness Measures} \label{sec:fairness_measures}

Recall that fairness measures are typically employed in optimization models to gauge the level of fairness associated with a given decision. Therefore, a function  $\phi:\R^N\rightarrow\R$  should satisfy some basic properties as a fairness measure. Indeed, as discussed in Section~\ref{sec:literature_review}, the literature on fairness measures often imposes various conditions (axioms) to characterize fairness measures. In optimization contexts, these conditions include mathematical properties that allow for deriving computationally tractable fairness-promoting models. To lay the foundation for the subsequent technical discussions, in this section, we discuss a set of essential axioms for fairness measures that we adopt (see \citealp{Abul-Naga_Yalcin:2008, Chakravarty:1999, Chakravarty:2007, Chen_Hooker:2023, Karsu_Morton:2015, Lan_et_al:2010, Mussard_Mornet:2019} for further discussions).

\setaxiomtag{C}
\begin{axiom}[Continuity]\label{axiom:continuity}
$\phi$ is continuous on $\R^N$.
\end{axiom}

\cite{Chakravarty:1999} argues that continuity is a minimal condition for a fairness measure. It ensures that a small perturbation of the outcome vector $\ub$ results in a slight change in the value of the fairness measure, which implies that a small change in the decision (of an optimization model) would lead to a small change in the fairness measure. Finally, from a computational viewpoint, optimizing continuous functions is often easier than other functions.

\setaxiomtag{N}
\begin{axiom}[Normalization]\label{axiom:normalization}
$\phi(\ub)\geq 0$ for any $\ub\in\R^N$ and $\phi(\ub)=0$ if and only if $\ub=\alpha \one$ for some $\alpha\in\R$.
\end{axiom}

The normalization property guarantees that the fairness measure is non-negative and only equals zero (perfect equality) when all $u_i$, $i \in [N],$ are identical, i.e., if a decision has the same impact on all the subjects of interest. Thus, a smaller value of $\phi$ implies a fairer distribution of impact on the considered subjects.

\setaxiomtag{S}
\begin{axiom}[Symmetry]\label{axiom:symmetry}
$\phi(\ub)=\phi(P\ub)$ for any $\ub\in\R^N$ and permutation matrix $P$.
\end{axiom}

Axiom~\ref{axiom:symmetry} is also known as the anonymity axiom. It captures the notion that the identity of individuals should not change the perceived degree of fairness.

\setaxiomtag{SCV}
\begin{axiom}[Schur Convexity]\label{axiom:Schur_convex}
$\phi$ is Schur convex, i.e., if $\ub^1\preceq\ub^2$, then $\phi(\ub^1)\leq \phi(\ub^2)$.
\end{axiom}

The Schur convexity property is also known as the Pigou-Dalton condition \citep{Dalton:1920, Moulin:2004}.  This, for example, implies that transferring a positive impact from a better-off subject to a worse-off subject reduces the value of the fairness measure. Mathematically, if $\ub^1\preceq\ub^2$, one can obtain $\ub^1$ from $\ub^2$ by a finite number of Robin Hood transfers (a.k.a. progressive transfers): replacing two entries $u^2_i$ and $u^2_j$ of $\ub^2$ such that $u^2_i<u^2_j$ with $u^2_i+\varepsilon$ and $u^2_j-\varepsilon$, respectively, for some $\varepsilon\in(0,u^2_j-u^2_i)$ \citep{Marshall_et_al:2011}. Therefore, Schur convexity ensures that if $\ub^1$ is fairer than $\ub^2$ (i.e., $\ub^1\preceq\ub^2$), the value of the fairness measure evaluated at $\ub^2$ should be no less than that of $\ub^1$.

\begin{remark}
 Axiom~\ref{axiom:Schur_convex} implies Axiom~\ref{axiom:symmetry} (see Theorem 2.4 of \citealp{Stepniak:2007}), i.e., if a fairness measure $\phi$ is Schur convex, then $\phi$ is also symmetric. 
\end{remark}

\setaxiomtag{TI}
\begin{axiom}[Translation Invariance]\label{axiom:trans_invariance}
$\phi(\ub+\alpha\one)=\phi(\ub)$ for any $\alpha\in\R$.
\end{axiom}

Axiom~\ref{axiom:trans_invariance} guarantees that if we replace each $u_i$ by $u_i+\alpha$, then the value of the \textit{absolute} fairness measure will not change. This is a desirable property because while replacing $u_i$ with $u_i+\alpha$ for all $i\in[N]$ changes the outcome or impact on the subjects, it does not change the degree of inequality across the subjects.

\setaxiomtag{PH}
\begin{axiom}[Positive Homogeneity]\label{axiom:pos_homo}
$\phi(\alpha\ub)=\alpha\phi(\ub)$ for any $\alpha > 0$.
\end{axiom}

Axiom~\ref{axiom:pos_homo} requires that if the vector $\ub$ is multiplied by a factor $\alpha$ (and hence, the difference between two entries of $\ub$ is enlarged by a factor of $\alpha$), the value of $\phi$ also scales with the same factor. This implies that the fairness measure $\phi(\ub)$ has the same measurement unit as $\ub$. Note that this is a desirable property if we use fairness-promoting optimization models of the form \eqref{prob:FM_eff_min_obj}, where the objective is a sum of the inefficiency and fairness measures, since both measures $f(\ub)$ and $\phi(\ub)$ have the same unit \citep{Mostajabdaveh_et_al:2019, Zhang_et_al:2021}.

We close this section by noting that the fairness measures listed in Table~\ref{table:existing_FM} satisfy Axioms~\ref{axiom:continuity}, \ref{axiom:normalization}, \ref{axiom:symmetry}, \ref{axiom:Schur_convex}, \ref{axiom:trans_invariance}, and \ref{axiom:pos_homo} (see \ref{appdx:eg_inequity_measures}).

\section{Order-Based Fairness Measures} 
\label{sec:order_based_fairness_measures}

In this section, we introduce a new class of fairness measures, which we call the order-based fairness measures, and discuss their properties. We shall see that this new class of fairness measures is the building block of our proposed general class of convex fairness measures presented later in Section~\ref{sec:convex_fairness_measures}. First, we formally define order-based fairness measures as follows.

\begin{definition}[Order-Based Fairness Measure]  \label{def:order_based_FM}
Let $\calW=\{\wb\in\R^N\mid \sum_{i=1}^N w_i = 0,\, w_1\leq\dots\leq w_N,\, w_1< 0,\, w_N> 0\}$. A fairness measure $\nu:\R^N\rightarrow \R$ is order-based if there exists $\wb\in\calW$ such that $\nu(\ub)=\nu_{\wb}(\ub)=\sum_{i=1}^N w_i u_{(i)}$.
\end{definition}

By definition, each order-based fairness measure is characterized by a weight vector $\wb\in\calW$, where the weight $w_i$ associated with the $i$th smallest entry $u_{(i)}$ can be considered as the relative priority. Intuitively, Definition \ref{def:order_based_FM} implies that increasing the $i$th smallest outcome $u_{(i)}$ by a small amount $\delta>0$ changes $\nu_{\wb}(\ub)$ to $\nu_{\wb}(\ub)+\delta w_i$. Since $w_i\leq w_j$ when $i<j$, this implies that increasing the $i$th smallest outcome would lead to a smaller increase or larger decrease in $\nu_{\wb}(\ub)$ (depending on the signs of $w_i$ and $w_j$), and thus promoting fairness; see Example~\ref{eg:fair_resource_order_based} in \ref{appdx:eg_order_based}. In addition, we can rewrite $\nu_{\wb}(\ub)=\sum_{j=2}^N \big(\sum_{i=j}^N w_i \big) \big[ u_{(j)}-u_{(j-1)} \big]$. That is, $\nu_{\wb}(\ub)$ is a weighted sum of the differences between two consecutive values $u_{(j)}$ and $u_{(j-1)}$, which is a natural way to measure the variability of the distribution of outcomes.

In Proposition~\ref{prop:order_based_FM_is_FM}, we show that $\nu_{\wb}(\ub)$ satisfies a set of basic axioms discussed in Section~\ref{sec:fairness_measures}.

\begin{proposition} \label{prop:order_based_FM_is_FM}
For any $\wb\in\calW$, the order-based fairness measure $\nu_{\wb}:\R^N\rightarrow\R$ defined as $\nu_{\wb}(\ub)=\sum_{i=1}^N w_i u_{(i)}$ satisfies Axioms~\ref{axiom:continuity}, \ref{axiom:normalization}, \ref{axiom:symmetry}, \ref{axiom:Schur_convex}, \ref{axiom:trans_invariance}, and \ref{axiom:pos_homo}.
\end{proposition}

Note that the order-based fairness measure resembles the form of the so-called ordered weighted average (OWA) operator \citep{Csiszar:2021, Yager:1988} or the ordered median function \citep{Nickel_Puerto:2006}. In \ref{appdx:comparison_OWA}, we discuss how our order-based fairness measure differs from the OWA operators. Next, in Theorem~\ref{thm:order_based_FM_characterization_I}, we show that $\nu_{\wb}$ is a supremum of linear functions (in $\ub$) over a set of permutations of $\wb$. As we show in Section~\ref{sec:opt_w_fairness_measures}, this characterization facilitates a unified reformulation of fairness-promoting optimization models with order-based fairness measures.

\begin{theorem} \label{thm:order_based_FM_characterization_I}
Let $\Pi$ be the set of permutation functions $\pi:\R^N\rightarrow\R^N$ on $[N]$. For any $\ub\in\R^N$, we have $\nu_{\wb}(\ub)=\sup_{\wbt\in\calWt} \sum_{i=1}^N \wt_i u_i$, where $\calWt=\big\{\wbt\in\R^N\mid \wt_i=w_{\pi(i)},\, i\in[N],\, \pi\in\Pi\big\}$.
\end{theorem}

Next, we now derive an axiomatic characterization of order-based fairness measures. First, the following axiom provides a practical mechanism to quantify the change in the value of the order-based fairness measure incurred by perturbation in the outcome vector $\ub$.

\setaxiomtag{PA}
\begin{axiom}[Proportional Adjustment] \label{axiom:order_based}
There exist constants $\{w_j\}_{j\in[N]}$ such that $\phi(\ub+\varepsilon\eb_j)=\phi(\ub)+\varepsilon w_j$ for any $\ub\in\Rup^N$, $j\in[N]$, and $\varepsilon\in[0,u_{j+1}-u_j]$, where $[0,u_{N+1}-u_N]=[0,\infty)$.
\end{axiom}

Axiom~\ref{axiom:order_based} has the following practical interpretation. Suppose we increase the $j$th smallest entry $u_{(j)}$ by a small amount $\varepsilon$. Then, the associated change in the value of the order-based fairness measure is proportional to $\varepsilon$. Mathematically, Axiom~\ref{axiom:order_based} implies that $\phi(\ub')=\phi(\ub)+\sum_{j=1}^N w_j (u'_j-u_j)$ for any perturbation $\ub'\in\R^N$ of the vector $\ub\in\R^N$ such that the order of each subject is preserved. Thus, Axiom~\ref{axiom:order_based} characterizes the linearity of the order-based fairness measure with respect to  $(u_{(1)},\dots,u_{(N)})^\tp$. It follows that the effect of a change in the impact $u_j$ of subject $j$ depends only on the order (i.e., rank) of that subject, irrespective of the value of the impact (see \citealp{Mehran:1976} for similar discussions). In Theorem \ref{thm:axiom_order_based}, we provide an axiomatic characterization of order-based fairness measures.

\begin{theorem} \label{thm:axiom_order_based}
The function $\nu:\R^N\rightarrow\R$ is an order-based fairness measure with weight vector $\wb\in\calW$ if and only if $\nu$ satisfies Axioms~\ref{axiom:normalization}, \ref{axiom:Schur_convex}, \ref{axiom:pos_homo}, and \ref{axiom:order_based}.
\end{theorem}

Among the set of popular measures listed in Table~\ref{table:existing_FM}, only measures (i)--(iii) are order-based (see Remark~\ref{rem:eg_order-based}), where measure (i) is equivalent to measure (iii) (see \ref{appdx:eg_inequity_measures}). In Examples~\ref{eg:axiom_characterization_range}--\ref{eg:axiom_characterization_Gini_deviation} in \ref{appdx:eg_order_based}, we show how one can apply Theorem~\ref{thm:axiom_order_based} to provide axiomatic characterizations of these order-based fairness measures, namely (i) and (ii).

\section{Convex Fairness Measures}  \label{sec:convex_fairness_measures}

In this section, we introduce our proposed class of convex fairness measures. We also derive a unified dual representation of these measures. Specifically, we show that one can equivalently represent any convex fairness measure as a robustified order-based fairness measure over a set of weights $\wb$ (Theorem~\ref{thm:convex_FM_dual}). In addition to providing a unified mathematical expression of convex fairness measures, we show that this dual representation provides a mechanism to investigate the equivalence of convex fairness measures from a geometric perspective (Theorem~\ref{thm:characterization_of_equivalence_CFM}). We later use this dual representation in Section~\ref{sec:opt_w_fairness_measures} to propose a unified framework for optimization models with a convex fairness measure. 

First, in Definition~\ref{def:absolute_convex_fairness_measures}, we formally introduce the class of convex fairness measures.

\begin{definition} \label{def:absolute_convex_fairness_measures}
A fairness measure $\nu:\R^N\rightarrow\R$ is a convex fairness measure if it satisfies Axioms~\ref{axiom:continuity}, \ref{axiom:normalization}, \ref{axiom:symmetry}, \ref{axiom:trans_invariance}, \ref{axiom:pos_homo}, and the following Axiom~\ref{axiom:convexity}.
\end{definition}

\setaxiomtag{CV}
\begin{axiom}[Convexity] \label{axiom:convexity}
$\phi\big(\alpha\ub^1+(1-\alpha)\ub^2\big)\leq \alpha\phi(\ub^1)+(1-\alpha)\phi(\ub^2)$ for any $\{\ub^1,\ub^2\}\subset\R^N$ and $\alpha\in(0,1)$.
\end{axiom}

Convexity is a desirable property of fairness measure for use in optimization contexts. In particular, introducing a convex fairness measure to an optimization model does not generally pose an additional optimization burden compared with a non-convex fairness measure.  In addition, as pointed out by \cite{Williamson_Menon:2019}, if a fairness measure is non-convex, its value could be decreased by partitioning the subjects or groups of interests into subgroups, which is counter to what we wish to achieve (see also \citealp{Kolm:1976} for a thorough discussion). Specifically, if the fairness measure $\phi$ satisfies Axioms~\ref{axiom:convexity} and \ref{axiom:pos_homo}, then $\phi$ is subadditive, i.e., $\phi(\ub^1+\ub^2)\leq \phi(\ub^1)+\phi(\ub^2)$. On the other hand, if $\phi$ is not convex, then $\phi$ may not be subadditive, i.e., $\phi(\ub^1+\ub^2)>\phi(\ub^1)+\phi(\ub^2)$. This implies that the fairness measure value would be smaller by dividing subjects into subgroups, which is undesirable. 

\begin{remark}
Axioms~\ref{axiom:convexity} and \ref{axiom:symmetry} imply Axiom~\ref{axiom:Schur_convex} \citep{Marshall_et_al:2011}. Thus,  we did not include  Axiom~\ref{axiom:Schur_convex} in the definition of convex fairness measures (see Definition~\ref{def:absolute_convex_fairness_measures}).
\end{remark}

Next, in Theorem~\ref{thm:convex_FM_dual}, we provide one of the key results of this paper, which shows that any convex fairness measure admits a dual representation as a worst-case order-based fairness measure over its dual set.

\begin{theorem} \label{thm:convex_FM_dual}
Consider a fairness measure $\nu:\R^N\rightarrow\R$, and let $\calS^N=\{\wb\in\Rup^N\mid \one^\tp\wb=0\}$. The following statements are equivalent: (a) $\nu$ is a convex fairness measure; (b) there exists a convex compact set $\calW_\nu\subseteq\calS^N$ with $\calW_\nu\ne\{\zero\}$ such that $\nu(\ub)=\sup_{\wb\in\calW_\nu} \nu_{\wb}(\ub)$; (c) there exists a compact set $\calW_\nu\subseteq\calS^N$ with $\calW_\nu\ne\{\zero\}$ such that $\nu(\ub)=\sup_{\wb\in\calW_\nu} \nu_{\wb}(\ub)$.
\end{theorem}

Theorem~\ref{thm:convex_FM_dual} provides a dual representation of any convex fairness measures characterized by a dual set, i.e., $\calW_\nu\subseteq\{\wb\in\Rup^N\mid \one^\tp\wb=0\}$. Specifically, it shows that any convex fairness measure can be equivalently expressed as a robustified order-based fairness measure (i.e., worst-case over the dual set $\calW_\nu$). Note that the dual set may not be unique in the sense that $\nu$ can equal $\sup_{\wb\in\calW_1}\nu_{\wb}$ and $\sup_{\wb\in\calW_2}\nu_{\wb}$ but $\calW_1\ne\calW_2$ (see Remark  \ref{rem:dual_set_any_compact_W}). Due to the existence of multiple weight vectors in $\calW_\nu$, convex fairness measures are generally non-linear in  $(u_{(1)},\dots,u_{(N)})^\tp$. This is in contrast to the order-based fairness measures, which are linear in $(u_{(1)},\dots,u_{(N)})^\tp$ as characterized by Axiom~\ref{axiom:order_based}. Finally,  Theorem \ref{thm:convex_FM_dual} implies that one can construct any convex fairness measure by defining a dual set. In particular, when preference information on $\wb$ is incomplete or ambiguous, one can construct a set of potential weight vectors (i.e., the dual set) instead of using a single (biased) weight vector. This is useful in practice where the decision-maker is concerned about fairness but cannot articulate her/his preference on $\wb$ (see \citealp{Armbruster_Delage:2015,Hu_et_al:2018} for similar discussions in preference robust optimization). 

\begin{remark}  \label{rem:zero_in_dual_set}
As shown in the proof of Theorem~\ref{thm:convex_FM_dual}, the dual set $\calW_\nu$ of any convex fairness measure $\nu$ is given by $\calW_\nu=\dom(\nu^*)\cap\Rup^N$, where $\nu^*$ is the convex conjugate of $\nu$. Notably, $\zero\in\dom(\nu^*)\cap\Rup^N$ since $\nu^*(\zero)=\sup_{\ub\in\R^N}\big\{ \ub^\tp\zero-\nu(\ub)\big\}=-\inf_{\ub\in\R^N}\nu(\ub)=0$.
\end{remark}

\begin{remark}  \label{rem:dual_set_any_compact_W}
Let $\calWt\subseteq\R^N$ be a compact set and consider the convex fairness measure $\nu(\ub)=\sup_{\wb\in\calWt}\nu_{\wb}(\ub)$. Then, $\nu(\ub)=\sup_{\wb\in\conv(\calWt)} \nu_{\wb}(\ub)$. Indeed, letting $\calP$ be the set of permutation matrices, we have
\begin{equation*} 
   \sup_{\wb\in\calWt} \nu_{\wb}(\ub)=\sup_{\wb\in\calWt} \sup_{P\in\calP} \ub^\tp(P\wb) = \sup_{P\in\calP} \sup_{\wb\in\calWt} (P^\tp\ub)^\tp \wb=\sup_{P\in\calP}\sup_{\wb\in\conv(\calWt)}  (P^\tp\ub)^\tp\wb, 
\end{equation*}
where the first equality follows from Theorem \ref{thm:order_based_FM_characterization_I} and the last equality follows from the linearity of the objective. 
\end{remark}

Note that the dual set  $\calW_\nu$ characterizing a convex fairness measure $\nu$ may have infinite weight vectors and thus may pose optimization burdens. However, in Proposition~\ref{prop:covnex_FM_dual_set_ex_pt}, we show that it suffices to consider only the non-zero extreme points of $\calW_\nu$. Such characterization could help establish theoretical convergence guarantees of solution approaches that exploit the extreme points of the dual set $\calW_\nu$ when solving optimization problems that involve a convex fairness measure. We propose such an algorithm and discuss its convergence in Section~\ref{sec:opt_w_fairness_measures}.

\begin{proposition} \label{prop:covnex_FM_dual_set_ex_pt}
Let $\nu$ be a convex fairness measure with a convex dual set $\calW_\nu$. Then, $\nu(\ub)=\sup_{\wb\in\calE(\calW_\nu)\setminus\{\zero\}} \nu_{\wb}(\ub)$, where $\calE(\calW_\nu)$ is the set of extreme points of $\calW_\nu$.
\end{proposition}

Since all the fairness measures in Table \ref{table:existing_FM} are convex fairness measures (see \ref{appdx:eg_inequity_measures}), they admit the dual representation in Theorem \ref{thm:convex_FM_dual}. Next, in Proposition \ref{prop:existing_FM_dual_set}, we derive the dual sets of these fairness measures, which enable us to investigate their equivalence from a geometric perspective (see Theorem~\ref{thm:characterization_of_equivalence_CFM} and Example~\ref{eg:dual_sets_metrics_equivalence}).

\begin{proposition} \label{prop:existing_FM_dual_set}
The dual sets of fairness measures in Table \ref{table:existing_FM} are as follows.
\begin{itemize}
    \item[(a)] For (i) and (iii), $\calW_\nu=\{(-1,0,\dots,0,1)\in\R^N\}$.
    \item[(b)] For (ii), $\calW_\nu=\{\wb'\in\R^N\mid  w'_i=2(2i-1-N),\,\forall i\in[N]\}$.
    \item[(c)] For (iv)--(vi), $\calW_\nu=\{\wb\in\R^N\mid \wb=\wb'-\wbar'\one,\, \wbar'=(1/N)\sum_{j=1}^N w'_j,\,\norms{\wb'}_q \leq 1,\, \wb'\in\Rup^N\}$, where $q=\infty$, $q=2$, and $q=1$ for (iv), (v), and (vi), respectively.
    \item[(d)] For (vii), $\calW_\nu=\{\wb\in\R^N\mid \wb=\wb'-\wbar'\one,\, \wbar'=(1/N)\sum_{j=1}^N w'_j,\,\norms{\wb'}_1 \leq N,\,\wb'\in\Rup^N\}$.
    \item[(e)] For (viii), $\calW_\nu=\{\wb\in\R^N\mid w_1 = -(N-k)-1,\, w_2=\dots=w_k=-1,\, w_{k+1}=\dots=w_{N-1}=1,\, w_N=k+1,\, k\in[N-1] \}$.
\end{itemize}
\end{proposition}

\begin{remark} \label{rem:eg_order-based}
It follows from Proposition \ref{prop:existing_FM_dual_set} that measures (i)--(iii) are order-based since the dual sets of measures (i)--(iii) have only one weight vector satisfying Definition \ref{def:order_based_FM}. Specifically, measures (i) and (iii) are order-based with $\wb=(-1,0,\dots,0,1)\in\R^N$ and metric (ii) is order-based with $\wb'$, where $w'_i=2(2i-1-N)$ for all $i\in[N]$. However, when $N>3$, dual sets of measures (iv)--(viii) consist of more than one non-zero extreme point, and hence, (iv)--(viii) do not satisfy Definition \ref{def:order_based_FM}, i.e., (iv)--(viii) are not order-based.
\end{remark}

Next, in Theorem~\ref{thm:characterization_of_equivalence_CFM}, we show that one can use the dual representation in Theorem~\ref{thm:convex_FM_dual} to verify whether two convex fairness measures are equivalent. Here, we say two fairness measures $\nu^1$ and $\nu^2$ are equivalent if there exists $\beta>0$ such that $\nu^1(\ub) = \beta\nu^2(\ub)$ for all $\ub\in\R^N$. Hence, replacing $\nu^2$ with $\nu^1$ (i.e., from $\nu^2$ to $\beta\nu^2$) in the fairness-promoting optimization problem \eqref{prob:FM_eff_min_obj} or \eqref{prob:FM_eff_min_constraint}, essentially scales the weight on the fairness criterion $\gamma$ or the upper bound $\eta$ by a factor of $\beta>0$.

\begin{theorem} \label{thm:characterization_of_equivalence_CFM}
Let $\nu_1$ and $\nu_2$ be two convex fairness measures with dual sets $\calW_1=\dom(\nu_1^*)\cap\Rup^N$ and $\calW_2=\dom(\nu_2^*)\cap\Rup^N$, respectively. Then, $\nu_1$ is equivalent to $\nu_2$ if and only if $\calW_1=\beta\calW_2$ for some $\beta>0$.
\end{theorem}

Theorem~\ref{thm:characterization_of_equivalence_CFM} provides a mechanism to investigate the equivalence of convex fairness measures from a geometric perspective via their dual sets. For example, let us first consider the case when there are two individuals or groups of interests ($N=2$). By Remark \ref{rem:zero_in_dual_set}, the dual set $\calW_\nu=\dom(\nu^*)\cap\Rup^2$ of any two-dimensional convex fairness measure $\nu$  on $\R^2$ is a line segment joining $\zero$ and a point on $\calS^2:=\{\wb\in\Rup^2\mid \one^\tp\wb=0\}$ as shown in Figure \ref{fig:dual_set_domain_2D}. Hence, the dual sets of any convex fairness measures when $N=2$ are proportional. It follows from Theorem~\ref{thm:characterization_of_equivalence_CFM} that all two-dimensional convex fairness measures are equivalent. In particular, it is easy to verify that any two-dimensional convex fairness measure equals $c|u_1-u_2|$ for some $c>0$. On the other hand, when $N>2$, convex fairness measures may not be equivalent (i.e., they yield different solutions with varying impacts on fairness). To see this, let us consider the case when $N=3$. In this case, as shown in Figure \ref{fig:dual_set_domain_3D}, the set $\calS^3$ is a surface. Since dual sets $\calW_{1}\subseteq \calS^3$ and $\calW_{2}\subseteq \calS^3$ of two convex fairness measures $\nu_1$ and $\nu_2$ may not be proportional, the two measures may not be equivalent. We illustrate this in Example~\ref{eg:dual_sets_metrics_equivalence}.

\begin{figure}[t]
     \centering
     \begin{subfigure}[b]{0.47\textwidth}
         \centering
         \includegraphics[scale=0.7]{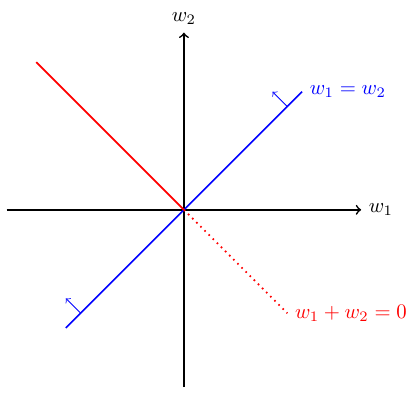}
         \caption{Two dimensional (solid red line)}
         \label{fig:dual_set_domain_2D}
     \end{subfigure}
     \hfill
     \begin{subfigure}[b]{0.47\textwidth}
         \centering
         \includegraphics[scale=0.75]{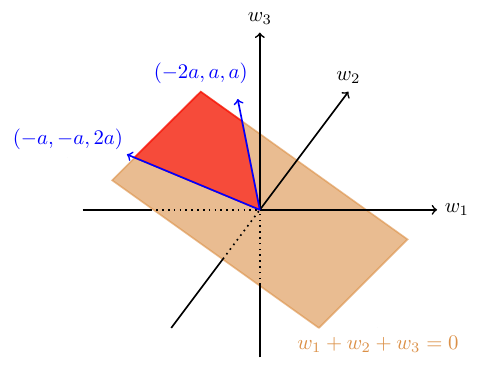}
         \caption{Three dimensional (red surface)}
         \label{fig:dual_set_domain_3D}
     \end{subfigure}
        \caption{Illustration of the set $\calS^N=\{\wb\in\Rup^N\mid \one^\tp\wb=0\}$ when (a) $N=2$ and (b) $N=3$ }
        \label{fig:dual_set_domain}
\end{figure}

\begin{example}[Dual sets in $\R^3$] \label{eg:dual_sets_metrics_equivalence}
Figure \ref{fig:dual_set_example} shows the dual sets for (iv)--(vi) when $N=3$. Note that the dual set of (iv) is proportional to that of (vi). Thus, by Theorem \ref{thm:characterization_of_equivalence_CFM}, (iv) and (vi) are equivalent. However, dual set of (v) has a curved boundary, implying that it is not proportional to that of (iv) and (vi). Thus, (v) is not equivalent to (iv) and (vi). These results are consistent with our algebraic proof of equivalence in \ref{appdx:eg_inequity_measures}.
\begin{figure}[t]
     \centering
     \begin{subfigure}[b]{0.32\textwidth}
         \centering
         \includegraphics[scale=0.45]{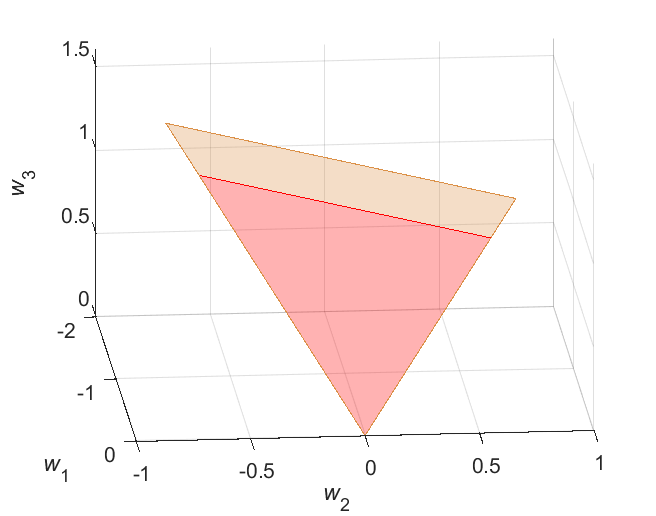}
         \caption{Measure (iv)}
         \label{fig:dual_set_iv}
     \end{subfigure}
     \hfill
     \begin{subfigure}[b]{0.32\textwidth}
         \centering
         \includegraphics[scale=0.45]{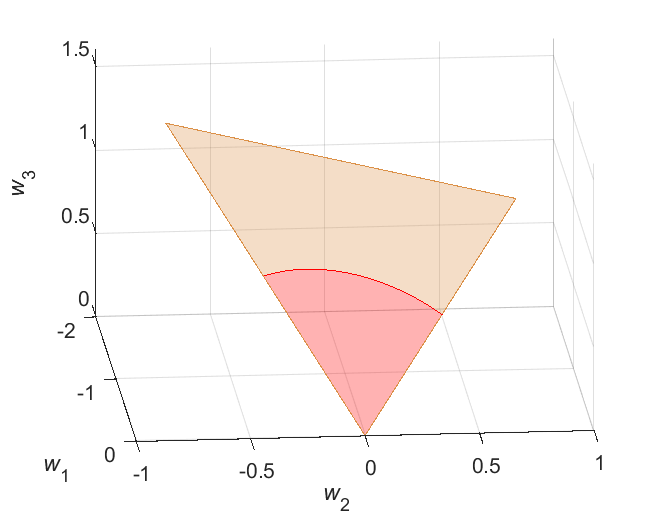}
         \caption{Measure (v)}
         \label{fig:dual_set_v}
     \end{subfigure}
     \hfill
     \begin{subfigure}[b]{0.32\textwidth}
         \centering
         \includegraphics[scale=0.45]{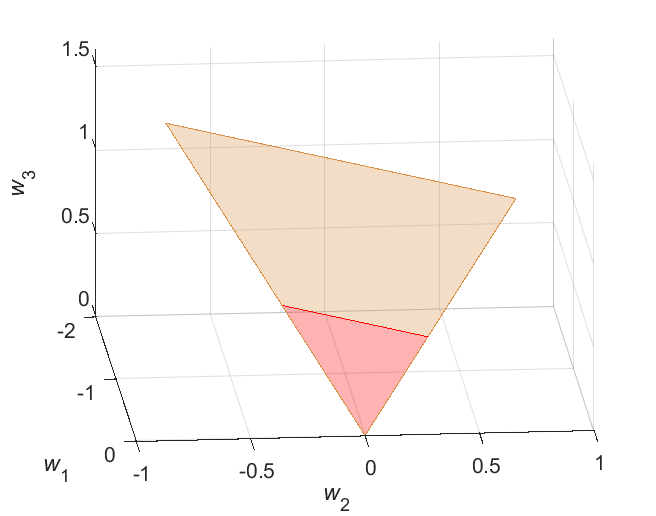}
         \caption{Measure (vi)}
         \label{fig:dual_set_vi}
     \end{subfigure}
        \caption{Dual sets $\calW_\nu$ (in red) of fairness measures (iv)--(vi) as subsets of $\calS^3$ (in brown)}
        \label{fig:dual_set_example}
\end{figure}
\end{example}

\section{Optimization with Convex Fairness Measures}  \label{sec:opt_w_fairness_measures}

In this section, we propose solution approaches to solve optimization problems with our proposed convex fairness measures. Specifically, we focus on fairness-promoting optimization models of the form \eqref{prob:FM_eff_min_obj}, where the fairness measure is incorporated into the objective function. In \ref{appdx:sol_approach_CFM_constraint}, we show how one can leverage the proposed approaches to solve fairness-promoting problems of the form \eqref{prob:FM_eff_min_constraint} using convex fairness measures and their relative counterparts.

Our proposed solution approaches can be used to solve problem \eqref{prob:FM_eff_min_obj} with any inefficiency measure $f$ (see Section~\ref{sec:num_experiments}). However, for notational convenience  and ease of exposition, we focus on the optimization problem
\begin{equation}\label{prob:FM_min}
    \min_{\xb,\,\ub} \Big\{ \nu(\ub) \,\Big|\, \ub = U(\xb),\, \xb\in\calX\Big\}.
\end{equation}
In Section~\ref{subsec:opt_w_order_based}, we derive unified reformulations of \eqref{prob:FM_min} when $\nu$ is an order-based fairness measure. Then, in Section~\ref{subsec:opt_w_CFM}, we propose a decomposition algorithm to solve \eqref{prob:FM_min} when $\nu$ is a convex fairness measure. In Section~\ref{subsec:expt_facility_location}, we provide numerical examples demonstrating the computational efficiency of our proposed approach over existing ones.

\subsection{Minimizing Order-based Fairness Measures} \label{subsec:opt_w_order_based}

We first consider problem \eqref{prob:FM_min} with an order-based fairness measure objective $\nu_{\wb}$ for some $\wb\in\calW$ (see Section \ref{sec:order_based_fairness_measures}). First, in Theorem~\ref{thm:FM_min_reformulation_order_based}, we use the characterization of $\nu_{\wb}$ in Theorem~\ref{thm:order_based_FM_characterization_I} to derive an equivalent reformulation of problem \eqref{prob:FM_min} with $\nu=\nu_{\wb}$.

\begin{theorem} \label{thm:FM_min_reformulation_order_based}
Problem \eqref{prob:FM_min} with an order-based fairness measure $\nu_{\wb}$ is equivalent to
\allowdisplaybreaks
\begin{subequations}  \label{eqn:FM_min_reformulation_order_based}
\begin{align}
\underset{\xb,\,\ub,\,\lambdab\in\R^N,\, \thetab\in\R^N}{\textup{minimize}\,} \quad
&  \one^\tp(\lambdab+\thetab)  \\   
\textup{subject to} \hspace{9mm}
&  \lambda_i + \theta_j \geq u_i w_j, \quad\forall i\in[N],\, j\in[N],  \label{eqn:FM_min_reformulation_direct_con1}\\
& \ub = U(\xb),\, \xb\in\calX.
\end{align}%
\end{subequations} 
\end{theorem}

We make two remarks in order. First, formulation  \eqref{eqn:FM_min_reformulation_order_based} is linear if $U$ is linear and $\calX$ is a set of linear constraints on $\xb$. Second, variables $\lambdab$ and $\thetab$ in \eqref{eqn:FM_min_reformulation_order_based} are unrestricted in sign.  In Proposition~\ref{prop:dual_variables_LB_UB}, we derive lower and upper bounds on $\lambdab$ and $\thetab$, which could reduce the search space of these variables. 

\begin{proposition} \label{prop:dual_variables_LB_UB}
Let $\Umax=\sup \limits_{\xb\in\calX} \max\limits_{i\in [N]} [U(\xb)]_i$ and $\Umin=\inf \limits_{\xb\in\calX} \min \limits_{i\in [N]} [U(\xb)]_i$, where $[U(\xb)]_i$ is the $i$th entry of $U(\xb)$. Without loss of optimality, we can impose the following bounds on variables $\lambdab$ and $\thetab$ in \eqref{eqn:FM_min_reformulation_order_based}. For all $i\in[N]$,
\begin{align}
    0 &\leq \lambda_i \leq (\Umax-\Umin) \norms{\wb}_\infty =: \lambdabar , \label{eqn:lambda_LBUB}\\
   \min\big\{\Umax w_i,\, \Umin w_i\big\} - \lambdabar &\leq \theta_i \leq \max\big\{\Umax w_i,\, \Umin w_i\big\}. \label{eqn:theta_LBUB}
\end{align}
\end{proposition}

We close this section by noting that formulation \eqref{eqn:FM_min_reformulation_order_based} provides a unified mechanism for analyzing and solving fairness-promoting optimization problems under different order-based fairness measures. In particular, instead of deriving a customized formulation for each order-based fairness measure, one can simply use formulation \eqref{eqn:FM_min_reformulation_order_based} and replace the weight vector $\wb$ on the right-hand side of constraint \eqref{eqn:FM_min_reformulation_direct_con1}  with the one corresponding to each measure.

\subsection{Minimizing Convex Fairness Measures}  \label{subsec:opt_w_CFM}

Let us now consider problem \eqref{prob:FM_min} with a convex fairness measure objective (see Section \ref{sec:convex_fairness_measures}). By the dual representation in Theorem \ref{thm:convex_FM_dual}, we can reformulate problem \eqref{prob:FM_min} as
\allowdisplaybreaks
\begin{align} 
    &\quad\,\,\min_{\xb,\,\ub} \bigg\{ \sup_{\wb\in\calW_\nu}\nu_{\wb}(\ub)\,\bigg|\, \ub = U(\xb),\, \xb\in\calX\bigg\} \nonumber \\
    &= \min_{\xb,\,\ub,\,\delta} \bigg\{ \delta \,\bigg|\, \delta\geq\textstyle\sum \limits_{i=1}^N w_i u_{(i)},\,\forall \wb\in\calW_\nu,\, \ub = U(\xb),\, \xb\in\calX\bigg\}. \label{prob:FM_min_convex_FM}
\end{align}
Problem \eqref{prob:FM_min_convex_FM} is challenging to solve because of the following two reasons. First, note that the order-based fairness measure $\nu_{\wb}(\ub)=\sum_{i=1}^N w_iu_{(i)}$ is a function of the ordered entries of $\ub$. Thus, the constraint $\delta\geq \sum_{i=1}^N w_iu_{(i)}$ is non-linear in $\ub$. Second, $\calW_\nu$ may not be finite, i.e., \eqref{prob:FM_min_convex_FM} is a semi-infinite program. However, from Proposition \ref{prop:covnex_FM_dual_set_ex_pt}, we know that it suffices to consider the non-zero extreme points of the dual set $\calW_\nu$ to solve problem~\eqref{prob:FM_min_convex_FM}. Leveraging this fact, we next propose a column-and-constraint generation algorithm (C\&CG) to solve problem~\eqref{prob:FM_min_convex_FM}. Algorithm~\ref{algo:decomposition_method} summarizes the steps of our C\&CG.   
\IncMargin{0em}
\begin{algorithm}[t!] 
\SetKwInOut{Initialization}{Initialization}
\Initialization{Set $LB=0$, $UB=\infty$, $\varepsilon>0$, $j=1$, $\wb^0\in\calW_\nu$.}
\textbf{1. Master problem.} \\
\hspace{5.5mm}Solve master problem \eqref{prob:C&CG_Master} with weights $\{\wb^0,\dots,\wb^{j-1}\}$.\\
\hspace{5.5mm}Record the optimal solution $(\xb^j,\ub^j)$ and optimal value $\delta^j$.\\
\hspace{5.5mm}Set $LB \leftarrow \max\{LB,\delta^j\}$. \\
\textbf{2. Subproblem.} Solve subproblem \eqref{prob:C&CG_Subproblem} for fixed $\ub=\ub^j$.\\
\hspace{5.5mm}Record the optimal solution $\wb^j$ and value $D^j$. Set $UB \leftarrow  \min\big\{UB,\, D^j\big\}$.\\
\hspace{5.5mm}If $(UB-LB)/UB<\varepsilon$, terminate and return the solution with best objective. \\
\textbf{3. Scenario set enlargement.}\\
\hspace{5.5mm}Update $j \leftarrow j+1$ and go back to step 1.
\BlankLine
\caption{The column-and-constraint generation algorithm (C\&CG)} \label{algo:decomposition_method}
\end{algorithm}\DecMargin{1em}

In C\&CG, we solve a master problem and subproblem at each iteration. Specifically,  at iteration $j$ of C\&CG, we aim to solve the following master problem
\begin{equation} \label{prob:C&CG_Master_original}
    \min_{\xb,\,\ub,\,\delta} \bigg\{ \delta \,\bigg|\, \delta\geq\textstyle\sum \limits_{i=1}^N w_i u_{(i)},\,\forall \wb\in\{\wb^0,\dots,\wb^{j-1}\},\, \ub = U(\xb),\, \xb\in\calX\bigg\},
\end{equation}
where $\{\wb^0,\dots,\wb^{j-1}\}\subseteq\calW_\nu$. Note again that \eqref{prob:C&CG_Master_original} is not directly solvable in the presented form due to the non-linearity of constraint $\delta\geq \sum_{i=1}^N w_iu_{(i)}$. However, in Proposition \ref{prop:FM_min_convex_FM_CCG_master}, we provide an equivalent solvable reformulation of  \eqref{prob:C&CG_Master_original}.

\begin{proposition} \label{prop:FM_min_convex_FM_CCG_master}
The master problem \eqref{prob:C&CG_Master_original} in C\&CG is equivalent to 
\allowdisplaybreaks
\begin{subequations}
\begin{align}
\underset{\xb,\,\ub,\,\lambdab,\,\thetab,\,\delta}{\textup{minimize}} \quad
&  \delta \\
\textup{subject to} \quad
& \delta \geq  \one^\tp(\lambdab^k+\thetab^k),\quad\forall k\in[0,j-1]_\Z, \label{eqn:C&CG_Master_con1} \\
& \lambda^k_i + \theta^k_{i'} \geq u_i w_{i'}^k,\quad\forall i\in[N],\,i'\in[N],\,k\in[0,j-1]_\Z, \label{eqn:C&CG_Master_con2}\\
& \ub=U(\xb),\,\xb\in\calX.
\end{align}  \label{prob:C&CG_Master}%
\end{subequations} \vspace{-12mm}
\end{proposition}

Since only a set of weight vectors in $\calW_\nu$ is considered, the master problem is a relaxation of the original problem \eqref{prob:FM_min_convex_FM}, and thus its optimal value $\delta^j$ provides a lower bound to \eqref{prob:FM_min_convex_FM}. With the optimal solution $\ub^j$ from the master problem, we solve the following subproblem
\begin{equation} \label{prob:C&CG_Subproblem}
    D^j = \max\bigg\{ \textstyle \sum \limits_{i=1}^N w_i u^j_{(i)}\,\bigg|\, \wb\in\calW_\nu \bigg\},
\end{equation}
and record the optimal solution $\wb^j\in\calW_\nu$ of the subproblem. Since $\ub^j$ obtained from the master problem is feasible, $D^j$ provides an upper bound to \eqref{prob:FM_min_convex_FM}. Note that subproblem \eqref{prob:C&CG_Subproblem} is always feasible by construction because $\calW_\nu$ is non-empty. Since $\calW_\nu$ is convex (see Remark \ref{rem:dual_set_any_compact_W}) and the objective $\sum_{i=1}^N w_i u_{(i)}$ is linear, subproblem \eqref{prob:C&CG_Subproblem} can be efficiently solved using convex optimization algorithms.  In particular, if $\calW_\nu$ is a polytope, subproblem \eqref{prob:C&CG_Subproblem} reduces to a linear program. Finally, if the gap between the lower and upper bounds is smaller than a pre-specified tolerance $\varepsilon$, C\&CG terminates and returns the solution with the best objective value $UB$. Otherwise, we proceed to the next iteration and solve the master problem with an enlarged subset of weights $\{\wb^0,\dots,\wb^j\}$. Note that one can set the initial weight $\wb^0$ as the zero vector $\zero$ (see Remark \ref{rem:zero_in_dual_set}). 

Next, we discuss the convergence of the proposed C\&CG. Note that $\calW_\nu$ is compact and convex. If, in addition, $\calX$ is compact (a mild assumption that holds valid in a wide range of applications, including facility location and scheduling problems; see \citealp{Ahmadi-Javid_et_al:2017, Celik-Turkoglu_Erol-Genevois:2020, Marynissen_Demeulemeester:2019}), Proposition 2 of \cite{Bertsimas_Shtern:2018} ensures that any accumulation point of the sequence $\{\xb^j\}_{j\in\mathbb{N}}$ generated from C\&CG is an optimal solution to \eqref{prob:FM_min_convex_FM}. Moreover, if $\calW_\nu$ is a polyhedron and a vertex is always returned when solving the subproblem, C\&CG terminates in a finite number of iterations by Proposition~\ref{prop:covnex_FM_dual_set_ex_pt}. 

We close this section by noting that our unified reformulation \eqref{prob:FM_min_convex_FM} and C\&CG provide a mechanism for analyzing and solving fairness-promoting optimization problems under different convex fairness measures. In particular, instead of deriving customized reformulations and solution methods for each convex fairness measure one would like to incorporate in the optimization model, one could use \eqref{prob:FM_min_convex_FM} and C\&CG to solve the reformulation with different dual set $\calW_\nu$ corresponding to each convex fairness measure.

\section{Stability Analysis} \label{sec:stability_analysis}

Given that there are many different convex fairness measures, it is crucial to quantify how the choice of convex fairness measure in the objective of an optimization problem would affect the optimal value and solution. In this section, we leverage the dual representation of convex fairness measures to investigate the stability of the optimal value and solution of the fairness-promoting optimization problem of the form \eqref{prob:FM_eff_min_obj}, i.e., $\min \big\{ f(\ub) + \gamma \phi(\ub) \,\big|\, \ub = U(\xb),\, \xb\in\calX\big\}$, with respect to the choice of different convex fairness measures in the objective. Note that replacing the fairness measure in the objective of problem \eqref{prob:FM_eff_min_obj} with another does not impact the feasibility of the problem. However, the optimal solution may differ if the two measures are not equivalent. In contrast, replacing the fairness measure in the constraints as in problem \eqref{prob:FM_eff_min_constraint} (i.e., $\phi(\ub)\leq \eta$) with another one may impact the feasibility of the problem if the two measures are not equivalent due to the upper bound $\eta$ on the value of the fairness measure.

Let us first define the following additional notation used in the analysis. We define the distance between a point $\wb\in\R^N$ and a set $\calW$ as $d(\wb,\calW):=\inf_{\wbt\in\calW}\norms{\wb-\wbt}_2$, and the Hausdorff distance between two sets $\calW_1$ and $\calW_2$ as
\begin{equation} \label{eqn:Hausdorff}
    d_H(\calW_1,\calW_2):=\max\bigg\{ \sup_{\wb_2\in\calW_2}d(\wb_2,\calW_1),\, \sup_{\wb_1\in\calW_1}d(\wb_1,\calW_2)\bigg\}
\end{equation}
First, in Lemma~\ref{lemma:CFM_dual_set_Hausdorff}, we show that the difference between two convex fairness measures is bounded by the Hausdorff distance between their dual sets.

\begin{lemma} \label{lemma:CFM_dual_set_Hausdorff}
Let $\nu_1$ and $\nu_2$ be two convex fairness measures with dual sets $\calW_1$ and $\calW_2$, respectively. For any $\ub\in\R^N$, we have $\big| \nu_1(\ub)-\nu_2(\ub) \big| \leq d_H(\calW_1,\calW_2) \cdot \norms{\ub}_2$.
\end{lemma}

Note that if $\calW_1=\calW_2$, then $d_H(\calW_1,\calW_2)=0$, implying that $\nu_1\equiv\nu_2$ from Lemma~\ref{lemma:CFM_dual_set_Hausdorff}. Next, in Theorem \ref{thm:stability_analysis}, we use the results in Lemma \ref{lemma:CFM_dual_set_Hausdorff} to show that the differences in optimal value and solution of \eqref{prob:FM_eff_min_obj} under two different convex fairness measures are bounded by the Hausdorff distance between their dual sets.

\begin{theorem} \label{thm:stability_analysis}
Let $\nu_1$ and $\nu_2$ be two convex fairness measures with dual sets $\calW_1$ and $\calW_2$, respectively. Let $\upsilon_k^\star$ and $\xb^\star_k$ be the optimal value and an optimal solution of \eqref{prob:FM_eff_min_obj} with $\nu=\nu_k$ for $k\in\{1,2\}$. Assume that $\calX$ is compact. Then, the following statements hold.
\begin{enumerate}
    \item [(a)] We have $\big|\upsilon^\star_1-\upsilon^\star_2\big| \leq \gamma\Umax d_H(\calW_1,\calW_2)$, where $\Umax=\max_{\xb\in\calX} \norms{U(\xb)}_2<\infty$.
    \item [(b)] Let $\wb^\star_k\in\argmax_{\wb\in\calW_k} \nu_{\wb}(U(\xb^\star_k))$ for $k\in\{1,2\}$. Assume that the following quadratic growth condition hold for $k\in\{1,2\}$: there exists $\tau_k>0$ such that
    \begin{equation}  \label{eqn:local_quad_growth_cond}
       f(U(\xb))+\gamma\nu_{\wb^\star_k}(U(\xb))-f(U(\xb^\star_k))-\gamma\nu_{\wb^\star_k}(U(\xb^\star_k)) \geq \tau_k\norms{\xb-\xb^\star_k}_2^2
    \end{equation}
    for all $\xb\in\calX$. Then, we have 
    \begin{equation}  \label{eqn:statbility_bounds_on_optimal_solution}
      \norms{\xb^\star_1-\xb^\star_2}_2\leq \sqrt{\frac{\gamma\Umax}{\min\{\tau_1,\tau_2\}}} \cdot d^{1/2}_H(\calW_1,\calW_2).  
    \end{equation}
\end{enumerate}
\end{theorem}

\begin{remark}  \label{rem:quad_growth_cond}
The quadratic growth condition in Theorem \ref{thm:stability_analysis} is a standard assumption in stability analysis of stochastic programs and distributionally robust optimization problems (see, e.g., \citealp{Liu_et_al:2019, Pichler_Xu:2018, Shapiro:1994}). For example, if $f$ is strongly convex and $U$ is affine of the form $U(\xb)=A\xb+\bb$ with $A$ being a matrix of rank $N$, then $f\circ U + \gamma\nu\circ U$ is also strongly convex in $\xb$ (where $\circ$ denotes the composition of two functions). Thus, in this case, the quadratic growth condition is satisfied (see, e.g., discussions in \citealp{Chang_et_al:2018}). 
\end{remark}

Theorem~\ref{thm:stability_analysis} provides a mechanism for quantifying the influence of the choice of convex fairness measure on the optimal value and solution of problem \eqref{prob:FM_eff_min_obj}. Specifically, it shows that the absolute difference between the optimal value and solution of \eqref{prob:FM_eff_min_obj} under two different convex fairness measures, $\nu_1$ and $\nu_2$, is upper bounded by the Hausdorff distance $d_H(\calW_1,\calW_2)$ between their dual sets $\calW_1$ and $\calW_2$.  Hence, the optimal value and solution would not deviate significantly if the two convex fairness measures are close enough (i.e., the Hausdorff distance between their dual sets is small).

\section{Numerical Experiments}   \label{sec:num_experiments}

In this section, we present numerical results demonstrating the computational efficiency of our proposed unified reformulations and solution methods for optimization problems with convex fairness measures over traditional ones using a fair facility location problem (Section~\ref{subsec:expt_facility_location}). In addition, we illustrate our stability results using a fair resource allocation problem (Section~\ref{subsec:expt_resource_allocation}). We implement all formulations and solution algorithms in the AMPL modeling language and use CPLEX (version 20.1.0.0) as the solver with default settings. We conduct all experiments on a computer with an Intel Xeon Silver processor, a 2.10 GHz CPU, and a 128 GB memory.

\subsection{Fair Facility Location} \label{subsec:expt_facility_location}

In this section, we illustrate the potential computational advantage of our proposed unified framework for optimization problems with a convex fairness measure in the context of a fair $p$-median facility location problem. Specifically, given sets of customer locations $I$ and potential facility locations $J$, this problem consists of determining where to open $p<|J|$ facilities to minimize the convex combination of the total transportation cost (i.e., an efficiency measure) and a measure of unfairness in the transportation cost across customers. The demand at $i\in I$ is $d_i$ and the transportation cost, per unit demand, between $i\in I$ and $j\in J$ is $c_{ij}$. Let $x_j$ be a binary variable taking value $1$ if facility $j$ is open, and $y_{ij}$ be a binary variable taking value $1$ if customer $i$ is assigned to facility $j$, and are $0$ otherwise.  Finally, we define a non-negative variable $r_i$ as the transportation cost of $i\in I$. Using this notation, we formulate the problem as
\allowdisplaybreaks
\begin{subequations}  \label{prob:FLP}
\begin{align}
\underset{\xb,\,\yb,\,\rb}{\textup{minimize}\,} \quad
&  \gamma \sum_{i\in I} r_i + (1-\gamma) \nu(\rb) \label{eqn:FLP_obj}  \\   
\textup{subject to} \hspace{4mm}
&  r_i = \sum_{j\in J} d_i c_{ij} y_{ij},\quad\forall i\in I, \label{eqn:FLP_con1} \\ 
& \sum_{j\in J} x_j = p,\quad \sum_{j\in J} y_{ij} = 1,\quad y_{ij} \leq x_j,\quad\forall i\in I,\, j\in J, \label{eqn:FLP_con2} \\
& \xb\in\{0,1\}^J,\,\yb\in\{0,1\}^{I\times J}, \label{eqn:FLP_con3}
\end{align}%
\end{subequations} 
where constraints \eqref{eqn:FLP_con2} ensure that only $p$ facilities are open, and each customer is assigned to exactly one open facility. The parameter $\gamma$ in \eqref{eqn:FLP_obj} controls the trade-off between efficiency (the first term) and fairness (the second term), where a smaller $\gamma$ reflects a higher emphasis on fairness. For illustrative purposes, we consider two convex fairness measures in the objective: the Gini deviation $\nu(\rb)=|I|^{-1}\sum_{i\in I}\sum_{i'\in I}|r_i-r_{i'}|$ and the absolute deviation from mean $\nu(\rb)=\sum_{i\in I} |r_i-\bar{r}|$, where $\bar{r}=|I|^{-1}\sum_{i\in I} r_i$. 

Next, we compare our proposed unified reformulations of these measures with the traditional ones.  For the absolute deviation from mean (MAD), existing literature \citep{Chen_Hooker:2023, Shehadeh_Snyder:2021} introduce auxiliary variables $z_{i}=|r_i-\bar{r}|$ and derive the following equivalent linear reformulation:
\begin{equation} \label{eqn:FLP_direct_reformulation_MAD}
    \min_{\xb,\,\yb,\,\rb,\,\zb} \Bigg\{\gamma\sum_{i\in I} r_i + (1-\gamma) \sum_{i=1}^N  z_i \,\Bigg|\, z_i \geq r_i - \frac{1}{N}\sum_{i'=1}^N r_{i'},\, z_i \geq \frac{1}{N}\sum_{i'=1}^N r_{i'} - r_i,\forall i\in[N],\,\text{\eqref{eqn:FLP_con1}--\eqref{eqn:FLP_con3}}\Bigg\}.
\end{equation}
Recall from Proposition~\ref{prop:existing_FM_dual_set} that the MAD is a convex fairness measure with a dual set $\calW_\nu=\{\wb\in\R^N\mid \wb=\wb'-\wbar'\one,\, \wbar'=(1/N)\sum_{j=1}^N w'_j,\,\norms{\wb'}_\infty \leq 1,\, \wb'\in\Rup^N\}$. Thus, we apply the techniques in Section~\ref{subsec:opt_w_CFM} to reformulate problem \eqref{prob:FLP} with the MAD as
\begin{equation} \label{eqn:FLP_dual_reformulation_MAD}
    \min_{\xb,\,\yb,\,\rb,\,\lambdab,\,\thetab,\,\delta} \Bigg\{ \gamma \sum_{i\in I} r_i + (1-\gamma) \delta \,\Bigg|\, \delta \geq \nu_{\wb}(\rb),\, \forall \wb\in\calW_\nu,\, \text{\eqref{eqn:FLP_con1}--\eqref{eqn:FLP_con3}}\Bigg\}
\end{equation}
and apply our C\&CG method to solve \eqref{eqn:FLP_dual_reformulation_MAD}. 

For the Gini deviation, existing literature \citep{Chen_Hooker:2023, Lejeune_Turner:2019, Shehadeh_Snyder:2021} introduce auxiliary variables $z_{i,i'}=|r_i-r_{i'}|$ and derive the following equivalent linear reformulation:
\begin{equation} \label{eqn:FLP_direct_reformulation_Gini_deviation}
    \min_{\xb,\,\yb,\,\rb,\,\zb} \Bigg\{ \gamma \sum_{i\in I} r_i + \frac{(1-\gamma)}{|I|}\sum_{i=1}^N \sum_{i'=1}^N z_{i,i'} \,\Bigg|\, z_{i,i'} \geq r_i-r_{i'},\, z_{i,i'} \geq r_{i'}-r_i,\, \forall \{i,i'\}\subseteq[N],\, \text{\eqref{eqn:FLP_con1}--\eqref{eqn:FLP_con3}}\Bigg\}.
\end{equation}
Recall from Section~\ref{sec:order_based_fairness_measures} that the Gini deviation is an order-based fairness measure. Thus, we can apply the results of Theorem~\ref{thm:FM_min_reformulation_order_based} to derive the following equivalent reformulation of problem \eqref{prob:FLP} with the Gini deviation:
\begin{equation} \label{eqn:FLP_dual_reformulation_Gini_deviation}
    \min_{\xb,\,\yb,\,\rb,\,\lambdab,\,\thetab} \Bigg\{ \gamma \sum_{i\in I} r_i + \frac{(1-\gamma)}{|I|} \one^\tp(\lambdab+\thetab) \,\Bigg|\, \lambda_i + \theta_{i'} \geq r_i w_{i'},\, \forall \{i,i'\}\subseteq[N],\, \text{\eqref{eqn:FLP_con1}--\eqref{eqn:FLP_con3}}\Bigg\},
\end{equation}
where $w_{i'}=2(2i'-|I|-1)$ for $i'\in[N]$. Note that in addition to variables $(\xb,\yb,\rb)$ and the constraints in \eqref{prob:FLP}, the traditional reformulation \eqref{eqn:FLP_direct_reformulation_Gini_deviation} requires $N^2$ variables $\zb$ and $N^2$ constraints on $\zb$, while our reformulation \eqref{eqn:FLP_dual_reformulation_Gini_deviation} involves only $2N$ variables $(\lambdab,\thetab)$ and $N^2$ constraints on $(\lambdab,\thetab)$. Thus, our reformulation has a significantly smaller number of variables, especially when $N$ is large. Note that the traditional reformulations and our unified reformulations of the fairness-promoting problem~\eqref{prob:FLP} under each convex fairness measure are equivalent. This means that optimal solutions to \eqref{eqn:FLP_direct_reformulation_MAD} and \eqref{eqn:FLP_dual_reformulation_MAD}, and similarly, optimal solutions to \eqref{eqn:FLP_direct_reformulation_Gini_deviation} and \eqref{eqn:FLP_dual_reformulation_Gini_deviation}, have the same objective value.

Next, we investigate the computational performance of the traditional approaches and ours. For illustrative purposes, we use data from \cite{Daskin:2013} to generate five random instances for each combination of $|I|\in\{40,50\}$ and $p\in\{|I|/3,|I|/4,|I|/5\}$, rounded to the nearest integer. To generate the instances, for a given $|I|$, we randomly select $|I|$ different locations in the ``2010 County Sorted 250'' data from \cite{Daskin:2013}, where we set $I=J$ as these locations. We compute the transportation cost using the Euclidean distance based on the location's latitude and longitude, and we set the demand as the location's population. Finally, we consider three different values of $\gamma\in\{0.4,0.3,0.2\}$. We do not observe significant differences between our approach and traditional ones with larger $\gamma$ (i.e., when more emphasis is placed on efficiency). In the experiments, we set the relative tolerance to $2\%$ and imposed a two-hour time limit.   Also, we include the lower and upper bounds \eqref{eqn:lambda_LBUB}--\eqref{eqn:theta_LBUB} on the dual variables $\lambdab$ and $\thetab$ in our reformulations \eqref{eqn:FLP_dual_reformulation_MAD} and \eqref{eqn:FLP_dual_reformulation_Gini_deviation}.

Tables \ref{table:FLP_comp_time_MAD}--\ref{table:FLP_comp_time_Gini_deviation} show the minimum (min), average (avg), and maximum (max) solution time (in seconds) for five generated instances solved using our formulations and traditional ones for the problem with the MAD and Gini deviation, respectively. We observe the following from these results. First, it is clear that using our unified reformulation and the proposed  C\&CG, we can solve all the generated instances significantly faster than the classical reformulation techniques.  Second, solution time increases as $\gamma$ decreases, i.e., when more emphasis is placed on fairness. This suggests that the problem becomes more challenging when seeking fairer decisions. Notably, solution times of traditional approaches are substantially longer than our solution times when $\gamma$ is smaller. Consider, for example, the instances with $|I|=50$ and $\gamma=0.2$. We can solve all the instances using our unified reformulations and C\&CG method within one hour. Specifically, the average solution using our approach ranges from from $19$ minutes to $45$ minutes for the MAD and $10$ minutes to $28$ minutes for the Gini deviation. In contrast, we are unable to solve the majority of these generated instances within two hours using traditional reformulations \eqref{eqn:FLP_direct_reformulation_MAD} and \eqref{eqn:FLP_direct_reformulation_Gini_deviation}. On average, the solution time using \eqref{eqn:FLP_direct_reformulation_MAD} is four times longer than solving our reformulation \eqref{eqn:FLP_dual_reformulation_MAD} with the C\&CG method for the MAD. Similarly, the solution time using \eqref{eqn:FLP_direct_reformulation_Gini_deviation} is three times longer than our reformulation \eqref{eqn:FLP_dual_reformulation_Gini_deviation} of the Gini deviation. For instances that were not solved by traditional formulations within two hours, the relative optimality gap reported by the solver at termination ranges from $4.1\%$ to $25.2 \%$. Third, the gap in solution time between the two approaches widens as $p$ decreases, i.e., as the problem becomes more constrained. 

These results demonstrate where significant improvements in performance can be gained with our proposed framework. In particular, it shows the computational advantages of our unified reformulation and the proposed C\&CG method over traditional ones when achieving fairness is of a higher priority in the decision-making process. In \ref{appdx:add_comp_results}, we investigate the computational performance of traditional approaches and ours on a set of larger instances. The results confirm these findings, further emphasizing the superior computational efficiency of our approach compared with traditional methods.

\begin{table}[t] 
\scriptsize
\center
\renewcommand{\arraystretch}{0.8}
{\caption{Solution time (in seconds) over five randomly generated instances with $\gamma\in\{0.4,0.3,0.2\}$ using the absolute deviation from mean (MAD). \textit{Note:} Solution times with `$>$' indicate that one or more instances cannot be solved within the imposed two-hour time limit. } \label{table:FLP_comp_time_MAD} }
\begin{tabular}{clrrrrrrrrrrrrrr}
\Xhline{1.0pt}
$|I|=40$    &                           &  & \multicolumn{3}{c}{$\gamma = 0.4$} &  & \multicolumn{3}{c}{$\gamma = 0.3$} &  & \multicolumn{3}{c}{$\gamma = 0.2$}                                                     \\ \cline{4-14}
            &                           &  & min        & avg       & max       &  & min        & avg       & max       &  & min  & avg                                    & max                                    \\ \cline{4-6} \cline{8-10} \cline{12-14}
$p =    8$  & C\&CG Method              &  & 1          & 3         & 5         &  & 6          & 22        & 45        &  & 74   & 600                                    & 1506                                   \\
            & Traditional Reformulation \eqref{eqn:FLP_direct_reformulation_MAD}  &  & 2          & 5         & 12        &  & 14         & 99        & 228       &  & 72   & \multicolumn{1}{r}{\textgreater{}3607} & \multicolumn{1}{r}{\textgreater{}7200} \\ \Xhline{0.5pt}
            &                           &  & min        & avg       & max       &  & min        & avg       & max       &  & min  & avg                                    & max                                    \\ \cline{4-6} \cline{8-10} \cline{12-14}
$p =    10$ & C\&CG Method              &  & 3          & 4         & 5         &  & 6          & 16        & 32        &  & 20   & 167                                    & 595                                    \\
            & Traditional Reformulation \eqref{eqn:FLP_direct_reformulation_MAD}  &  & 3          & 7         & 22        &  & 16         & 49        & 169       &  & 98   & 1497                                   & 6220                                   \\ \Xhline{0.5pt}
            &                           &  & min        & avg       & max       &  & min        & avg       & max       &  & min  & avg                                    & max                                    \\ \cline{4-6} \cline{8-10} \cline{12-14}
$p = 13$    & C\&CG Method              &  & 2          & 4         & 9         &  & 8          & 23        & 37        &  & 15   & 130                                    & 342                                    \\
            & Traditional Reformulation \eqref{eqn:FLP_direct_reformulation_MAD} &  & 3          & 10        & 16        &  & 16         & 41        & 106       &  & 26   & 475                                    & 1332                                   \\      \Xhline{1.0pt}
$|I|=50$    &                           &  & \multicolumn{3}{c}{$\gamma = 0.4$} &  & \multicolumn{3}{c}{$\gamma = 0.3$} &  & \multicolumn{3}{c}{$\gamma = 0.2$}                                                     \\ \cline{4-14}
            &                           &  & min        & avg       & max       &  & min        & avg       & max       &  & min  & avg                                    & max                                    \\ \cline{4-6} \cline{8-10} \cline{12-14}
$p =    10$ & C\&CG Method              &  & 3          & 4         & 5         &  & 32         & 80        & 213       &  & 798  & 2599                                   & \multicolumn{1}{r}{6657}               \\
            & Traditional Reformulation \eqref{eqn:FLP_direct_reformulation_MAD}  &  & 7          & 10        & 15        &  & 185        & 223       & 275       &  & 1637 & \multicolumn{1}{r}{\textgreater{}6114} & \multicolumn{1}{r}{\textgreater{}7200} \\ \Xhline{0.5pt}
            &                           &  & min        & avg       & max       &  & min        & avg       & max       &  & min  & avg                                    & max                                    \\ \cline{4-6} \cline{8-10} \cline{12-14}
$p =    13$ & C\&CG Method              &  & 3          & 7         & 14        &  & 22         & 53        & 115       &  & 621  & 1151                                   & \multicolumn{1}{r}{1946}               \\
            & Traditional Reformulation \eqref{eqn:FLP_direct_reformulation_MAD}  &  & 19         & 26        & 34        &  & 142        & 194       & 347       &  & 5888 & \multicolumn{1}{r}{\textgreater{}6830} & \multicolumn{1}{r}{\textgreater{}7200} \\ \Xhline{0.5pt}
            &                           &  & min        & avg       & max       &  & min        & avg       & max       &  & min  & avg                                    & max                                    \\ \cline{4-6} \cline{8-10} \cline{12-14}
$p = 17$    & C\&CG Method              &  & 5          & 9         & 13        &  & 66         & 115       & 256       &  & 713  & 1434                                   & 2564                                   \\
            & Traditional Reformulation \eqref{eqn:FLP_direct_reformulation_MAD}  &  & 23         & 28        & 36        &  & 76         & 202       & 378       &  & 538  & \multicolumn{1}{r}{\textgreater{}2788} & \multicolumn{1}{r}{\textgreater{}7200} \\             
\Xhline{1.0pt}
\end{tabular}
\end{table}

\begin{table}[t]  
\scriptsize
\center
\renewcommand{\arraystretch}{0.8}
{\caption{Solution time (in seconds) over five randomly generated instances with $\gamma\in\{0.4,0.3,0.2\}$ using the Gini deviation. \textit{Note:} Solution times with `$>$' indicate that one or more instances cannot be solved within the imposed two-hour time limit.  } \label{table:FLP_comp_time_Gini_deviation} }
\begin{tabular}{clrrrrrrrrrrrrrr}
\Xhline{1.0pt}
$|I|=40$    &                                &   & \multicolumn{3}{c}{$\gamma = 0.4$} & & \multicolumn{3}{c}{$\gamma = 0.3$} &  & \multicolumn{3}{c}{$\gamma = 0.2$}                                                     \\ \cline{4-14}
            &                                &              & min        & avg       & max       &  & min        & avg       & max       &  & min  & avg                                    & max                                    \\ \cline{4-6} \cline{8-10} \cline{12-14}
$p =    8$  & Our Reformulation \eqref{eqn:FLP_dual_reformulation_Gini_deviation}             &              & 21         & 27        & 35        &  & 28         & 51        & 90        &  & 37   & 343                                    & 1101                                   \\
            & Traditional Reformulation \eqref{eqn:FLP_direct_reformulation_Gini_deviation} & & 11         & 26        & 42        &  & 36         & 92        & 179       &  & 43   & 869                                    & 2861                                   \\  \Xhline{0.5pt}
            &                                &              & min        & avg       & max       &  & min        & avg       & max       &  & min  & avg                                    & max                                    \\ \cline{4-6} \cline{8-10} \cline{12-14}
$p =    10$ & Our Reformulation \eqref{eqn:FLP_dual_reformulation_Gini_deviation}             &              & 21         & 27        & 37        &  & 26         & 41        & 78        &  & 25   & 102                                    & 245                                    \\
            & Traditional Reformulation \eqref{eqn:FLP_direct_reformulation_Gini_deviation} & & 17         & 37        & 60        &  & 25         & 56        & 120       &  & 37   & \multicolumn{1}{r}{\textgreater{}1545} & \multicolumn{1}{r}{\textgreater{}7200} \\   \Xhline{0.5pt}
            &                                &              & min        & avg       & max       &  & min        & avg       & max       &  & min  & avg                                    & max                                    \\ \cline{4-6} \cline{8-10} \cline{12-14}
$p = 13$    & Our Reformulation \eqref{eqn:FLP_dual_reformulation_Gini_deviation}            &              & 21         & 25        & 29        &  & 23         & 33        & 52        &  & 26   & 45                                     & 85                                     \\
            & Traditional Reformulation \eqref{eqn:FLP_direct_reformulation_Gini_deviation}  & & 12         & 37        & 66        &  & 42         & 51        & 58        &  & 44   & 60                                     & 94                                     \\  \Xhline{1.0pt}                        
$|I|=50$    &                                &              & \multicolumn{3}{c}{$\gamma = 0.4$} &  & \multicolumn{3}{c}{$\gamma = 0.3$} &  & \multicolumn{3}{c}{$\gamma = 0.2$}                                                     \\ \cline{4-14}
            &                                &              & min        & avg       & max       &  & min        & avg       & max       &  & min  & avg                                    & max                                    \\ \cline{4-6} \cline{8-10} \cline{12-14}
$p =    10$ & Our Reformulation \eqref{eqn:FLP_dual_reformulation_Gini_deviation}             &              & 61         & 74        & 103       &  & 226        & 333       & 579       &  & 655  & 1652                                   & 3351                                   \\
            & Traditional Reformulation \eqref{eqn:FLP_direct_reformulation_Gini_deviation} & & 53         & 105       & 212       &  & 207        & 426       & 633       &  & 3477 & \multicolumn{1}{r}{\textgreater{}6512} & \multicolumn{1}{r}{\textgreater{}7200}\\ \Xhline{0.5pt}         
            &                                &              & min        & avg       & max       &  & min        & avg       & max       &  & min  & avg                                    & max                                    \\ \cline{4-6} \cline{8-10} \cline{12-14}
$p =    13$ & Our Reformulation \eqref{eqn:FLP_dual_reformulation_Gini_deviation}             &              & 42         & 60        & 75        &  & 57         & 205       & 315       &  & 424  & 629                                    & 994                                    \\
            & Traditional Reformulation \eqref{eqn:FLP_direct_reformulation_Gini_deviation} & & 73         & 89        & 101       &  & 95         & 206       & 335       &  & 272  & \multicolumn{1}{r}{\textgreater{}2017} & \multicolumn{1}{r}{\textgreater{}7200} \\ \Xhline{0.5pt}
            &                                &              & min        & avg       & max       &  & min        & avg       & max       &  & min  & avg                                    & max                                    \\ \cline{4-6} \cline{8-10} \cline{12-14}
$p = 17$    & Our Reformulation \eqref{eqn:FLP_dual_reformulation_Gini_deviation}             &              & 39         & 55        & 61        &  & 54         & 72        & 99        &  & 85   & 673                                    & 2194                                   \\
            & Traditional Reformulation \eqref{eqn:FLP_direct_reformulation_Gini_deviation} &  & 89         & 109       & 129       &  & 121        & 154       & 190       &  & 133  & 591                                    & 1511                                   \\  
\Xhline{1.0pt}
\end{tabular}
\end{table}

\subsection{Fair Resource Allocation} \label{subsec:expt_resource_allocation}

In this section, we illustrate our stability results in the context of a fair resource allocation problem. Specifically, we consider a decision-maker who wants to efficiently distribute $R$ resources among $N$ individuals (or groups) while ensuring fairness. Let $x_i$ denote the allocation of resources to individual $i \in [N]$. The benefit (or positive outcome) resulting from allocating $x_i$ resources to $i$ is quantified as $u_i=a_ix_i$, where $a_i$ may represent the efficiency per unit resource allocated to $i \in [N]$. The goal is to find allocation decisions that simultaneously maximize the total benefits (efficiency) and minimize a measure of unfairness in the distribution of benefits among individuals. Using this notation, we formulate the problem as
\begin{subequations}
\begin{align}
\underset{\xb,\,\ub}{\textup{minimize}} \,\quad &   -\one^\tp\ub+ \gamma\nu(\ub) \label{prob:resource_allocation_obj}\\
\textup{subject to} \quad &   u_i=a_ix_i,\quad\forall i\in[N],\\
&  \sum_{i=1}^N x_i \leq R,\\
&  x_i\geq 0,\quad\forall i\in[N].
\end{align}\label{prob:resource_allocation}%
\end{subequations}
To ensure that the quadratic growth condition holds (see Remark~\ref{rem:quad_growth_cond}), we add a very small quadratic term $\varepsilon\ub^\tp\ub$ in the objective function with $\varepsilon=10^{-8}$ so that the efficiency measure $f(\ub):=-\one^\tp\ub + \varepsilon\ub^\tp\ub$ is strongly convex \citep{Kakade_Tewari:2008, Shalev-Shwartz_Zhang:2012}. We observe that the results and conclusions remain consistent whether or not we incorporate this negligible quadratic term into the objective function.

Let us now investigate the effects of employing different convex fairness measures on the optimal solution and value of problem \eqref{prob:resource_allocation}. As the baseline measure $\nu_0$, we use the maximum absolute deviation from the mean (MaxMAD), i.e., $\nu_0=\nu^\text{MaxMAD}$. Then, we study the differences in optimal value and solution of \eqref{prob:resource_allocation} under $\nu_0$ and each measure $\nu_k$ from the set $\calV=\{\nu_k\}_{k \in K}$.  For illustrative purposes, we consider $\calV=\{\nu^\text{MAD},\nu^\text{GD},\nu^\text{SMaxPD}\}$, where $\nu^\text{MAD}$, $\nu^\text{GD}$, and $\nu^\text{SMaxPD}$ are the absolute deviation from mean, Gini deviation, and sum of maximum pairwise deviation, respectively; see Table~\ref{table:existing_FM} for the definition of these measures. To ensure that the efficiency and fairness measures are comparable and on the same scale, we normalize the fairness measure by dividing it by the maximum weight $w_\text{max}$ of its dual set $\calW_\nu$, i.e., $w_\text{max}=\sup_{\wb\in\calW_\nu}\norms{\wb}_\infty$.

Recall from Theorem~\ref{thm:stability_analysis} that for a fixed $\gamma$, the differences in optimal value and solution under two convex fairness measures $\nu_0$ and $\nu_k \in \calV$  are upper bounded by a constant multiplied by the Hausdorff distance between the dual sets of these measures. To illustrate this numerically, we set $R=100$ and consider instances with $N\in\{20,30,40\}$, $\gamma\in\{0,0.02,0.04,\dots,2\}$, and $a_i$ generated from a uniform distribution on $[1,10]$ for all $i\in[N]$. For each $N$, $\gamma$, and fairness measure, we generate and solve  $T=100$ instances. Let $\upsilon_0^t(N,\gamma)$ and $\xb_0^t(N,\gamma)$ respectively be the optimal value and optimal solution to \eqref{prob:resource_allocation} with $\nu=\nu_0$ in replication $t\in[T]$. Similarly, let $\upsilon_k^t(N,\gamma)$ and $\xb_k^t(N,\gamma)$ respectively be the optimal value and optimal solution to \eqref{prob:resource_allocation} with $\nu=\nu_k$ for instance $t\in[T]$. We compute the average differences in optimal value and solution to \eqref{prob:resource_allocation} under  $\nu_0$ and $\nu_k$ as $\texttt{val\_diff}_k(N,\gamma)=T^{-1}\sum_{t=1}^T |\upsilon_0^t(N,\gamma)-\upsilon_k^t(N,\gamma)|$ and $\texttt{sol\_diff}_k(N,\gamma)=T^{-1}\sum_{t=1}^T \norms{\xb_0^t(N,\gamma)-\xb_k^t(N,\gamma)}_2$, respectively. Finally, we compute the  Hausdorff distance between the dual sets of $\nu^\text{MaxMAD}$ and $\nu\in\{\nu^\text{MAD},\nu^\text{GD},\nu^\text{SMaxPD}\}$, which we call $d_{H}^{\text{MAD}}$, $d_{H}^{\text{GD}}$, and $d_{H}^{\text{SMaxPD}}$, respectively, for different $N$.

Figures~\ref{fig:expt_RA_Hausdorff_distance}, \ref{fig:expt_RA_diff_opt_val_fix_range}, and \ref{fig:expt_RA_diff_opt_sol_fix_range} illustrate $d_{H}^k$, $\texttt{val\_diff}_k(N,\gamma)$, and $\texttt{sol\_diff}_k(N,\gamma)$, respectively. We observe the following from these figures. First, the Hausdorff distance $d_{H}^{\text{SMaxPD}}$  between the dual sets of $\nu_0=\nu^\text{MaxMAD}$ and $\nu^\text{SMaxPD}$ is the smallest under all $N\in\{20,30,40\}$. It follows from Theorem~\ref{thm:stability_analysis} that the differences between the optimal values and solutions under $\nu_0$ and $\nu^\text{SMaxPD}$ could be smaller than those between $\nu_0$ and $\nu_k$ for $k\in\{\text{MAD},\text{GD}\}$. Indeed, Figures \ref{fig:expt_RA_diff_opt_val_fix_range} and \ref{fig:expt_RA_diff_opt_sol_fix_range} clearly show that $\texttt{val\_diff}_\text{SMaxPD}(N,\gamma)$ and $\texttt{sol\_diff}_\text{SMaxPD}(N,\gamma)$ are the smallest for all $N\in\{20,30,40\}$ and $\gamma\in[0,2]$. Second, the Hausdorff distance $d_{H}^{\text{GD}}$ between the dual sets of $\nu^\text{MaxMAD}$ and $\nu^\text{GD}$ increases with $N$ (the number of individuals or groups). Similarly, the Hausdorff distance $d_{H}^{\text{MAD}}$ between the dual sets of $\nu^\text{MaxMAD}$ and $\nu^\text{MAD}$ increases with $N$. Again, it follows from Theorem~\ref{thm:stability_analysis} that the differences between the optimal value and solution of \eqref{prob:FM_eff_min_obj} under $\nu^\text{MaxMAD}$ and $\nu^\text{GD}$ (and similarly under $\nu^\text{MaxMAD}$  and $d_{H}^{\text{MAD}}$) could be larger for larger $N$. Indeed, $\texttt{val\_diff}_k(N,\gamma)$ and $\texttt{sol\_diff}_k(N,\gamma)$ are generally larger under larger values of $N$ for $k\in\{\text{MAD},\text{GD}\}$  as illustrated in Figures~\ref{fig:expt_RA_diff_opt_val_fix_range} and~\ref{fig:expt_RA_diff_opt_sol_fix_range}. In contrast, $d_{H}^{\text{SMaxPD}}$ is relatively stable as $N$ increases, suggesting that the optimal values and solutions under $\nu^\text{MaxMAD}$ and $\nu^\text{SMaxPD}$ do not deviate significantly as $N$ increases, as illustrated in Figures~\ref{fig:expt_RA_diff_opt_val_fix_range} and~\ref{fig:expt_RA_diff_opt_sol_fix_range}.

These results demonstrate how one can use the stability results in Section~\ref{sec:stability_analysis} to investigate the stability of the optimal value and solution of the fairness-promoting optimization problem within their decision-making contexts. In particular, the results show how the Hausdorff distance between the dual sets of two different convex fairness measures serves as an indicator of the differences between the optimal value and solution obtained under these measures (Theorem~\ref{thm:stability_analysis}).

\begin{figure}[t]
    \centering
    \includegraphics[width=0.42\textwidth]{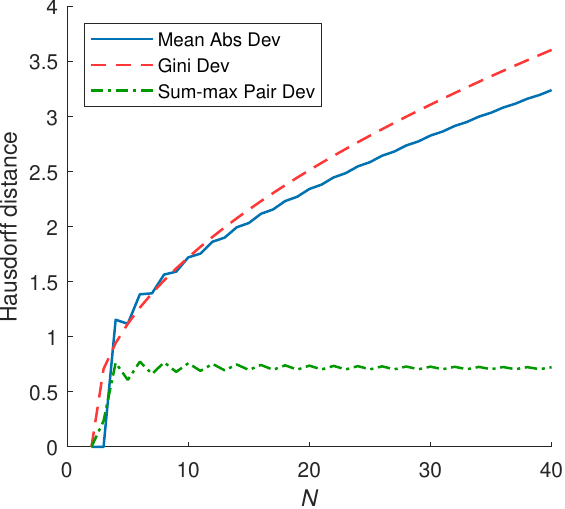}
    \caption{Hausdorff distance $d_H$ between the dual sets of $\nu^\text{MaxMAD}$ and $\nu\in\{\nu^\text{MAD},\nu^\text{GD},\nu^\text{SMaxPD}\}$}
    \label{fig:expt_RA_Hausdorff_distance}
\end{figure}

\begin{figure}[t]
    \centering
    \includegraphics[width=\textwidth]{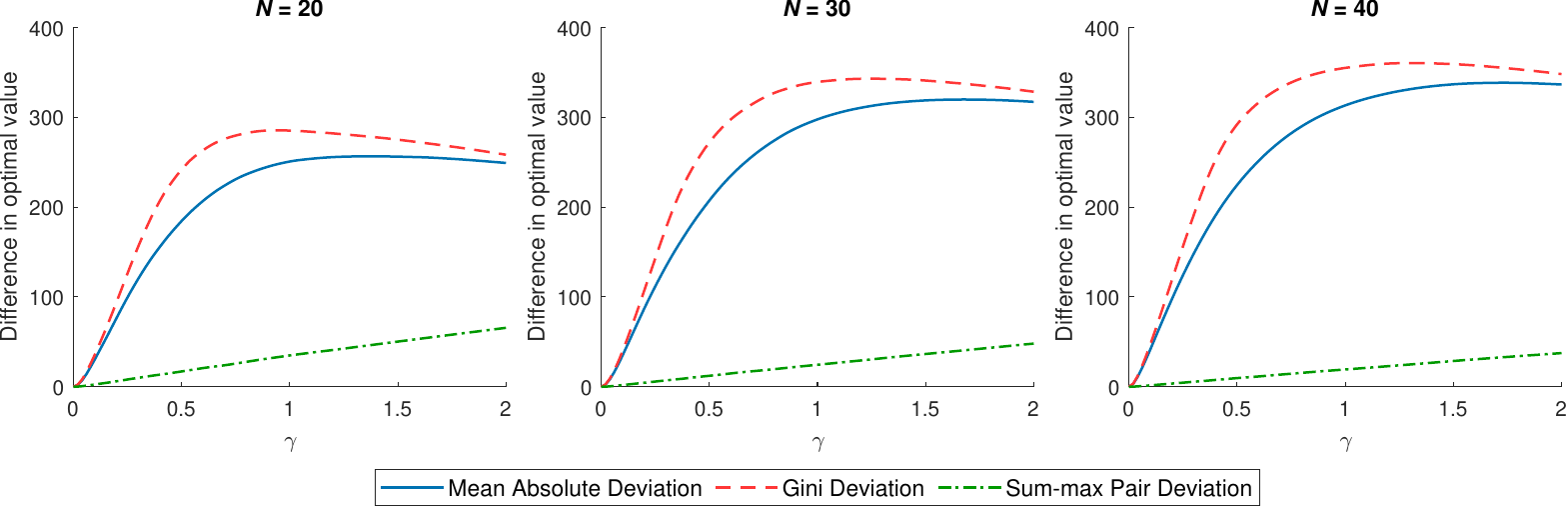}
    \caption{Average difference in optimal value $\texttt{val\_diff}_k(N,\gamma)$ for different values of $N$ and $\gamma$}
    \label{fig:expt_RA_diff_opt_val_fix_range}
\end{figure}

\begin{figure}[t]
    \centering
    \includegraphics[width=\textwidth]{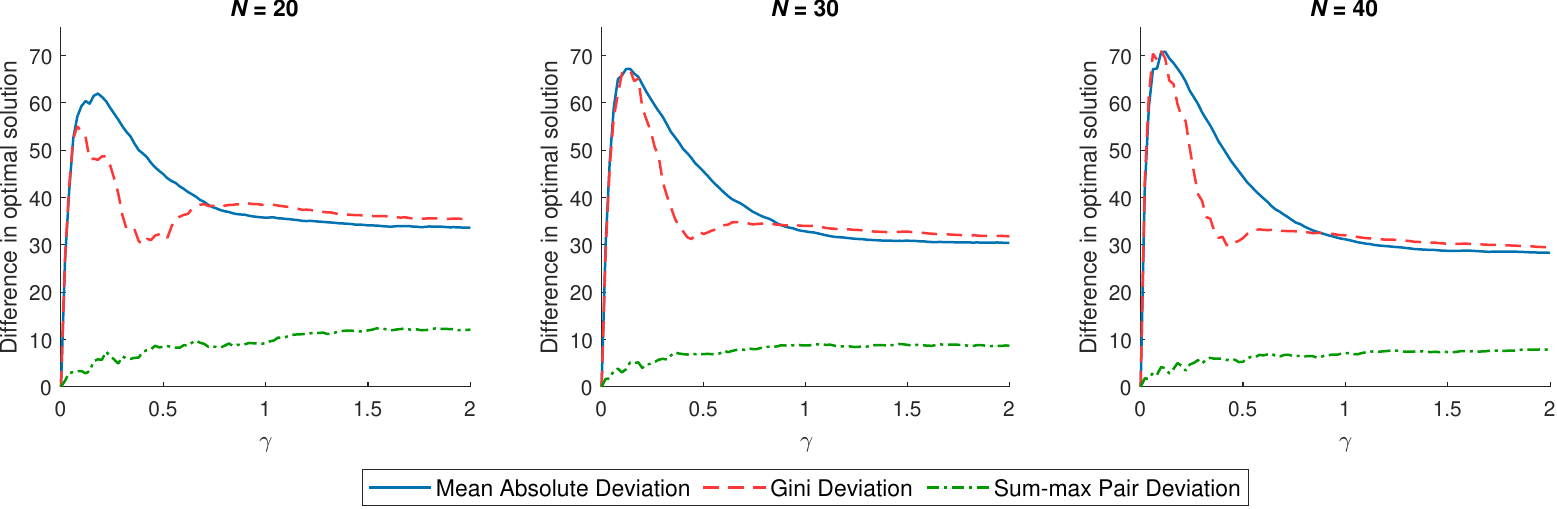}
    \caption{Average difference in optimal solution $\texttt{sol\_diff}_k(N,\gamma)$ for different values of $N$ and $\gamma$}
    \label{fig:expt_RA_diff_opt_sol_fix_range}
\end{figure}

\section{Conclusion}   \label{sec:conclusion}
In this paper, we propose a new unified framework for analyzing a class of convex fairness measures suitable for optimization contexts, including characterization of their theoretical properties, reformulations, and solution approaches. We first introduce a new class of order-based fairness measures, which serves as the building block of our proposed class of convex fairness measures. We analyze the order-based fairness measures and provide an axiomatic characterization for such measures. Then, we introduce our class of convex fairness measures, discuss their theoretical properties, and derive a dual representation of these measures. This dual representation allows for equivalently presenting any convex fairness measure as a robustified order-based fairness measure. Moreover, it provides a mechanism for investigating the equivalence of convex fairness measures from a geometric perspective based on their dual sets. In addition, using the dual representation, we propose a generic unified framework for optimization problems with a convex fairness measure objective or constraint, including reformulations and solution approaches. This provides decision-makers with a unified optimization tool to solve fairness-promoting optimization models with their favorite choice of convex fairness measures. Moreover, using the dual representation, we conduct a stability analysis on the choice of convex fairness measure in the objective of optimization models. Finally, we demonstrate the computational advantages of our unified framework over traditional ones and illustrate our stability results using a fair facility location and fair resource allocation problems, respectively. 

Our paper presents a new step toward deriving unified frameworks for analyzing fairness measures and formulating fairness-promoting optimization problems. We suggest the following areas for future steps. First, future studies could focus on extending the proposed framework to stochastic settings, where problem parameters, the fairness measure value, and the outcome vector  $\ub$ are random. This extension will allow researchers and practitioners to address fairness concerns in various application domains (e.g., facility location, scheduling, humanitarian logistics) where the problem involves random factors such as random travel time, demand, and service time. Second, in future directions, we aim to propose new classes of stochastic optimization approaches with fairness criteria that combine our work with different methodologies, such as (distributionally) robust optimization and their solution approaches. 
Finally, we acknowledge that our paper does not provide explicit guidelines on which class of fairness measures (order-based or convex) to adopt in a particular context. However, our theoretical and computational results in Sections~\ref{sec:opt_w_fairness_measures} and \ref{subsec:expt_facility_location} indicate that optimization problems with order-based fairness measures may be easier to solve than those with convex fairness measures. This observation highlights the need for further investigation into optimal solutions' properties under each class and their practical implications for fairness in specific application domains. We consider this an important avenue for future research.

\appendix
\newpage

\section{Properties of the Deviation-Based Fairness Measures in Table~\ref{table:existing_FM}} \label{appdx:eg_inequity_measures}

\subsection{Equivalence}

In this section, we investigate the equivalence between fairness measures in Table~\ref{table:existing_FM}. We say two fairness measures $\phi^1$ and $\phi^2$ are equivalent if there exists $\beta>0$ such that $\phi^1(\ub) = \beta\phi^2(\ub)$ for all $\ub\in\R^N$. Hence, replacing $\phi^2$ by $\phi^1$ as a fairness criterion in the fairness-promoting optimization models, e.g., models \eqref{prob:FM_eff_min_obj} and \eqref{prob:FM_eff_min_constraint}, essentially scales the weight on the fairness criterion $\gamma$ or the upper bound $\eta$ by a factor of $\beta>0$ (i.e., from $\phi^2$ to $\beta\phi^2$). It is easy verify that all of these measures are equivalent when $N=2$. However, in Proposition \ref{prop:equivalence_FM}, we show that only some of these measures are equivalent when $N\geq 3$, 

\begin{proposition} \label{prop:equivalence_FM}
Only the following equivalence relationships between fairness measures shown in Table~\ref{table:existing_FM} hold: (a) for any $N\geq3$, (i) and (iii) are equivalent; (b) for any $N\geq3$, (vi) and (vii) are equivalent; (c) when $N=3$, (i), (ii), and (iii) are equivalent; (d) when $N=3$, (iv), (vi), and (vii) are equivalent. The remaining pairs of fairness measures are not equivalent.
\end{proposition}

\begin{proof}
We first prove (a)--(d). Without loss of generality, we assume that $\ub$ is sorted in ascending order, i.e., $u_1\leq u_2\leq \dots \leq u_N$.
\begin{enumerate}[topsep=1pt,itemsep=-5pt]
    \item[(a)] Note that the maximum pairwise difference (iii) is equal to $u_N-u_1$, which is the same as the range (i).
    
    \item[(b)] We claim that $\max_{i\in[N]} \sum_{j=1}^N |u_i-u_j|=N\max_{i\in[N]} |u_i-\ubar|$. To prove this claim, note that we can write (vii) as
    $$\max_{i\in[N]} \sum_{j=1}^N |u_i-u_j|=\max\Bigg\{\sum_{i=2}^N u_i - (N-1)u_1,\, (N-1)u_N - \sum_{i=1}^{N-1} u_i \Bigg\}=: \max\{C_1, C_2\},$$
    where $C_1$ and $C_2$ represent the first and second expressions in the $\max$ operator, respectively. Consider the case when $C_1 \leq C_2$. This implies that $2\sum_{j=2}^{N-1} u_i\leq (N-2) (u_N+u_1)$. Adding $2(u_1+u_N)$ on both sides of the inequality results in $2\ubar \leq u_1+u_N$, implying that $\ubar - u_1 \leq u_N-\ubar$. Hence, we have 
    $$\max_{i\in[N]} \sum_{j=1}^N |u_i-u_j|=(N-1)u_N - \sum_{i=1}^{N-1} u_i = N(u_N-\ubar)=N\max_{i\in[N]}|u_i-\ubar|.$$
    A similar argument holds for the case when $C_1 \geq C_2$.
    
    \item[(c)] By (a), it suffices to show that (i) and (ii) are equivalent when $N=3$. By \cite{Mesa_et_al:2003}, we can write (ii) as
    $$\sum_{i=1}^3 \sum_{j=1}^3 |u_i-u_j| = \sum_{i=1}^3 2(2i-4) u_i = 4(u_3-u_1),$$
    which shows the equivalence between (i) and (ii).
    
    \item[(d)] By (b), it suffices to show that (iv) and (vi) are equivalent when $N=3$. First, we claim that if $\ubar\in[u_2,u_3]$, then $u_3-\ubar \geq \ubar - u_1$. Indeed, since $\ubar \geq u_2$, we have $u_1+u_3\geq 2u_2$. Adding $2(u_1+u_3)$ on both sides of the inequality results in $u_1+u_3 \geq 2\ubar$, implying that $u_3-\ubar \geq \ubar -u_1$. Therefore, when $\ubar\in[u_2,u_3]$, we have
    \allowdisplaybreaks
    \begin{align*}
        \sum_{i=1}^3|u_i-\ubar| &= (\ubar-u_1)+(\ubar-u_2)+(u_3-\ubar)= \frac{1}{3}(u_1+u_2+u_3) - u_1-u_2+u_3 \\
        & =\frac{2}{3}(2u_3 - u_1- u_2) = 2(u_3-\ubar) = 2\max_{i=1,2,3} |u_i-\ubar|.
    \end{align*}
    Similarly, in the case when $\ubar\in[u_1,u_2]$, we have $\ubar-u_1\geq u_3-\ubar$. Thus,
    \allowdisplaybreaks
    \begin{align*}
        \sum_{i=1}^3|u_i-\ubar| &= (\ubar-u_1)+(u_2-\ubar)+(u_3-\ubar)= -\frac{1}{3}(u_1+u_2+u_3) - u_1+u_2+u_3 \\
        & =\frac{2}{3}(-2u_1 + u_2 + u_3) = 2(\ubar-u_1) = 2\max_{i=1,2,3} |u_i-\ubar|.
    \end{align*}
    This proves the equivalence between (iv) and (vi).
\end{enumerate}

Finally, we show that the remaining pairs of fairness measures are not equivalent. To prove two measures $\phi$ and $\widetilde{\phi}$ are not equivalent, it suffices to find vectors $\ub^1$ and $\ub^2$ such that $\phi(\ub^1)=\phi(\ub^2)$ but $\widetilde{\phi}(\ub^1)\ne \widetilde{\phi}(\ub^2)$. In Table \ref{table:pf_equivalence_FM},  we provide examples showing that the remaining pairs of fairness measures are not equivalent. Specifically, example A shows that the pairs $\{$(iv,i), (iv,ii), (iv,v), (iv,viii), (vi,i), (vi,ii), (vi,v), (vi,viii)$\}$ are not equivalent; example B and B' shows that the pairs $\{$(viii,i), (viii,ii), (viii,v)$\}$ are not equivalent; example C shows that the pair (i, ii) is not equivalent when $N>3$ and (i,v) is not equivalent for all $N$; example D shows that the pair (ii, v) is not equivalent; example E shows that the pair (iv, vi) is not equivalent when $N>3$. Note that in example B, fairness measure (ii) at $\ub^1$ and $\ub^2$ are equal when $N=5$. Example B' shows that the pair (viii,ii) is not equivalent even when $N=5$. 
\begin{table}[t]  
\scriptsize
\center
\renewcommand{\arraystretch}{1}
\caption{Examples that some fairness measures are not equivalent} 
\begin{tabular}{c|l|cccccc}
\Xhline{1.0pt}
E.g. & Outcome vectors (in $\R^N$) & (i) & (ii) & (iv) & (v) & (vi) & (viii) \\ \hline
A     & $\ub^1=(1,\,2,\,2.5,\dots,\,2.5,\,4.5)$ & $3.5$ & $14+8(N-3)$ & $4$ & $\sqrt{6.5}$ &        $2$ & $9.5+2(N-3)$ \\
      & $\ub^2=(1,\,1,\,2,\dots,\,2,\,4)$       & $3$   & $12+8(N-3)$ & $4$ & $\sqrt{6}$   & $2$ & $9+2(N-3)$ \\ \hline
B     & $\ub^1=(2,\,5,\,5,\dots,\,5,\,9)$       & $7$   & $28+14(N-3)$ & / &                        $\sqrt{25-\frac{1}{N}}$ & / & $18+4(N-3)$ \\
      & $\ub^2=(2,\,2,\,4,\dots,\,4,\,8)$       & $6$   & $24+16(N-3)$ & /             &    $\sqrt{24}$            & /  & $18+4(N-3)$ \\ \hline
B'    & $\ub^1=(2,\,5,\,5,\,6,\,9)\in\R^5$             & / & $60$  & / &  /  &  / & $26$  \\               & $\ub^2=(2,\,2,\,4,\,4,\,8)\in\R^5$             & / & $56$  & / &  /  &  / & $26$  \\   \hline       
C     & $\ub^1=(2,\,5,\,\frac{16}{3},\dots,\,\frac{16}{3},\,9)$ & $7$ &                             $28+\frac{44}{3}(N-3)$ & / &       $\sqrt{\frac{74}{3}}$ & / & / \\
      & $\ub^2=(2,\,2,\,\frac{13}{3},\dots,\,\frac{13}{3},\,9)$ & $7$ &  $28+\frac{56}{3}(N-3)$ & / & $\sqrt{\frac{98}{3}}$ & / & / \\ \hline
D     & $\ub^1=(1,\,2,\,3,\dots,\,3,\,6)$ &                                                         / & $20+12(N-3)$ & / & $\sqrt{14}$ & / & / \\                                 
      & $\ub^2=3+(0,\,0,\,\frac{\sqrt{21}}{3},\dots,\,\frac{\sqrt{21}}{3},\,\sqrt{21})$  &    / & $4\sqrt{21}+\frac{8\sqrt{21}}{3}(N-3)$ & / & $\sqrt{14}$ & / & / \\      \hline
E     & $\ub^1=(1,\,7,\,7,\dots,\,7,\,8,\,12)$ & / & / & $12$ & / & $6$ & / \\
      & $\ub^2=(5,\,10,\,10.5,\dots,\,10.5,\,13,\,14)$ & / & / & $12$ & / & $5.5$ & / \\
\Xhline{1.0pt}
\end{tabular}\label{table:pf_equivalence_FM}
\end{table}
\end{proof}

Tables \ref{table:apdx_equivalence_FM_1}--\ref{table:apdx_equivalence_FM_2} summarize the equivalence of the fairness measures shown in Table~\ref{table:existing_FM}. The two groups of equivalent fairness measures proved in Proposition \ref{prop:equivalence_FM} are highlighted in red and blue with `Equiv.' representing equivalence in the tables. If a given pair of fairness measures is not equivalent, one of the corresponding counterexamples from A to E is stated (see Table \ref{table:pf_equivalence_FM}).

\begin{table}[t]  
\small
\center
\renewcommand{\arraystretch}{0.6}
\caption{Equivalence of fairness measures or counterexamples when $N=3$} 
\begin{tabular}{l|cccccccc}
\Xhline{1.0pt}
$N = 3$                    & {\color{red} i} & {\color{red} ii}     & {\color{red} iii}    & {\color{blue} iv} & v   & {\color{blue} vi}     & {\color{blue} vii}    & viii \\ \hline
{\color{red} i}   & /                        & {\color{red} Equiv.} & {\color{red} Equiv.} & A                         & C   & A                             & A                             & B    \\
{\color{red} ii}  &                          & /                             & {\color{red} Equiv.} & A                         & C/D & A                             & A                             & B    \\
{\color{red} iii} &                          &                               & /                             & A                         & C   & A                             & A                             & B    \\
{\color{blue} iv}  &                          &                               &                               & /                         & A   & {\color{blue} Equiv.} & {\color{blue} Equiv.} & A    \\
v                          &                          &                               &                               &                           & /   & A                             & A                             & B    \\
{\color{blue} vi}  &                          &                               &                               &                           &     & /                             & {\color{blue} Equiv.} & A    \\
{\color{blue} vii} &                          &                               &                               &                           &     &                               & /                             & A    \\
viii                       &                          &                               &                               &                           &     &                               &                               & /   \\
\Xhline{1.0pt}
\end{tabular}\label{table:apdx_equivalence_FM_1}
\end{table}

\begin{table}[t]  
\small
\center
\renewcommand{\arraystretch}{0.6}
\caption{Equivalence of fairness measures or counterexamples when $N>3$} 
\begin{tabular}{l|cccccccc}
\Xhline{1.0pt}
$N > 3$                    & {\color{red} i} & ii & {\color{red} iii}    & iv & v & {\color{blue} vi} & {\color{blue} vii}    & viii  \\ \hline
{\color{red} i}   & /                        & C  & {\color{red} Equiv.} & A                         & C & A                         & A                             & B     \\
ii  &                          & /                         & C      & A                         & D & A                         & A                             & B, B' \\
{\color{red} iii} &                          &                           & /                             & A                         & C & A                         & A                             & B     \\
iv  &                          &                           &                               & /                         & A & E  & E      & A     \\
v                          &                          &                           &                               &                           & / & A                         & A                             & B     \\
{\color{blue} vi}  &                          &                           &                               &                           &   & /                         & {\color{blue} Equiv.} & A     \\
{\color{blue} vii} &                          &                           &                               &                           &   &                           & /                             & A     \\
viii                       &                          &                           &                               &                           &   &                           &                               & /  \\
\Xhline{1.0pt}
\end{tabular}\label{table:apdx_equivalence_FM_2}
\end{table}

\subsection{Axioms}

In this section, we show that the fairness measures in Table~\ref{table:existing_FM} are convex fairness measures (see Section~\ref{sec:convex_fairness_measures}).
\begin{proposition} \label{prop:existing_FM}
The measures (i)--(viii) in Table \ref{table:existing_FM} are convex fairness measures, i.e., satisfying Axioms~\ref{axiom:continuity}, \ref{axiom:normalization}, \ref{axiom:symmetry}, \ref{axiom:Schur_convex}, \ref{axiom:trans_invariance}, \ref{axiom:pos_homo}, and \ref{axiom:convexity}.
\end{proposition}

\begin{proof}

It is easy to verify that measures (i)--(viii) satisfy Axioms~\ref{axiom:continuity}, \ref{axiom:normalization}, \ref{axiom:symmetry}, \ref{axiom:trans_invariance}, and \ref{axiom:pos_homo}. Next, note that if $\phi$ is convex and symmetric, then $\phi$ is Schur convex \citep{Marshall_et_al:2011}. Thus, it suffices to show that measures (i)--(viii) are convex. In the following, we assume that $\{\ub^1,\ub^2\}\subseteq\R^N$ and $\lambda\in[0,1]$. For measure (i), note that $u_i$ is a linear function in $\ub$. Since a maximum (resp. minimum) of linear functions is convex (resp. concave), measure (i) is also convex. For measure (ii), we have
\allowdisplaybreaks
\begin{align*}
\phi(\lambda\ub^1 + (1-\lambda)\ub^2) &= \sum_{i=1}^N \sum_{j=1}^N \Big| \big[\lambda u^1_i + (1-\lambda) u^2_i \big] - \big[\lambda u^1_j + (1-\lambda) u^2_j \big] \Big| \\
& \leq \lambda \sum_{i=1}^N \sum_{j=1}^N |u^1_i-u^1_j| + (1-\lambda)  \sum_{i=1}^N \sum_{j=1}^N |u^2_j-u^2_j| \\
& = \lambda \phi(\ub^1)+(1-\lambda) \phi(\ub^2).
\end{align*}
For measure (iii), following a similar argument in (ii), we have
\begin{align*}
\phi(\lambda\ub^1 + (1-\lambda)\ub^2) &= \max_{i\in[N]} \max_{j\in[N]} \Big| \big[\lambda u^1_i + (1-\lambda) u^2_i \big] - \big[\lambda u^1_j + (1-\lambda) u^2_j \big] \Big| \\
& \leq \max_{i\in[N]} \max_{j\in[N]} \Big\{ \lambda |u^1_i-u^1_j| + (1-\lambda) |u^2_i-u^2_j| \Big\} \\ 
& \leq \lambda \max_{i\in[N]} \max_{j\in[N]}  |u^1_i-u^1_j| + (1-\lambda)  \max_{i\in[N]} \max_{j\in[N]}  |u^2_i-u^2_j| \\
& = \lambda \phi(\ub^1)+(1-\lambda) \phi(\ub^2).
\end{align*}
For measure (iv), note that $\ubar=\lambda \ubar^1 + (1-\lambda)\ubar^2$, where $\ubar^k=(1/N)\sum_{i=1}^N u^k_i$ for $k\in\{1,2\}$. Convexity follows from a similar argument for measure (ii). For measure (v), we have

\begin{align*}
    \phi(\ub)=\sqrt{\sum_{i=1}^N (u_i-\ubar)^2}=\sqrt{\sum_{i=1}^N \bigg(u_i-\frac{1}{N}\one^\tp\ub\bigg)^2}&=\norm{\ub-\frac{1}{N}\one \one^\tp\ub}_2\\
    &= \norm{\Bigg(I-\frac{1}{N}\one\one^\tp\Bigg)\ub}_2 =: \norm{A\ub}_2, 
\end{align*}
where $I\in\R^{N\times N}$ is the identity matrix. Since $\ell_2$ norm is convex, it follows that $\phi$ is also convex \cite{Bertsekas:2015}. For measure (vi), one can easily verify its convexity by following the same logic used to verify the convexity of (iv). Similarly, one can verify the convexity of (vii) and (viii) by following a similar argument as in measures (ii) and (iii). 
\end{proof}

\section{Examples Related to Order-Based Fairness Measures} \label{appdx:eg_order_based}

\begin{example}[Fair resource allocation] \label{eg:fair_resource_order_based}

Consider the problem of allocating $R$ resources fairly to $N$ individuals, where we use $i\in[N]$ to denote each individual. Let $x_i$ be the number of resources allocated to $i$. The impact on $i$ is measured as $u_i=a_ix_i$, where $a_i$ may represent the efficiency per unit resource allocated to $i$. Moreover, there is a limit $K\leq R$ on the number of resources allocated to each individual. Let $N=6$ and $a_i=i$, for all $i \in [6]$. Suppose we use the order-based fairness measure $\nu_{\wb}(\ub)$ to ensure fair allocations with $w_i=2(2i-7)$, for $i\in[6]$. Then, our fair resource allocation optimization problem can be stated as
\begin{subequations}
\begin{align}
\underset{\xb,\,\ub}{\textup{minimize}}  & \ \  -10u_{(1)}-6u_{(2)}-2u_{(3)}+2u_{(4)}+6 u_{(5)}+10u_{(6)} \\
\textup{subject to}  &  \ \ u_1 = x_1,\, u_2=2x_2,\, \ldots, u_6=6x_6,\, \textstyle\sum_{i=1}^6 x_i = R,\, 0\leq x_i \leq K,\, \forall i\in[6].
\end{align}\label{Fair_allocation}%
\end{subequations}
Clearly, the optimization problem \eqref{Fair_allocation} prioritizes allocating resources to less advantaged individuals (i.e., those with lower efficiency). To illustrate, we solve \eqref{Fair_allocation} numerically (see Section \ref{sec:opt_w_fairness_measures}) with $R=25$. Figure \ref{fig:eg_optimal_allocation_decision} shows the optimal allocation decisions with different values of $K$. It is clear that more resources are allocated to less advantaged individuals. Even when $K$ decreases from $10$ to $7$, i.e, when a smaller upper bound is imposed on $x_i$, more resources are allocated to the less advantaged individuals with a priority to individual $2$, followed by $3$ to $6$. (Note that the resources allocated to individual $1$ decrease because of the decrease in the imposed upper bound $K$ on $x_i$).
\begin{figure}
    \centering
    \includegraphics[scale=0.72]{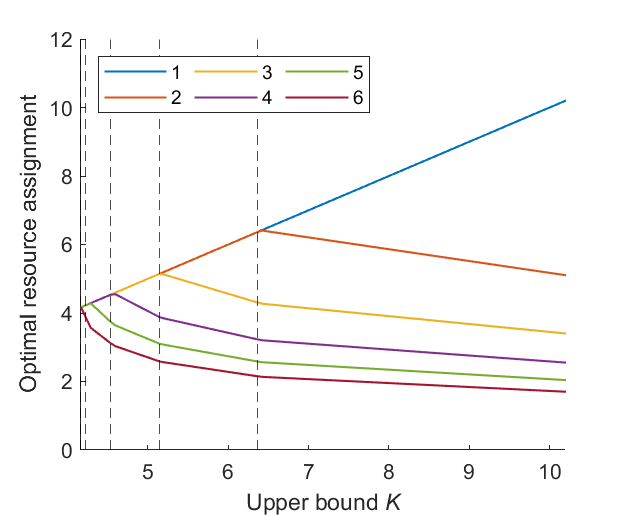}
    \caption{Optimal allocation decisions (each line corresponds to an individual)}
    \label{fig:eg_optimal_allocation_decision}
\end{figure}
\end{example}

\begin{example}[Range] \label{eg:axiom_characterization_range}
Since $\max_{i \in [N]}u_i -\min_{i \in [N]}=-u_{(1)}+u_{(N)}$, it follows that the range is an order-based fairness measure with weight vector $\wb':=(-1,0,\dots,0,1)^\tp\in\R^N$. By Theorem~\ref{thm:axiom_order_based}, we have that Axioms~\ref{axiom:normalization}, \ref{axiom:Schur_convex}, \ref{axiom:pos_homo}, and \ref{axiom:order_based} with $\wb=\wb'$ characterize the range. Specifically, Axiom~\ref{axiom:order_based} with $\wb=\wb'$ implies (a) there exists constant $C>0$ such that $\phi(\ub+\varepsilon\eb_1)=\phi(\ub)-\varepsilon C$ for all $\varepsilon\in[0,u_2-u_1]$ and  $\phi(\ub+\varepsilon\eb_N)=\phi(\ub)+\varepsilon C$ for all $\varepsilon\geq 0$; (b) for all $j\in[2,N-1]$, $\phi(\ub+\varepsilon\eb_j)=\phi(\ub)$ for all $\varepsilon\in[0,u_{j+1}-u_j]$. This indicates that the fairness measure's value decreases (resp. increases) by a constant if the smallest (resp. largest) entry of $\ub$ increases by a small amount $\varepsilon$, and it remains unchanged otherwise. This reflects the nature of the range that its value depends only on the value of the smallest ($u_{(1)}$) and largest ($u_{(N)}$) entries of $\ub$.
\end{example}

\begin{example}[Gini Deviation] \label{eg:axiom_characterization_Gini_deviation}
 We can rewrite Gini deviation as $\sum_{i=1}^N \sum_{j=1}^N |u_i-u_j| = \sum_{i=1}^N 2(2i-1-N) u_{(i)}=\sum_{i=1}^N w'_i u_{(i)}$, where $w'_i=2(2i-1-N)$ \citep{Mesa_et_al:2003}. It is easy to verify that $\wb'\in\calW$ (see Definition \ref{def:order_based_FM}), i.e., Gini deviation is an order-based fairness measure with weight vector $\wb'$. By Theorem \ref{thm:axiom_order_based},  we have that Axioms~\ref{axiom:normalization}, \ref{axiom:Schur_convex}, \ref{axiom:pos_homo}, and \ref{axiom:order_based} with $\wb=\wb'$ characterize Gini deviation. Specifically, Axiom~\ref{axiom:order_based} with $\wb=\wb'$ implies that there exists constant $C>0$ such that $\phi(\ub+\varepsilon\eb_j)=\phi(\ub)+\varepsilon C(2j-N-1)$ for any $\ub\in\Rup^N$, $j\in[N]$, and $\varepsilon\in[0,u_{j+1}-u_j]$. Thus, if we increase the $j$th smallest entry $u_{(j)}$ by a small amount $\varepsilon$, the change in fairness measure's value is proportional to $\varepsilon$ and scales linearly in $j$. In particular, if we increase $u_{(j)}$ for any $j\in[\lfloor (N-1)/2\rfloor]$, the fairness measure's value decreases since $2j-N+1<0$, and the decrease is the most pronounced when we increase the smallest entry ($u_{(1)}$).
\end{example}

\section{Comparison between the Class of Order-Based Fairness Measures and OWA Operators} \label{appdx:comparison_OWA}

Our proposed class of order-based fairness measures differs from OWA operators in the following aspects. First, the OWA operator (and order-median function) takes the form $F_{\wbt}(\ub)=\sum_{i=1}^N \wt_i u_{(i)}$, where the weight vector $\wbt\in\R^N$ satisfies $\wt_i\in[0,1]$ and $\sum_{i=1}^N \wt_i = 1$. In contrast, the weight vector $\wb\in\Rup^N$ in our order-based fairness measure satisfies $\sum_{i=1}^N w_i=0$ with $w_1<0$ and $w_N>0$ (see Definition \ref{def:order_based_FM}). Second, as detailed in \cite{Yager:1988}, the OWA operator is a way to aggregate different values of $u_i$, which was designed for multi-criteria decision-making problems and not for measuring unfairness or inequality. Indeed, if $\wbt\in\Rup^N$, which is a common assumption in the literature (see, e.g., \citealp{Blanco_et_al:2016, Nickel_Puerto:2006, Rodriguez-Chia_et_al:2000}), then we can write $F_{\wbt}$ as $F_{\wbt}(\ub)=f(\ub)+\nu_{\wb}(\ub)$, where $f(\ub)=N^{-1}\sum_{j=1}^N u_j$ and $\wb=\wbt-N^{-1}\one$. Here, $\nu_{\wb}$ is an order-based fairness measure when $\wbt\ne N^{-1}\one$. Thus, $F_{\wb}$ (OWA) can be viewed as a special case of the general problem defined in \eqref{prob:FM_eff_min_obj} with a specific choice of the inefficiency measure $f(\ub)=N^{-1}\sum_{j=1}^N u_j$ (i.e., mean of $\ub)$ and the fairness measure $\nu_{\wb}(\ub)$.  In contrast, our framework allows adopting any classical inefficiency measure $f(\ub)$ in \eqref{prob:FM_eff_min_obj}.

\section{Relative Convex Fairness Measures} \label{appdx:relative_convex_fairness_measures}

In this section, building on our analyses of the convex fairness measure proposed in Section~\ref{sec:convex_fairness_measures}, we introduce its relative counterpart and study its properties. We relegate all proofs to \ref{appdx:proofs}. Note that while absolute measures are characterized by translation invariance (Axiom~\ref{axiom:trans_invariance}), relative measures are characterized by the following scale invariance axiom \citep{Chakravarty:1999, Mussard_Mornet:2019}. 

\setaxiomtag{SI}
\begin{axiom}[Scale Invariance]\label{axiom:scale_invariance}
$\phi(\alpha\ub)=\phi(\ub)$ for any $\alpha>0$.
\end{axiom}

Axiom~\ref{axiom:scale_invariance} guarantees that if all $u_i$ are multiplied by a positive scalar, then the value of the \textit{relative} fairness measure will not change (i.e., the inequality across the population would not change). This implies that a change in the measurement unit of the vector $\ub$ (e.g., from US dollar to British pound if $\ub$ represents costs) will not affect the fairness measure, i.e., the relative fairness measure is unitless.

We are now ready to introduce the relative counterpart of the class of convex fairness measures, which we call relative convex fairness measures.

\begin{definition} \label{def:relative_convex_fairness_measure}
Let $\nu:\R^N_+\rightarrow\R$ be a convex fairness measure and $g:\R^N_+\rightarrow\R$ be a continuous function. A fairness measure $\rho:\R^N_+\rightarrow\R$ is a relative convex fairness measure if  
\begin{equation}\label{eqn:rel_CFM}
\rho(\ub) = \frac{\nu(\ub)}{g(\ub)}
\end{equation}
and $\rho$ satisfies Axioms \ref{axiom:symmetry}, \ref{axiom:Schur_convex}, \ref{axiom:scale_invariance}, and the following Axiom~\ref{axiom:normalization_rel}. Here, we adopt the conventions $0/0=0$ and $a/0=\infty$ for any $a>0$.
\end{definition}

\setaxiomtag{NR}
\begin{axiom}[Normalization-Relative]\label{axiom:normalization_rel}
$\phi(\ub)\in[0,1]$ for any $\ub\in\R_+^N$ and $\phi(\ub)=0$ if and only if $\ub=\alpha \one$ for some $\alpha\in\R_+$.
\end{axiom}

We make a few remarks in order. First, we consider non-negative $\ub$, which is common in many practical applications and consistent with prior studies (e.g., $\ub$ could be patient waiting time, distances from demand nodes to open facilities, income, etc; see \citealp{Ahmadi-Javid_et_al:2017, Chakravarty:1999, Marynissen_Demeulemeester:2019, Mussard_Mornet:2019, Pinedo:2016}). Second, Axiom~\ref{axiom:normalization_rel} is the relative counterpart of the normalization axiom discussed in Section~\ref{sec:fairness_measures}. Specifically, we require that the relative convex fairness $\rho$ is normalized such that $\rho(\ub)\in[0,1]$ for any $\ub\in\R^N_+$, which is consistent with the literature on relative fairness measures \citep{Donaldson_Weymark:1980, Mehran:1976}. In particular, with a suitable choice of $g$, we have that $\rho=0$ represents perfect equality while $\rho=1$ represents perfect inequality (see Proposition~\ref{prop:perfect_inequality_rel_CFM} and Corollary~\ref{coro:perfect_inequity}). Finally, since both $\nu$ and $g$ are continuous, $\rho$ is continuous on $\{\ub\in\R_+^N\mid g(\ub)\ne0\}$ by definition in \eqref{eqn:rel_CFM}. For example, if $g(\ub)=C \one^\tp\ub$ for some constant $C>0$, then $\rho(\ub)$ is continuous on $\R^N_+$ except at $\ub=\zero$ (see Theorem~\ref{thm:sufficient_rel_CFM}).

It is clear from the definition of $\rho$ in \eqref{eqn:rel_CFM} that one should carefully choose the normalization function $g$ such that $\rho$ satisfies the desired set of axioms (i.e.,  \ref{axiom:normalization_rel}, \ref{axiom:symmetry}, \ref{axiom:Schur_convex}, and \ref{axiom:scale_invariance}). Hence, we next derive conditions and provide guidelines for choosing the normalization function $g$.  First, in Theorem \ref{thm:necessary_rel_CFM}, we provide conditions on $g$ such that $\rho$ satisfies Axioms~\ref{axiom:normalization_rel}, \ref{axiom:symmetry}, and \ref{axiom:scale_invariance}.

\begin{theorem} \label{thm:necessary_rel_CFM}
Let $\nu:\R^N_+\rightarrow\R$ be a convex fairness measure and $g:\R^N_+\rightarrow\R$ be a continuous function. Then, the relative convex fairness measure $\rho=\nu/g$ satisfies Axioms~\ref{axiom:normalization_rel}, \ref{axiom:symmetry}, and \ref{axiom:scale_invariance} if and only if $g$ is symmetric and positive homogeneous with $g\geq \nu$.
\end{theorem}

Theorem~\ref{thm:necessary_rel_CFM} provides sufficient and necessary conditions on $g$ for which the relative convex fairness measure $\rho$ satisfies Axioms~\ref{axiom:normalization_rel}, \ref{axiom:symmetry}, and \ref{axiom:scale_invariance}. Specifically, $g$ belongs to the class of continuous and positive homogeneous functions with $g\geq\nu$. Note that although the convex fairness measure $\nu$ is Schur convex, the relative counterpart might not. Thus, in addition to Theorem~\ref{thm:necessary_rel_CFM}, we need to impose conditions on the normalization function $g$ to ensure that $\rho$ is Schur convex. Theorem~\ref{thm:sufficient_rel_CFM} provides a sufficient condition on $g$ such that $\rho$ is Schur convex, and hence, satisfies Definition~\ref{def:relative_convex_fairness_measure}.  

\begin{theorem} \label{thm:sufficient_rel_CFM}
 Let $\nu:\R^N_+\rightarrow\R$ be a convex fairness measure and $g:\R^N_+\rightarrow\R$ be a continuous function. If $g(\ub)=CN\ubar$ with $\ubar = N^{-1}\sum_{i=1}^N u_i$ for any $C\geq \sup_{\wb\in\calW_\nu}\norms{\wb_+}_\infty$, where $\wb_+=\max\{\wb,\zero\}$ with the maximum taken component-wisely, then the relative convex fairness measure $\rho=\nu/g$ satisfies Axioms~\ref{axiom:normalization_rel}, \ref{axiom:symmetry}, \ref{axiom:Schur_convex}, and \ref{axiom:scale_invariance}.
\end{theorem}

Let us provide an intuitive justification for choosing $g(\ub)=CN\ubar$ as suggested by Theorem~\ref{thm:sufficient_rel_CFM}. Let $\ub\in\R^N_+$ with $\one^\tp\ub=\gamma>0$, i.e., the total impact on the $N$ subjects is $\gamma$. Suppose that we can re-distribute the $\gamma$ units of impact among the subjects. Then, $\ub_\text{min}=\gamma/N\cdot\one$ minimizes $\nu$ with $\nu_\text{min}=\nu(\ub_\text{min})=0$, and $\ub_\text{max}=(0,\dots,0,\gamma)^\tp$ maximizes $\nu$ with $\nu_\text{max}=\nu(\ub_\text{max})=\gamma \sup_{\wb\in\calW_\nu} \norms{\wb_+}_\infty$. Thus, an intuitive way to normalize $\nu(\ub)$ such that $0 \leq \rho(\ub) \leq 1$ is as follows:
\begin{equation} \label{eqn:rel_CFM_linear_normalization_intuition}
\frac{\nu(\ub)-\nu_\text{min}}{\nu_\text{max}-\nu_\text{min}}=\frac{\nu(\ub)}{\gamma \sup_{\wb\in\calW_\nu} \norms{\wb_+}_\infty}= \frac{\nu(\ub)}{\one^\tp\ub  \cdot\sup_{\wb\in\calW_\nu} \norms{\wb_+}_\infty}= \frac{\nu(\ub)}{(N\sup_{\wb\in\calW_\nu} \norms{\wb_+}_\infty)\cdot\ubar }.    
\end{equation}
The denominator in \eqref{eqn:rel_CFM_linear_normalization_intuition} takes the same form of $g$ suggested in Theorem~\ref{thm:sufficient_rel_CFM}. Indeed, it is also common in the existing literature to normalize a fairness measure by a function of the mean $\ubar$ of $\ub$ \citep{Chakravarty:1999, Mussard_Mornet:2019, Zheng:2007}.  

Note that there are no general necessary conditions on $g$ for the relative convex fairness measure $\rho$ to be Schur convex. In Example~\ref{eg:SCV_rho}, we demonstrate that $\rho$ can be Schur convex even for some non-convex and non-concave  $g$.

\begin{example} \label{eg:SCV_rho} 
Consider the two-dimensional inequality measure $\nu(u_1,u_2)=|u_1-u_2|$ and the normalization function $g(u_1,u_2)=2\min\{|u_1-u_2|,(u_1+u_2)/2\}$. By definition, $g$ is symmetric and positive homogeneous with $g\geq \nu$, and hence, $\rho$ is symmetric and scale invariant by Theorem~\ref{thm:necessary_rel_CFM}. Note that $g$ is neither convex nor concave, which can be observed in Figure~\ref{fig:eg_Schur_convex_rel_measure} (the left plot). However, we can show that the relative convex fairness measure $\rho(u_1,u_2)=\nu(u_1,u_2)/g(u_1,u_2)$ is Schur convex. Indeed, it is straightforward to show that the level sets of $\rho$, denoted as $\calL_\beta=\{(u_1,u_2)\in\R^2_+\mid\rho(u_1,u_2)\leq \beta\}$ for all $\beta\geq 0$, are convex. Specifically,
$$\calL_\beta = \begin{cases}
\big\{(u_1,u_2)\in\R^2_+\mid u_1=u_2\big\}, &\text{if } \beta\in[0,1/2), \\
\displaystyle\Bigg\{(u_1,u_2)\in\R^2_+\,\,\Bigg|\,\, \frac{1-\beta}{1+\beta}u_1\leq u_2\leq \frac{1+\beta}{1-\beta} u_2\Bigg\}, &\text{if } \beta\in[1/2,1), \\
\R^2_+, &\text{if } \beta\in[1,\infty).
\end{cases}$$
In Figure~\ref{fig:eg_Schur_convex_rel_measure} (the right plot), we show the contours of $\rho$, where the black dotted line represents the level set with $\beta\in[0,1/2)$. This shows that $\rho$ is quasi-convex, which implies that $\rho$ is Schur convex (see Chapter 3 of \citealp{Marshall_et_al:2011}).
\begin{figure}
    \centering 
    \includegraphics[scale=0.75]{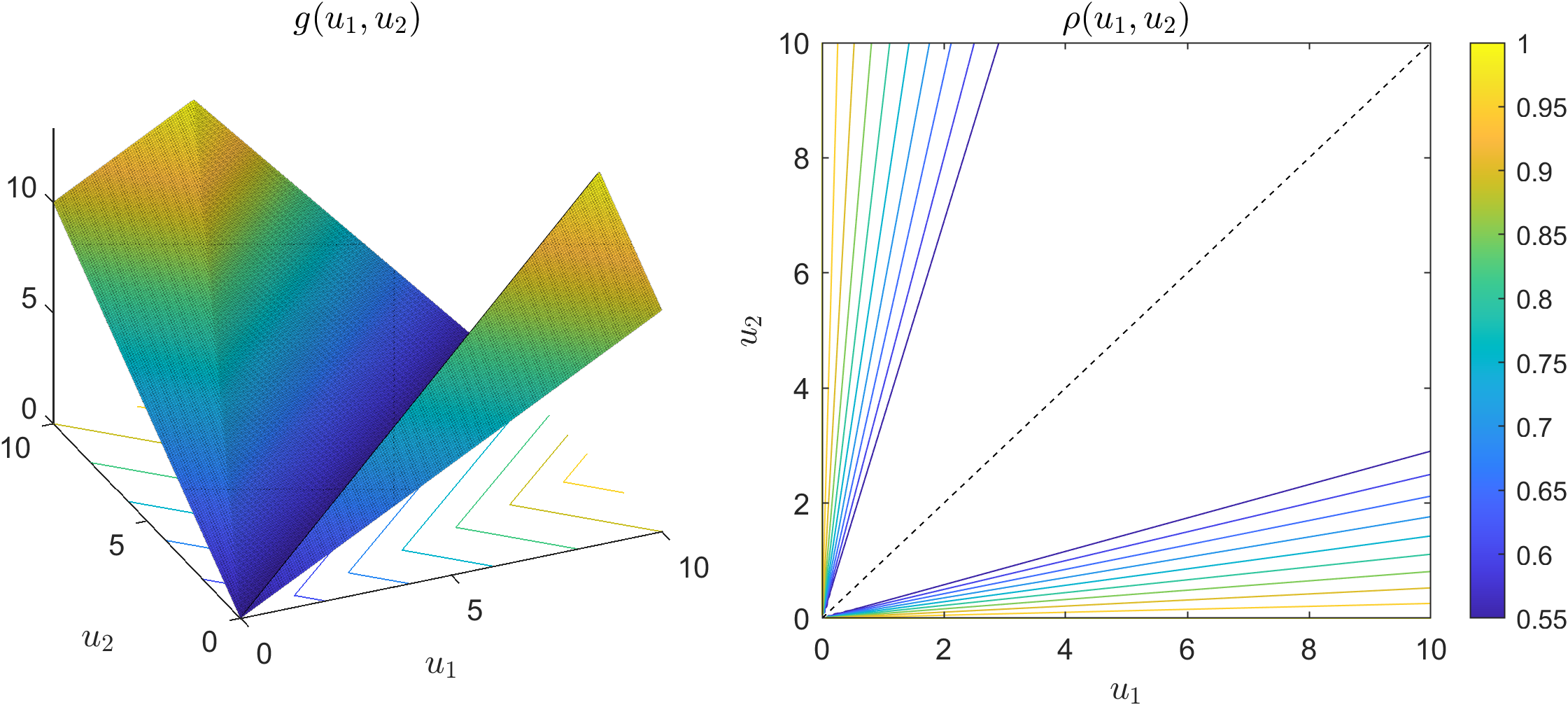}
    \caption{Example of the normalization function $g(u_1,u_2)=2\min\{|u_1-u_2|,(u_1+u_2)/2\}$ and the relative convex fairness measure $\rho(u_1,u_2)=|u_1-u_2|/g(u_1,u_2)$.}
    \label{fig:eg_Schur_convex_rel_measure}
\end{figure}
\end{example}

Next, in Theorem~\ref{thm:nceessary_sufficient_rel_CFM_mean}, we show that if we restrict the normalization function $g$ to be a function of the mean $\ubar$, then $g(\ub)$ can only take the following form $g(\ub)=CN\ubar$, i.e., a necessary condition on $g$ for $\rho$ to be a relative convex fairness measure satisfying Definition~\ref{def:relative_convex_fairness_measure}. 

\begin{theorem}\label{thm:nceessary_sufficient_rel_CFM_mean}
Let $\nu:\R^N_+\rightarrow\R$ be a convex fairness measure and $g:\R^N_+\rightarrow\R$ be a continuous function. Suppose that $g(\ub)=\gt(\ubar)$ is a function of mean $\ubar = N^{-1}\sum_{i=1}^N u_i$ of $\ub$. Then, the relative convex fairness measure $\rho=\nu/g$ satisfies Axioms~\ref{axiom:normalization_rel}, \ref{axiom:symmetry}, \ref{axiom:Schur_convex}, and \ref{axiom:scale_invariance} if and only if $\gt(\ubar)=CN\ubar$ for any $C\geq \sup_{\wb\in\calW_\nu}\norms{\wb_+}_\infty$.
\end{theorem}

Theorem \ref{thm:nceessary_sufficient_rel_CFM_mean} shows that if the normalization function $g$ is a function of the mean $\ubar$, then a linear function in $\ubar$ is the only possible choice such that $\rho$ satisfies Axioms~\ref{axiom:normalization_rel}, \ref{axiom:symmetry}, \ref{axiom:Schur_convex}, and \ref{axiom:scale_invariance}. We can also see from \eqref{eqn:rel_CFM_linear_normalization_intuition} that $g(\ub)=CN\ubar$ with $C=\sup_{\wb\in\calW_\nu}\norms{\wb_+}_\infty$ is an intuitive choice. 

Observe from \eqref{eqn:rel_CFM_linear_normalization_intuition}  that $\rho(\ub)=0$ when the convex fairness measure attains its minimum value of zero, i.e., $\nu(\ub)=\nu(\ub_{\text{min}})=0$, where $\ub_\text{min}\in\argmin\{\nu(\ub)\mid \one^\tp\ub=\gamma\}$ for any $\gamma>0$. Thus, like other relative fairness measures in the literature, $\rho(\ub)=0$ implies perfect equality (Axiom~\ref{axiom:normalization_rel}). On the other hand, $\rho(\ub)$ attains its maximum value of one when the convex fairness measure attains its maximum value, i.e., $\nu(\ub)=\nu(\ub_{\text{max}})$, where $\ub_\text{max}\in\argmax\{\nu(\ub)\mid \one^\tp\ub=\gamma\}$ for any $\gamma>0$. It follows that $\rho(\ub)=1$ indicates perfect inequality. We formalize the latter observation in Proposition~\ref{prop:perfect_inequality_rel_CFM}.

\begin{proposition}  \label{prop:perfect_inequality_rel_CFM}
Let $\nu$ be a convex fairness measure and $g(\ub)=\sup_{\wb\in\calW_\nu}\norms{\wb_+}_\infty N\ubar$ with $\ubar = N^{-1}\sum_{i=1}^N u_i$. Then, the relative convex fairness measure $\rho(\ub)=\nu(\ub)/g(\ub)=1$ if and only if $\ub\in\calU^\star=\bigcup_{\gamma>0} \calU^\star_\gamma$, where $\calU^\star_\gamma=\argmax_{\ub'\in\R^N_+}\{\nu(\ub')\mid \one^\tp\ub'=\gamma\}$.
\end{proposition}

In Corollary~\ref{coro:perfect_inequity}, we  show that under some additional mild assumption on $\nu$, the relative convex fairness measure achieves its maximum $\rho(\ub)=1$ if and only if $\ub$ takes the form $P(0,\dots,0,\gamma)^\tp\in\R^N$ for any $\gamma>0$ and permutation matrix $P$. 

\begin{corollary} \label{coro:perfect_inequity}
Let $\nu$ be a convex fairness measure and $g(\ub)=\sup_{\wb\in\calW_\nu}\norms{\wb_+}_\infty N\ubar$ with $\ubar = N^{-1}\sum_{i=1}^N u_i$. Assume that $\nu(0,\dots,0,0,\gamma)>\nu(0,\dots,0,\gamma/2,\gamma/2)$ for any $\gamma>0$. Then, $\rho(\ub)=\nu(\ub)/g(\ub)=1$ if and only if $\ub=P(0,\dots,0,\gamma)^\tp$ for any $\gamma>0$ and permutation matrix $P$.
\end{corollary}

\begin{remark}
Note that the constant $w_\text{max}=\sup_{\wb\in\calW_\nu}\norms{\wb_+}_\infty$ can be computed directly without solving the supremum problem. Specifically, since $\nu(\ub)=\sup_{\wb\in\calW_\nu}\nu_{\wb}(\ub)$ (see Theorem~\ref{thm:convex_FM_dual}), we have $\nu(0,\dots,0,\gamma)=\gamma w_\text{max}$, which implies $w_\text{max} = \nu(0,\dots,0,\gamma)/\gamma$ for any $\gamma>0$.     
\end{remark}

We close this section with Example \ref{eg:rel_CFM}, where we derive the relative counterparts of the deviation-based fairness measures in Table~\ref{table:existing_FM} and make connections to existing relative fairness measures. 

\begin{example}[Relative counterparts of the fairness measures in Table~\ref{table:existing_FM}] \label{eg:rel_CFM}
\phantom{a}

\begin{enumerate}[topsep=1pt,itemsep=-5pt]
    \item [(i)] Recall that measure (i) is order-based with $\wb=(-1,0,\dots,0,1)$. Thus, we have $w_\text{max}=1$, and hence,
    $$\rho(\ub)=\frac{\nu(\ub)}{N\ubar}=\frac{u_{(N)}-u_{(1)}}{N\ubar},$$
    which is equivalent to the relative range \citep{Cowell:2011}. Note that (iii) is equivalent to (i) (see \ref{appdx:eg_inequity_measures}). Thus, the relative counterpart of (iii) takes the same form as (i).
    
    \item[(ii)] Recall that measure (ii) is order-based with $w_i=2(2i-N-1)$ for $i\in[N]$. Thus, we have $w_\text{max}=2(N-1)$, and hence,
    $$\rho(\ub)=\frac{\nu(\ub)}{2N(N-1)\ubar}=\frac{\sum_{i\in [N]}\sum_{j\in[N]} |u_i-u_j|}{2N(N-1)\ubar},$$
    which is equivalent to the Gini index \citep{Gini:1912}.
    
    \item[(iv)] For measure (iv), we have $w_\text{max}=\nu(0,\dots,0,\gamma)/\gamma=2[1-(1/N)]$, and hence,
    $$\rho(\ub)=\frac{\nu(\ub)}{2(N-1)\ubar}=\frac{\sum_{i\in [N]}|u_i-\ubar|}{2(N-1)\ubar},$$
    which is equivalent to the relative mean absolute deviation, also known as the Hoover index \citep{Hoover:1936}.

    \item[(v)] For measure (v), we have $w_\text{max}=\nu(0,\dots,0,\gamma)/\gamma=\sqrt{1-(1/N)}$, and hence,
    $$\rho(\ub)=\frac{\nu(\ub)}{\sqrt{N}\sqrt{N-1}\ubar}=\frac{\sqrt{(N-1)^{-1}\sum_{i\in[N]} (u_i-\ubar)^2}}{\sqrt{N}\ubar},$$
    which is equivalent to the relative standard deviation or the coefficient of variation \citep{Cowell:2011}.

    \item[(vi)] For measure (vi), we have $w_\text{max}=\nu(0,\dots,0,\gamma)/\gamma=1-(1/N)$, and hence,
    $$\rho(\ub)=\frac{\nu(\ub)}{(N-1)\ubar}=\frac{\max_{i\in [N]}|u_i-\ubar|}{(N-1)\ubar}.$$
    Note that measure (vii) is equal to $N$ times measure (vi) (see \ref{appdx:eg_inequity_measures}). Thus, the relative counterpart of (vii) takes the same form as (vi).

    \item[(viii)]  For measure (viii), we have $w_\text{max}=\nu(0,\dots,0,\gamma)/\gamma=N$, and hence,
    $$\rho(\ub)=\frac{\nu(\ub)}{N^2\ubar}=\frac{\sum_{i\in[N]} \max_{j\in[N]} |u_i-u_j|}{N^2\ubar}.$$
    
\end{enumerate}
\end{example}


\section{Mathematical Proofs} \label{appdx:proofs}

In this section, we present proofs of the theoretical results in the order of their appearance.

\subsection{Proof of Proposition \ref{prop:order_based_FM_is_FM}}

\begin{itemize}
    \item[(a)] \textit{Axiom~\ref{axiom:continuity}}. First, we claim that the sorting operator $S:\R^N\rightarrow\R^N$ that maps a vector $\ub$ to $\ub^\uparrow\in\Rup^N$ with entries in ascending order is continuous. Consider two vectors $\ub^1$ and $\ub^2$ with $\norms{\ub^1-\ub^2}_\infty =\varepsilon$, i.e., $|u_i^1-u_i^2|\leq\varepsilon$ for all $i\in[N]$. We claim that $|u^1_{(i)}-u^2_{(i)}|\leq \varepsilon$ for all $i\in[N]$. To show this, suppose, on the contrary, that $|u^1_{(i)}-u^2_{(i)}| > \varepsilon$ for some $i\in[N]$. Consider the following two cases. First, if $u^2_{(i)}<u^1_{(i)}-\varepsilon$, define $\calJ^-_i=\{\pi_2(1),\dots,\pi_2(i)\}$, where $\pi_2(k)$ is the index such that $\ub^2_{\pi_2(k)}=\ub^2_{(k)}$ for $k\in[N]$. For all $j\in\calJ^-_i$, we have
    \begin{equation}  \label{ref:pf_prop_order_based_FM_is_FM_1}
        u^1_j\leq u^2_j+\varepsilon\leq u^2_{(i)}+\varepsilon<u^1_{(i)},
    \end{equation}
    where the first inequality follows from $\varepsilon=\norms{\ub^1-\ub^2}_\infty$, and the second inequality follows from $j\in\calJ^-_i$. This contradicts that $u^1_{(i)}$ is the $i$th smallest entry in $\ub^1$. Similarly, if $u^2_{(i)} > u^1_{(i)}+\varepsilon$, define $\calJ^+_i=\{\pi_2(i),\dots,\pi_2(N)\}$. Then, following a similar argument in \eqref{ref:pf_prop_order_based_FM_is_FM_1}, for all $j\in\calJ^+_i$, we have
    $$u^1_j\geq u^2_j-\varepsilon\geq u^2_{(i)}-\varepsilon>u^1_{(i)},$$
    which leads to the contradiction that $u^1_{(i)}$ is the $i$th smallest entry in $\ub^1$. Therefore,
    $$  |\nu_{\wb}(\ub^1)-\nu_{\wb}(\ub^2)| = \Bigg|\sum_{i=1}^N w_i u^1_{(i)} -\sum_{i=1}^N w_i u^2_{(i)}\Bigg| \leq\sum_{i=1}^N |w_i|\, |u^1_{(i)} - u^2_{(i)}| \leq \norms{\wb}_1 \norms{\ub^1-\ub^2}_\infty. $$
    %
    This shows that $\nu_{\wb}$ is continuous.
    
    \item[(b)] \textit{Axiom~\ref{axiom:normalization}}. We first show that $\nu_{\wb}(\ub)\geq 0$ for any $\ub\in\R^N$. Letting $u_{(0)}=0$, we can write
    $$\nu_{\wb}(\ub)=\sum_{i=1}^N w_i u_{(i)}=\sum_{i=1}^N w_i \sum_{j=1}^i \big[ u_{(j)}-u_{(j-1)} \big]=\sum_{j=1}^N \Bigg(\sum_{i=j}^N w_i \Bigg) \big[ u_{(j)}-u_{(j-1)} \big].$$
    Since $\sum_{i=1}^N w_i = 0$, we have
    \begin{equation} \label{pf_eqn:order_based_FM_is_FM_1}
     \nu_{\wb}(\ub)= \sum_{j=2}^N \Bigg(\sum_{i=j}^N w_i \Bigg) \big[ u_{(j)}-u_{(j-1)} \big].
    \end{equation}
    Since $u_{(j)}-u_{(j-1)}\geq 0$, to show that $\nu_{\wb}(\ub)\geq0$, it suffices to show that $\sum_{i=j}^N w_i > 0$ for all $j\in[2,N]_\Z$. We show $\sum_{i=j}^N w_i > 0$ by induction. When $j=2$, since $\sum_{i=1}^N w_i = 0$ and $w_1< 0$, we have $\sum_{i=2}^N w_i = \sum_{i=1}^N w_i - w_1 > 0$. Next, suppose that $\sum_{i=j-1}^N w_i > 0$. If $w_j \geq 0$, then it is trivial that $\sum_{i=j}^N w_j > 0$ since $0\leq w_j\leq \dots\leq w_N$ and $w_N>0$. If $w_j \leq 0$, then $\sum_{i=j}^N w_i = \sum_{i=j-1}^N w_i - w_{j-1} > 0$ by induction hypothesis. This completes the induction step and shows that $\nu_{\wb}(\ub)\geq 0$. Finally, we show that $\nu_{\wb}(\ub)=0$ if and only if $\ub=\alpha\one$ for some $\alpha\in\R$. If $\ub=\alpha\one$, then it trivial that $\nu_{\wb}(\ub)=\alpha \sum_{i=1}^N w_i=0 $. If $\nu_{\wb}(\ub)=0$, by \eqref{pf_eqn:order_based_FM_is_FM_1} and $\sum_{i=j}^N w_i > 0$, we must have $u_{(j)}-u_{(j-1)}=0$ for all $j\in[2,N]_\Z$, which in turn implies $u_{(1)}=\dots=u_{(N)}$.
    
    \item[(c)] \textit{Axiom~\ref{axiom:symmetry}}. Symmetry follows directly from the definition of $\nu_{\wb}$ that depends only on the order of $\ub$.
    
    \item[(d)] \textit{Axiom~\ref{axiom:Schur_convex}}. Letting $w_0=0$, we can write
    \begin{align*}
        \nu_{\wb}(\ub)&=\sum_{i=1}^N w_i u_{(i)}=\sum_{i=1}^N \Bigg[\sum_{j=1}^i (w_j-w_{j-1})\Bigg] u_{(i)} = \sum_{j=1}^N (w_j-w_{j-1}) \Bigg[\sum_{i=j}^N u_{(i)}\Bigg] \\
        &= w_1  \Bigg[\sum_{i=1}^N u_{(i)}\Bigg] + \sum_{j=2}^N (w_j-w_{j-1}) \Bigg[\sum_{i=j}^N u_{(i)}\Bigg].
    \end{align*}
     By definition, $\ub^1 \preceq \ub^2$ implies $\sum_{i=1}^N u_i^1=\sum_{i=1}^N u^2_i$ and $\sum_{i=j}^N u^1_i \leq \sum_{i=j}^N u^2_i$ for all $j\in[N]$. Since $w_j-w_{j-1}\geq 0$, we have $\nu_{\wb}(\ub^1)\leq \nu_{\wb}(\ub^2)$.
     
    \item[(e)] \textit{Axiom~\ref{axiom:trans_invariance}}. It is straightforward to verify that
    $$\nu_{\wb}(\ub+\alpha\one)=\sum_{i=1}^N w_i \big[u_{(i)}+\alpha \big] = \sum_{i=1}^N w_i u_{(i)} + \alpha  \sum_{i=1}^N w_i = \nu_{\wb}(\ub). $$
    
    \item[(f)] \textit{Axiom~\ref{axiom:pos_homo}}. It is straightforward to verify that
    $$\nu_{\wb}(\alpha\ub)=\sum_{i=1}^N w_i \big[\alpha u_{(i)}\big] = \alpha \sum_{i=1}^N w_i u_{(i)} = \alpha \nu_{\wb}(\ub). $$

\end{itemize}
This completes the proof. 

\subsection{Proof of Theorem \ref{thm:order_based_FM_characterization_I}}
From the rearrangement inequality \citep{Marshall_et_al:2011}, for any $\pi\in\Pi$, we have 
$$\sum_{i=1}^N w_{\pi(i)} u_i \leq \sum_{i=1}^N w_{(i)} u_{(i)} = \sum_{i=1}^N w_i u_{(i)}= \nu_{\wb}(\ub),$$
where the first equality follows from $w_1\leq\dots\leq w_N$. 

\subsection{Proof of Theorem \ref{thm:axiom_order_based}}
Suppose that $\nu$ is an order-based fairness measure with weight $\wb$. Then, $\nu$ satisfies Axioms~\ref{axiom:normalization}, \ref{axiom:Schur_convex}, and \ref{axiom:pos_homo} by Proposition~\ref{prop:order_based_FM_is_FM}. In addition, for any $\ub\in\Rup^N$, $\varepsilon\in[0,u_{j+1}-u_j]$, and $j\in[N]$, since the order of the entries in $\ub+\varepsilon\eb_j$ are the same as that of $\ub$, we also have $\nu(\ub+\varepsilon\eb_j)=\nu(\ub)+\varepsilon w_j$ by definition of the order-based fairness measure. If follows that $\nu$ satisfies Axiom~\ref{axiom:order_based}.

Now, suppose that $\nu$ satisfies Axioms~\ref{axiom:normalization}, \ref{axiom:Schur_convex}, \ref{axiom:pos_homo}, and \ref{axiom:order_based}. Recall that Axiom~\ref{axiom:Schur_convex} implies Axiom~\ref{axiom:symmetry}, i.e., $\nu$ is symmetric. Thus, without loss of generality, we focus on the function $\nu$ in the space $\Rup^N$. We show that $\nu$ is an order-based fairness measure in two steps. First, we show that for any $\ub\in\Rup^N$,
\begin{equation}\label{eqn_pf:thm_axiom_order_based_1}
    \nu(\ub+\varepsilon\ubbar^j)=\nu(\ub)+\varepsilon\nu(\ubbar^j),
\end{equation}
for all $j\in[N]$ and $\varepsilon\geq 0$, where $\ubbar^j = \sum_{i=j}^N \eb_i$. To show \eqref{eqn_pf:thm_axiom_order_based_1}, applying Axiom \ref{axiom:order_based} iteratively, we obtain that for all $j\in[N]$ and $\ub\in\Rup^N$,
\allowdisplaybreaks
\begin{align}
    \nu(\ub + \varepsilon\ubbar^j)&=\nu\Big(\big[\ub + \varepsilon\ubbar^{j+1}\big]+\varepsilon\eb_j\Big) \nonumber \\
    &= \nu\Big(\ub + \varepsilon\ubbar^{j+1}\Big) + \varepsilon w_j \nonumber \\
    &= \nu\Big(\big[\ub + \varepsilon\ubbar^{j+2}\big]+\varepsilon\eb_{j+1}\Big) + \varepsilon w_j  \nonumber \\
    &= \nu\Big(\ub + \varepsilon\ubbar^{j+2}\Big) + \varepsilon (w_j + w_{j+1})  \nonumber\\
    &= \cdots  \nonumber\\
    &= \nu(\ub) + \varepsilon \sum_{i=j}^N w_j.\label{eqn_pf:thm_axiom_order_based_2}
\end{align}
In particular, if we set $\ub=\zero$ in \eqref{eqn_pf:thm_axiom_order_based_2}, this gives $\nu(\varepsilon \ubbar^j)=\varepsilon\sum_{i=j}^N w_j$ for all $j\in[N]$, which implies that $\nu(\ubbar^j)=\sum_{i=j}^N w_j$ for all $j\in[N]$ by Axiom~\ref{axiom:pos_homo}. Replacing $\sum_{i=j}^N w_j$ in \eqref{eqn_pf:thm_axiom_order_based_2} by $\nu(\varepsilon \ubbar^j)$, we obtain the desired equality in \eqref{eqn_pf:thm_axiom_order_based_1}.

Second, using \eqref{eqn_pf:thm_axiom_order_based_1}, we show that Axioms~\ref{axiom:normalization}, \ref{axiom:Schur_convex}, \ref{axiom:pos_homo}, and \ref{axiom:order_based} characterize order-based fairness measures. For any $\ub\in\Rup^N$, we can write $\ub$ as a linear combination of $\{\ubbar^j\}_{j=1}^N$:
$$\ub=u_1\cdot\ubbar^1 + (u_2-u_1)\cdot\ubbar^2+\cdots+(u_N-u_{N-1})\cdot\ubbar^N=u_1\cdot\ubbar^1 +\sum_{j=2}^N \Big[(u_j-u_{j-1})\cdot\ubbar^j\Big].$$
Applying \eqref{eqn_pf:thm_axiom_order_based_1} iteratively, we have $\nu(\ub)=\nu(u_1\cdot\ubbar^1)+\sum_{j=2}^N (u_j-u_{j-1})\nu(\ubbar^j)$. Note that $\nu(u_1\cdot\ubbar^1)=\nu(u_1\cdot\one)=0$ by Axiom~\ref{axiom:normalization}. Thus, we have $\nu(\ub)=\sum_{j=2}^N (u_j-u_{j-1})\nu(\ubbar^j)$. It follows that we can rewrite the function $\nu$ as $\nu(\ub)=\sum_{j=1}^N w'_ju_j$, where $w'_j=\nu(\ubbar^j)-\nu(\ubbar^{j+1})$ for $j\in[N-1]$ and $w'_N=\nu(\ubbar^N)$. Hence, to show that $\nu$ is an order-based fairness measure, it suffices to show that $\wb'=(w'_1,\dots,w'_N)^\tp\in\calW=\{\wb\in\R^N\mid \sum_{i=1}^N w_i = 0,\, w_1\leq\dots\leq w_N,\, w_1< 0,\, w_N> 0\}$.
\begin{enumerate}
    \item[(a)] First, we show that $\sum_{j=1}^N w'_j =0$. From the definition of $w_j'$, we have $\sum_{j=1}^N w'_j = \nu(\ubbar^N) + \sum_{j=1}^{N-1} [\nu(\ubbar^j)-\nu(\ubbar^{j+1})] = \nu(\ubbar^1)=\nu(\one)=0$ by Axiom~\ref{axiom:normalization}.
    \item[(b)] Second, we show that $w'_j\geq w'_{j-1}$ for all $j\in[2,N]$. Indeed, note that $(\ubbar^{j-1} +\ubbar^{j+1})/2 \preceq \ubbar^j$ since the vector $(\ubbar^{j-1} +\ubbar^{j+1})/2$ is obtained by transferring $1/2$ from the $j$th to the $(j-1)$th individual in $\ubbar^j$, where we let $\ubbar^{N+1}=\zero$. Axiom~\ref{axiom:Schur_convex} implies that $\nu\big((\ubbar^{j-1} +\ubbar^{j+1})/2\big) \leq \nu(\ubbar^j)$. Also, by \eqref{eqn_pf:thm_axiom_order_based_1}, we have $\nu\big((\ubbar^{j-1} +\ubbar^{j+1})/2\big)=\big[\nu(\ubbar^{j-1}) +\nu(\ubbar^{j+1})\big]/2$. Therefore, we obtain $\nu(\ubbar^{j-1})-\nu(\ubbar^j) \leq \nu(\ubbar^j)-\nu(\ubbar^{j+1})$, and thus, $w'_j\geq w'_{j-1}$.
    \item[(c)] Third, be definition, we have $w'_1=\nu(\ubbar^1)-\nu(\ubbar^2)=-\nu(\ubbar^2)<0$, where the inequality follows from Axiom~\ref{axiom:normalization}. Together with (a) and (b), we must have $w'_N>0$.
\end{enumerate}
Therefore, we have $\nu(\ub)=\sum_{j=1}^N w'_ju_j$ is an order-based fairness measure. Finally, note that we have $\nu(\ubbar^j)=\sum_{i=j}^N w_j$ for all $j\in[N]$ from the first step of the proof, which implies $\wb'=\wb$. This completes the proof. 

\subsection{Proof of Theorem \ref{thm:convex_FM_dual}}
It is easy to verify that (b) implies (c). We first prove that (c) implies (a). Suppose there exists a compact set $\calW_\nu\subseteq\{\wb\in\Rup^N\mid \one^\tp\wb=0\}$ such that $\nu(\ub)=\sup_{\wb\in\calW_\nu} \nu_{\wb}(\ub)$. We want to verify that $\nu(\ub)$ is an absolute convex fairness measure satisfying Axioms~\ref{axiom:continuity}, \ref{axiom:normalization}, \ref{axiom:symmetry}, \ref{axiom:trans_invariance}, \ref{axiom:pos_homo}, and \ref{axiom:convexity} (Definition~\ref{def:absolute_convex_fairness_measures}).

\begin{enumerate}
    \item [(I)] \textit{Axiom~\ref{axiom:continuity}}. From the proof of Proposition \ref{prop:order_based_FM_is_FM}, for any $\{\ub^1,\ub^2\}\subset\R^N$, we have $|\nu_{\wb}(\ub^1)-\nu_{\wb}(\ub^2)|\leq \norms{\wb}_1\norms{\ub^1-\ub^2}_\infty$. Hence,
    \begin{align*}
        |\nu(\ub^1)-\nu(\ub^2)| &= \bigg| \sup_{\wb\in\calW} \nu_{\wb}(\ub^1) - \sup_{\wb\in\calW} \nu_{\wb}(\ub^2) \bigg| \\
        &\leq \sup_{\wb\in\calW} |\nu_{\wb}(\ub^1)-\nu_{\wb}(\ub^2)| \leq \sup_{\wb\in\calW} \norms{\wb}_1\cdot \norms{\ub^1-\ub^2}_\infty,
    \end{align*}
    where $\sup_{\wb\in\calW} \norms{\wb}_1<\infty$ since $\calW$ is compact. Hence, $\nu$ is continuous in $\ub$.    
    
    \item [(II)] \textit{Axiom~\ref{axiom:normalization}}. Since $\nu_{\wb}(\ub)\geq 0$ for any $\wb\in\calW$ by Proposition \ref{prop:order_based_FM_is_FM}, we have $\nu(\ub)\geq 0$. Moreover, note that $\nu(\ub)=0$ is equivalent to $\nu_{\wb}(\ub)=0$ for all $\wb\in\calW$. By Proposition \ref{prop:order_based_FM_is_FM}, $\nu_{\wb}(\ub)=0$ if and only if $\ub=\alpha\one$ for some $\alpha\in\R$. 
    
    \item [(III)] \textit{Axiom~\ref{axiom:symmetry}}. Symmetry holds since each $\nu_{\wb}(\ub)$ is symmetric by Proposition \ref{prop:order_based_FM_is_FM}.

    \item [(IV)] \textit{Axiom~\ref{axiom:convexity}}. By Theorem \ref{thm:order_based_FM_characterization_I}, $\nu_{\wb}$ is a maximum of linear functions. It follows that $\nu_{\wb}$ is convex. Hence, since $\nu$ is a supremum of convex functions, $\nu$ is also convex.
        
    \item [(V)] \textit{Axiom~\ref{axiom:trans_invariance}}. By Proposition \ref{prop:order_based_FM_is_FM}, for any $\alpha\in\R$,
    $$\nu(\ub+\alpha\one) = \sup_{\wb\in\calW} \nu_{\wb}(\ub+\alpha\one) =  \sup_{\wb\in\calW} \nu_{\wb}(\ub) = \nu(\ub).$$
    
    \item [(VI)] \textit{Axiom~\ref{axiom:pos_homo}}. By Proposition \ref{prop:order_based_FM_is_FM}, for any $\alpha>0$,
    $$\nu(\alpha\ub) = \sup_{\wb\in\calW} \nu_{\wb}(\alpha\ub) =  \alpha \sup_{\wb\in\calW} \nu_{\wb}(\ub) = \alpha\nu(\ub).$$

\end{enumerate}

Next, we prove that (a) implies (b). Note that by definition of absolute convex fairness measures, $\nu$ is proper, continuous, and convex. By Fenchel–Moreau theorem, $\nu(\ub)=\sup_{\wb\in\R^N} \big\{\ub^\tp\wb - \nu^*(\wb)\big\}$, where $\nu^*(\wb)=\sup_{\ub\in\R^N} \{\ub^\tp\wb-\nu(\ub)\}$ is the convex conjugate of $\nu$ \citep{Bertsekas:2009}. We divide the proof into the following four steps.

\begin{itemize}
    \item \textit{Step 1.} By Axiom~\ref{axiom:pos_homo}, we have for any $\alpha > 0$,
    \begin{align*}
        \nu^*(\wb)&=\sup_{\ub\in\R^N}\big\{\ub^\tp\wb-\nu(\ub)\big\} \\
        &= \alpha \sup_{\ub\in\R^N} \Bigg\{ \bigg(\frac{\ub}{\alpha}\bigg)^\tp \wb - \nu\bigg(\frac{\ub}{\alpha}\bigg) \Bigg\} = \alpha \sup_{\ub'\in\R^N} \big\{ (\ub')^\tp \wb - \nu(\ub')\big\} = \alpha \nu^*(\wb),
    \end{align*}
    where we apply a change of variable $\ub'=\ub/\alpha$. Thus, $\nu^*(\wb)$ equals $0$ if $\wb\in\dom(\nu^*)$, and $\infty$ otherwise. Note that $\dom(\nu^*)= \{\wb\in\R^N\mid\nu^*(\wb)\leq 0\} = \big\{\wb\in\R^N\mid \ub^\tp\wb-\nu(\ub)\leq 0,\,\forall\ub\in\R^N\big\} = \partial \nu(\zero)$, where $\nu(\zero)$ is the subdifferential of $\nu$ at $\ub=\zero$. Therefore, the set $\dom(\nu^*)$ is closed, bounded, and convex (see Proposition 5.4.2 of \citealp{Bertsekas:2009}). Thus, we have     $\nu(\ub)=\sup_{\wb\in\dom(\nu^*)} \big\{\ub^\tp\wb\big\}$.
    
    \item \textit{Step 2.} By Axiom~\ref{axiom:trans_invariance}, we have for any $\ubt\in\R^N$,
    \begin{align*}
        \nu^*(\wb)=\sup_{\ub\in\R^N}\big\{\ub^\tp\wb-\nu(\ub)\big\} &\geq \sup_{\alpha\in\R} \big\{ (\ubt+\alpha\one)^\tp \wb - \nu(\ubt+\alpha\one)\big\} \\
        &= \sup_{\alpha\in\R} \big\{\ubt^\tp\wb - \nu(\ubt) + \alpha\one^\tp\wb\big\}.
    \end{align*}
    Therefore, if $\wb\in\dom(\nu^*)$, we must have $\one^\tp\wb=0$. 
    
    \item \textit{Step 3.} By Axiom~\ref{axiom:symmetry}, for any permutation matrix $P\in\calP$ (the set of all $N\times N$ permutation matrices),
    \begin{align*}
        \nu^*(P\wb)&=\sup_{\ub\in\R^N}\big\{ \ub^\tp(P\wb)-\nu(\ub)\big\} = \sup_{\ub\in\R^N}\big\{ (P^\tp\ub)^\tp\wb - \nu(\ub)\big\} \\
        &= \sup_{\ub'\in\R^N} \big\{(\ub')^\tp\wb-\nu(P\ub')\big\} = \sup_{\ub'\in\R^N}\big\{(\ub')^\tp\wb-\nu(\ub')\big\}=\nu^*(\wb),
    \end{align*}
    where we apply a change of variable $\ub'=P^\tp\ub$. Thus, $\nu^*$ is also symmetric. That is, if $\wb\in\dom(\nu^*)$, then $P\wb\in\dom(\nu^*)$ for any $P\in\calP$. As a result, $\dom(\nu^*)=\bigcup_{P\in\calP} \{\wbt\in\R^N\mid \wbt=P\wb,\, \wb\in\calW_\nu\}$ with $\calW_\nu=\dom(\nu^*)\cap\Rup^N$ still being compact and convex. Hence, we can write
    \begin{align*}
        \nu(\ub)=\sup_{\wb\in\dom(\nu^*)} \big\{\ub^\tp\wb\}=\sup_{\wb\in\calW_\nu} \sup_{P\in\calP} \big\{ \ub^\tp(P\wb)\big\} = \sup_{\wb\in\calW_\nu} \nu_{\wb}(\ub),
    \end{align*}
    where the last equality follows from Theorem \ref{thm:order_based_FM_characterization_I}. 
    
    \item \textit{Step 4.} We have shown that if $\nu$ satisfies Axioms~\ref{axiom:continuity}, \ref{axiom:symmetry}, \ref{axiom:convexity}, \ref{axiom:trans_invariance}, \ref{axiom:pos_homo}, then $\nu(\ub)=\sup_{\wb\in\calW_\nu} \nu_{\wb}(\ub)$. Therefore, we immediately have $\nu(\ub)\geq0$ since $\nu_{\wb}(\ub)\geq0$. Moreover, by Axiom~\ref{axiom:normalization}, since $\nu$ equals zero if and only if $\ub=\alpha\one$ for some $\alpha\in\R$. This implies that $\calW_\nu\ne\{\zero\}$ (otherwise, $\nu(\ub)=0$ for any $\ub\in\R^N$).
\end{itemize}

To conclude, if $\nu$ is an absolute convex fairness measure, then $\nu(\ub)=\sup_{\wb\in\calW_\nu} \nu_{\wb}(\ub)$, where $\calW_\nu=\dom(\nu^*)\cap\Rup^N\subseteq\{\wb\in\R^N\mid \one^\tp\wb=0\}$ and $\calW_\nu\ne\{\zero\}$. This completes the proof.  

\subsection{Proof of Proposition \ref{prop:covnex_FM_dual_set_ex_pt}}
Since $\calW_\nu$ is compact and convex, we have $\wb\in\conv(\calE(\calW_\nu))$. For any $\wb\in\conv(\calE(\calW_\nu))$, we have $\wb=\sum_{j=1}^{K} \alpha_j\wbt^j$ for some $K > 0$, $\wbt^j\in\calE(\calW_\nu)$, and $\alpha_j\geq 0$ with $\sum_{j=1}^K \alpha_j = 1$. Then,
$$\nu_{\wb}(\ub)=\sum_{i=1}^N w_i u_{(i)}=\sum_{i=1}^N \Bigg(\sum_{j=1}^K \alpha_j \wt^j_i\Bigg) u_{(i)}=\sum_{j=1}^K \alpha_j \Bigg(\sum_{i=1}^N \wt^j_i u_{(i)} \Bigg)=\sum_{j=1}^K \alpha_j \nu_{\wbt^j}(\ub).$$
That is, $\nu_{\wb}(\ub)$ is a convex combination of $\nu_{\wbt^j}(\ub)$ for $j\in[K]$. As a result, we have $\nu_{\wb}(\ub)\leq\max_{j\in[K]} \nu_{\wbt^j}(\ub)$. Hence, it suffices to consider the supremum over the set of extreme points $\calE(\calW_\nu)$. Finally, since $\nu_{\zero}\equiv 0$ is always dominated by $\nu_{\wb}$ for any $\wb\ne\zero$, we can consider the supremum over the set $\calE(\calW_\nu)\setminus\{\zero\}$. This completes the proof.

\subsection{Proof of Proposition \ref{prop:existing_FM_dual_set}}
\begin{itemize}
\item[(a)] From Proposition \ref{prop:equivalence_FM} (in \ref{appdx:eg_inequity_measures}), both (i) and (iii) are equivalent to $\max_{i\in[N]} u_i - \min_{i\in[N]} u_i = -u_{(1)} + u_{(N)}$. Therefore, a dual set $\calW_\nu$ is given by the singleton $\{(-1,0,\dots,0,1)\in\R^N\}$.

\item[(b)] From \cite{Mesa_et_al:2003}, we can write the Gini deviation (ii) as
$$\sum_{i=1}^N \sum_{j=1}^N |u_i-u_j| = \sum_{i=1}^N 2(2i-1-N) u_{(i)}=\sum_{i=1}^N w'_i u_{(i)},$$
where we let $w'_i=2(2i-1-N)$. Note that $w'_1=2(1-N)<0$, $w'_N=2(N-1)>0$ and $w'_i$ is increasing in $i$. Moreover,
$$\sum_{i=1}^N w'_i = 2\sum_{i=1}^N (2i-1-N) = 2[N(N+1) - N(N+1)] = 0.$$
Therefore, a dual set $\calW_\nu$ is given by the singleton $\{\wb'\}$.

\item[(c)] Consider $\norm{\ub-\ubar\one}_p$ for some $p\in[1,\infty]$. Let $q$ be such that $1/p+1/q=1$, where we let $q=\infty$ if $p=1$, and $q=1$ if $p=\infty$. Since $\ell_p$ norm is the dual of $\ell_q$ norm,
\begin{equation} \label{eqn:pf_existing_FM_dual_set_1}
    \norms{\ub-\ubar\one}_p = \sup_{\wb':\norms{\wb'}_q \leq 1} (\ub-\ubar\one)^\tp\wb' = \sup_{\wb':\norms{\wb'}_q \leq 1} \ub^\tp(\wb'-\wbar'\one),
\end{equation}
where $\wbar'=(1/N)\one^\tp\wb'$. Note that (iv)--(vi) can be written as $\norms{\ub-\ubar\one}_1$, $\norms{\ub-\ubar\one}_2$, and $\norms{\ub-\ubar\one}_\infty$ respectively.  Thus, from \eqref{eqn:pf_existing_FM_dual_set_1}, the desired dual set is
\begin{align*}
    \calW_\nu &=\calS^N\cap\Bigg\{\wb\in\R^N\,\,\Bigg|\,\, \wb=\wb'-\wbar'\one,\, \wbar'=\frac{1}{N}\sum_{j=1}^N w'_j,\,\norms{\wb'}_q \leq 1\Bigg\} \\
    &= \Bigg\{\wb\in\R^N\,\,\Bigg|\,\,\wb=\wb'-\wbar'\one,\, \wbar'=\frac{1}{N}\sum_{j=1}^N w'_j,\,\norms{\wb'}_q \leq 1,\, \wb'\in\Rup^N\Bigg\},
\end{align*}
where the last equality follows from the facts that $\one^\tp\wb=\one^\tp\wb'-\wbar'(\one^\tp\one)=0$, and $\wb\in\Rup^N$ if and only if $\wb'\in\Rup^N$.


\item[(d)] From Proposition \ref{prop:equivalence_FM} (in \ref{appdx:eg_inequity_measures}), (vii) is equivalent to $N\norms{\ub-\ubar\one}_\infty$. The same argument in \eqref{eqn:pf_existing_FM_dual_set_1} shows that 
$$N\norms{\ub-\ubar\one}_p = \sup_{\wb':\norms{\wb'}_q \leq 1} (\ub-\ubar\one)^\tp(N\wb') = \sup_{\wb':\norms{\wb'}_q \leq N} \ub^\tp(\wb'-\wbar'\one),$$
which gives the desired dual set $\calW_\nu$.

\item[(e)] Finally, for (viii), let $k=k(\ub)$ be the number of entries in $\ub$ that are closer to $u_{(1)}$, i.e., $u_{(i)}-u_{(1)} < u_{(N)}-u_{(i)}$ for $i\in[k]$. Then,
\allowdisplaybreaks
\begin{align*}
    \sum_{i=1}^N \max_{j\in[N]} |u_i-u_j|&=\sum_{i=1}^k \big[ u_{(N)} - u_{(i)} \big] + \sum_{i=k+1}^N \big[ u_{(i)} - u_{(1)} \big] \\
    &= \big[-(N-k)-1\big] u_{(1)} + \sum_{i=2}^k (-1) u_{(i)} + \sum_{i=k+1}^{N-1} u_{(i)} + (k+1) u_{(N)}.
\end{align*}
Note that $k$ takes value in $[N]$ only, it suffices to consider $\wb^k$ with entries $w^k_1 = -(N-k)-1$, $w^k_2=\dots=w^k_k=-1$, $w^k_{k+1}=\dots=w^k_{N-1}=1$, and $w^k_N=k+1$. Moreover, it is easy to verify that if there are $k$ entries in $\ub$ that are closer to $u_{(1)}$, then $\nu_{\wb^k}(\ub)\geq \nu_{\wb^h}(\ub)$ for $h \ne k$. Indeed, if $h<k$, then
\allowdisplaybreaks
\begin{align*}
   &\quad\,\,\nu_{\wb^k}(\ub)-\nu_{\wb^h}(\ub) \\
   &=\Bigg\{\sum_{i=1}^k \big[ u_{(N)} - u_{(i)} \big] + \sum_{i=k+1}^N \big[ u_{(i)} - u_{(1)} \big]\Bigg\} - \Bigg\{\sum_{i=1}^h \big[ u_{(N)} - u_{(i)} \big] + \sum_{i=h+1}^N \big[ u_{(i)} - u_{(1)} \big] \Bigg\} \\
   &= \sum_{i=h+1}^k \Big\{\big[ u_{(N)} - u_{(i)} \big] -  \big[ u_{(i)} - u_{(1)} \big]\Big\} > 0.
\end{align*}
Following a similar argument, if $h>k$, we also have $\nu_{\wb^k}(\ub)-\nu_{\wb^h}(\ub)>0$. Hence, a dual set $\calW_\nu$ is given by $\{\wb^k\}_{k=1}^N$. 
\end{itemize}
This completes the proof.

\subsection{Proof of Theorem \ref{thm:characterization_of_equivalence_CFM}}
First, if $\calW_1=\beta\calW_2$ for some $\beta>0$, then we have
$$\nu_1(\ub)=\sup_{\wb\in\calW_1} \nu_{\wb}(\ub) = \sup_{\wb\in\beta\calW_2} \nu_{\wb}(\ub)=\beta\sup_{\wb\in\calW_2} \nu_{\wb}(\ub)=\beta\nu_2(\ub),$$
implying the equivalence of $\nu_1$ and $\nu_2$. Next, if $\nu_1$ is equivalent to $\nu_2$, then $\nu_1(\ub)=\beta\nu_2(\ub)$ for some $\beta >0$. From the proof of Theorem \ref{thm:convex_FM_dual}, we have
\begin{equation} \label{eqn:pf_characterization_of_equivalence_CFM_1}
    \sup_{\wb\in\dom(\nu^*_1)} \ub^\tp\wb =\nu_1(\ub)=\beta\nu_2(\ub) = \beta \sup_{\wb\in\dom(\nu^*_2)} \ub^\tp\wb = \sup_{\wb\in\beta\dom(\nu^*_2)} \ub^\tp\wb.
\end{equation}
As a result,  \eqref{eqn:pf_characterization_of_equivalence_CFM_1} implies that the support functions of the convex compact sets $\dom(\nu^*_1)$ and $\beta\dom(\nu^*_2)$ are the same. Therefore, we have $\dom(\nu^*_1)=\beta\dom(\nu^*_2)$ by Theorem 13.2 of \cite{Rockafellar:1970}, and thus, $\calW_1=\dom(\nu^*_1)\cap\Rup^N=\beta\dom(\nu^*_2)\cap\Rup^N=\beta\calW_2$. 

\subsection{Proof of Theorem \ref{thm:FM_min_reformulation_order_based}}
From Theorem \ref{thm:order_based_FM_characterization_I}, we can write $\nu(\ub)=\max_{P\in\calP} \ub^\tp (P\wb)$, where $\calP$ is the set of all permutation matrices given by
\begin{equation} \label{eqn:permutation_matrix_set}
    \calP=\Bigg\{P\in\R^{N\times N} \,\Bigg|\, \sum_{i=1}^N P_{ij}=1,\, \sum_{j=1}^N P_{ij} = 1,\, P_{ij}\in\{0,1\},\,\forall i\in[N],\,j\in[N] \Bigg\}.
\end{equation}
Note that the objective $\ub^\tp (P\wb)$ is linear in $P$, and the constraint matrix formed by the assignment constraints in $\calP$ is totally unimodular \citep{Martello_Toth:1987}. Hence, $\max_{P\in\calP} \ub^\tp (P\wb)$ is a linear program in variables $P_{ij}$ and we can take its dual as
\begin{align}
\underset{P\in\calP}{\textup{max}} \,\,\ub^\tp (P\wb)=\underset{\lambdab\in\R^N,\, \thetab\in\R^N}{\textup{min}\,}  \Big\{
&  \one^\tp(\lambdab+\thetab) \,\Big|\,  \lambda_i + \theta_j \geq u_i w_j, \,\forall i\in[N],\, j\in[N] \Big\}. \label{eqn:order_based_FM_characterization_II}
\end{align}%
Combining \eqref{eqn:order_based_FM_characterization_II} with the outer minimization over $\xb$ and $\ub$ in problem \eqref{prob:FM_min}, we obtain the reformulation in \eqref{eqn:FM_min_reformulation_order_based}. 

\begin{remark}
We note that a similar reformulation technique also appears in studies that optimize the ordered-median function (see, e.g., \citealp{Blanco_et_al:2016}). However, as pointed out in Section~\ref{sec:order_based_fairness_measures}, our order-based fairness measure is different from the ordered-median function. In particular, the latter is a combination of the mean (an inefficiency measure) and an order-based fairness measure; see Section~\ref{sec:order_based_fairness_measures} for a thorough discussion on the differences.
\end{remark}

\subsection{Proof of Proposition \ref{prop:dual_variables_LB_UB}}
First, note that we can relax the set of permutation matrices $\calP$ in \eqref{eqn:permutation_matrix_set} as
\allowdisplaybreaks
\begin{equation*}
    \calP=\Bigg\{P\in\R^{N\times N} \,\Bigg|\, \sum_{i=1}^N P_{ij}=1,\, \sum_{j=1}^N P_{ij} \leq 1,\, P_{ij}\in\{0,1\},\,\forall i\in[N],\,j\in[N] \Bigg\}.
\end{equation*}
Indeed, the first constraint requires that every column of $P$ has exactly one entry with value $1$ and the second constraint requires that every row of $P$ has at most one entry with value $1$. Since $P$ is an $N$-by-$N$ matrix, this immediately ensures that $P$ has exactly one entry with value $1$ in every row and column. As a result, the dual variable $\lambda_i$ associated with the second constraint is non-negative. Next, for a given $\ub\in\R^N$, let $P^\star$ be an optimal solution to the primal problem $\max_{P\in\calP} \ub^\tp (P\wb)$ and $(\lambda^\star,\theta^\star)$ be an optimal dual solution. Let $\{(i,\pi(i))\}_{i\in[N]}$ be the set of pairs of indices such that $P^\star_{i,\pi(i)}=1$ (and is $0$ otherwise). Without loss of generality, we can assume that there exists $i'\in[N]$ such that $\lambda^\star_{i'}=0$. Indeed, if $\varepsilon:=\min_{i\in[N]}\lambda^\star_i >0$, then  $(\widetilde{\lambdab},\widetilde{\thetab})$ defined by $\widetilde{\lambda}_i = \lambda^\star_i - \varepsilon$ and $\widetilde{\theta}_i = \theta^\star_i + \varepsilon$ for all $i\in[N]$ is another optimal solution to \eqref{eqn:FM_min_reformulation_order_based}.

Now, we derive an upper bound on $\theta^\star_i$. By the complementary slackness condition, if $P^\star_{i,\pi(i)}=1$, then we have $\lambda^\star_i+\theta^\star_{\pi(i)} = u_i w_{\pi(i)}$ (i.e., with a zero slack variable), implying that $\theta^\star_{\pi(i)}=u_iw_{\pi(i)}-\lambda^\star_i$. Since $\lambda^\star_i\geq 0$, we immediately have $\theta^\star_{\pi(i)}\leq u_iw_{\pi(i)}\leq u_\text{max} w_{\pi(i)}$, where $u_\text{max}=\max_{i\in[N]} u_i$. Now, we derive an upper bound on $\lambda_i$. Using $\theta^\star_{\pi(i)}=u_iw_{\pi(i)}-\lambda^\star_i$ and letting $i=\pi^{-1}(j)$, constraints \eqref{eqn:FM_min_reformulation_direct_con1} implies that $\lambda^\star_i+ \big( u_{\pi^{-1}(j)} w_j - \lambda^\star_{\pi^{-1}(j)} \big) \geq u_i w_j$ for all $i\in[N]$ and $j\in[N]$, which is equivalent to $\lambda^\star_i+ \big( u_j w_{\pi(j)} - \lambda^\star_j \big) \geq u_i w_{\pi(j)}$ for all $i\in[N]$ and $j\in[N]$. Hence, we have $\lambda^\star_i - \lambda^\star _j  \geq (u_i -u_j) w_{\pi(j)}$ for all $i\in[N]$ and $j\in[N]$. Setting $i=i'$, since $\lambda^\star_{i'}=0$, the inequalities imply that $\lambda^\star_j\leq (u_j-u_{i'})w_{\pi(j)} \leq (u_\text{max}-u_\text{min}) \norms{\wb}_\infty$ for all $j\in[N]\setminus\{i'\}$, where $u_\text{min} = \min_{i\in[N]} u_i$. Finally, the lower bound of $\theta_i$ follows from the upper bound of $\lambda_i$ that $\theta^\star_{\pi(i)}=u_iw_{\pi(i)}-\lambda^\star_i \geq u_iw_{\pi(i)}- (u_\text{max}-u_\text{min}) \norms{\wb}_\infty$.

Note that the above lower and upper bounds on $\lambdab$ and $\thetab$ are obtained by fixing a vector $\ub\in\R^N$. Thus, the desired lower bound follows from taking the infimum over all feasible $\ub\in\calU := \{\ub\in\R^N\mid \ub=U(\xb),\,\xb\in\calX\}$ and the desired upper bounds follow from taking the supremum over all $\ub\in\calU$. Thus, we obtain the bounds on $\lambdab$
$$0\leq \lambda_i \leq \sup_{\ub\in\calU} \big\{u_\text{max}-u_\text{min}\big\}\cdot\norms{\wb}_\infty \leq (\Umax-\Umin)\norms{\wb}_\infty =: \lambdabar, $$
the lower bound on $\thetab$
$$ \theta_j \geq \inf_{\ub\in\calU} \Big\{u_{\pi^{-1}(j)} w_j - (u_\text{max}-u_\text{min}) \norms{\wb}_\infty\Big\} \geq \min\big\{\Umax w_j,\, \Umin w_j\big\} - \lambdabar, $$
and the upper bound on $\thetab$
$$ \theta_j \leq \sup_{\ub\in\calU} \big\{u_\text{max} w_j \big\} \leq \max\big\{\Umax w_j,\, \Umin w_j\big\}.$$
This completes the proof. 

\subsection{Proof of Proposition \ref{prop:FM_min_convex_FM_CCG_master}}

By the LP reformulation of $\nu_{\wb^k}(\ub)$ in \eqref{eqn:order_based_FM_characterization_II}, $\delta\geq\nu_{\wb^k}(\ub)$ if and only if there exist $\lambdab^k$ and $\thetab^k$ such that $\lambda^k_i+\theta^k_{i'}\geq u_iw^k_{i'}$ for all $i\in[N]$, $i'\in[N]$ and $\delta\geq \one^\tp(\lambdab^k+\thetab^k)$. Therefore, by introducing the variables $\lambdab^k$ and $\thetab^k$, we can reformulate \eqref{prob:C&CG_Master_original} into \eqref{prob:C&CG_Master}.

\subsection{Proof of Lemma \ref{lemma:CFM_dual_set_Hausdorff}}
First, note that
\begin{align*}
    \sup_{\wb_1\in\calW_1} \nu_{\wb_1}(\ub) - \sup_{\wb_2\in\calW_2} \nu_{\wb_2}(\ub) &= \sup_{\wb_1\in\calW_1} \inf_{\wb_2\in\calW_2} \Big\{\nu_{\wb_1}(\ub)-\nu_{\wb_2}(\ub) \Big\} \\
    &\leq \sup_{\wb_1\in\calW_1} \inf_{\wb_2\in\calW_2} \norms{\wb_1-\wb_2}\cdot\norms{\ub}_2 \leq d_H(\calW_1,\calW_2)\cdot\norms{\ub}_2,
\end{align*}
where the first inequality follows from Cauchy-Schwarz inequality, and the second inequality follows from the definition of $d_H$. Next, using the same argument, we have 
\begin{align*}
    \sup_{\wb_1\in\calW_1} \nu_{\wb_1}(\ub) - \sup_{\wb_2\in\calW_2} \nu_{\wb_2}(\ub) &= \inf_{\wb_2\in\calW_2} \sup_{\wb_1\in\calW_1} \Big\{\nu_{\wb_1}(\ub)-\nu_{\wb_2}(\ub) \Big\} \\
    &\geq -\sup_{\wb_2\in\calW_2} \inf_{\wb_1\in\calW_1} \norms{\wb_1-\wb_2}\cdot\norms{\ub}_2 \geq -d_H(\calW_1,\calW_2)\cdot\norms{\ub}_2.
\end{align*}
Hence, we obtain  $\big| \nu_1(\ub)-\nu_2(\ub) \big| \leq d_H(\calW_1,\calW_2) \cdot \norms{\ub}_2$. 

\subsection{Proof of Theorem \ref{thm:stability_analysis}}
First, we prove part (a). Note that
\allowdisplaybreaks
\begin{subequations}
    \begin{align}
    \upsilon^\star_2 - \upsilon^\star_1 &= \min_{\xb\in\calX} \Big\{f(U(\xb))+\gamma\nu_2(U(\xb))\Big\} - \min_{\xb\in\calX} \Big\{f(U(\xb))+\gamma\nu_1(U(\xb))\Big\}\nonumber \\
    &\leq \Big[f(U(\xb^\star_1)) + \gamma\sup_{\wb\in\calW_2} \nu_{\wb}(U(\xb^\star_1))\Big]-\Big[ f(U(\xb^\star_1)) + \gamma\sup_{\wb\in\calW_1} \nu_{\wb}(U(\xb^\star_1)) \Big] \label{eqn:pf_stability_optimal_sol_1}\\
    &\leq \gamma\Umax d_H(\calW_1,\calW_2), \label{eqn:pf_stability_optimal_sol_2}
    \end{align}
\end{subequations}
where \eqref{eqn:pf_stability_optimal_sol_1} follows from $\xb^\star_1\in\calX$ and the optimality of $\xb^\star_1$, and \eqref{eqn:pf_stability_optimal_sol_2} follows from Lemma \ref{lemma:CFM_dual_set_Hausdorff}. Using the same logic, it is easy to verify that $\upsilon^\star_1-\upsilon^\star_2 \leq \gamma\Umax d_H(\calW_1,\calW_2)$. This completes the proof showing that $\big|\upsilon^\star_1-\upsilon^\star_2\big| \leq \gamma\Umax d_H(\calW_1,\calW_2)$.

Next, we proceed to prove part (b). Note that
\begin{subequations}
\begin{align}
    \upsilon^\star_2 - \upsilon^\star_1 &= \min_{\xb\in\calX} \Big\{f(U(\xb))+\gamma\nu_2(U(\xb))\Big\} - \min_{\xb\in\calX} \Big\{f(U(\xb))+\gamma\nu_1(U(\xb))\Big\}\nonumber \\
    &= \Big[f(U(\xb^\star_2)) + \gamma\nu_{\wb^\star_2}(U(\xb^\star_2))\Big] -  \Big[f(U(\xb^\star_1)) + \gamma\nu_{\wb^\star_1}(U(\xb^\star_1))\Big] \nonumber\\
    &= \bigg\{\Big[f(U(\xb^\star_2)) + \gamma\nu_{\wb^\star_2}(U(\xb^\star_2))\Big] - \Big[f(U(\xb^\star_1)) + \gamma\nu_{\wb^\star_2}(U(\xb^\star_1))\Big]   \bigg\} \nonumber\\
    &\quad\quad+\bigg\{\Big[f(U(\xb^\star_1)) + \gamma\nu_{\wb^\star_2}(U(\xb^\star_1))\Big] - \Big[f(U(\xb^\star_1)) + \gamma\nu_{\wb^\star_1}(U(\xb^\star_1))\Big]   \bigg\} \nonumber \\
    &\leq -\tau_2 \norms{\xb^\star_2-\xb^\star_1}_2^2 + \gamma\bigg\{ \sup_{\wb\in\calW_2}\nu_{\wb}(U(\xb^\star_1)) - \sup_{\wb\in\calW_1}\nu_{\wb}(U(\xb^\star_1)) \bigg\} \label{eqn:pf_stability_optimal_sol_3}\\
    &\leq -\tau_2 \norms{\xb^\star_2-\xb^\star_1}_2^2 + \gamma\Umax d_H(\calW_1,\calW_2), \label{eqn:pf_stability_optimal_sol_4}
\end{align}
\end{subequations}
where \eqref{eqn:pf_stability_optimal_sol_3} follows from \eqref{eqn:local_quad_growth_cond}, $\wb^\star_2\in\calW_2$, and the optimality of $\wb^\star_1$, and \eqref{eqn:pf_stability_optimal_sol_4} follows from Lemma \ref{lemma:CFM_dual_set_Hausdorff}. Using the same logic, it is easy to verify that $\upsilon^\star_1-\upsilon^\star_2 \leq -\tau_1\norms{\xb^\star_1-\xb^\star_2}_2^2 + \gamma\Umax d_H(\calW_1,\calW_2)$. Thus, we have
$$0 \leq \big| \upsilon^\star_1 - \upsilon^\star_2 \big| \leq -\min\{\tau_1,\tau_2\} \norms{\xb^\star_1-\xb^\star_2}_2^2 + \gamma\Umax d_H(\calW_1,\calW_2),$$
which directly implies the desired inequality \eqref{eqn:statbility_bounds_on_optimal_solution}. 

\subsection{Proof of Theorem~\ref{thm:necessary_rel_CFM}}
Note that it is straightforward to verify the sufficient part using the definition of $\rho$. Thus, we focus on the necessary part. Assuming that $\rho=\nu/g$ satisfies Axioms~\ref{axiom:normalization_rel}, \ref{axiom:symmetry}, and \ref{axiom:scale_invariance}, we next show that $g$ is symmetric and positive homogeneous with $g \geq \nu$. For notational simplicity, we define $\calU_0=\{\ub\mid\ub=\alpha\one,\,\alpha\geq 0\}$ and its complement $\calU_0^c=\R^N_+\setminus\calU_0$. We divide the proof into the following three steps.

\begin{itemize}[topsep=2mm,itemsep=0mm]
    \item \textit{Step 1}. First, consider $\ub\in\calU_0^c$. By Axiom~\ref{axiom:normalization_rel}, we have $\rho(\ub)>0$. Also, by definition of absolute convex fairness measures (Axiom~\ref{axiom:normalization}), we have $\nu(\ub)>0$. Since $0<\rho(\ub)=\nu(\ub)/g(\ub)$, we have $g(\ub)>0$. Also, since $\rho(\ub)\leq 1$ by Axiom~\ref{axiom:normalization_rel}, it follows that $g(\ub)\geq \nu(\ub)\geq 0$ for any $\ub\in\calU_0^c$. Now, consider $\ub\in\calU_0$, i.e., $\ub=\ub_\alpha=\alpha\one$ for some $\alpha>0$. Consider the sequence $\{\ub_\alpha^k=[(\alpha+1)/k,\alpha,\dots,\alpha]^\tp\}_{k\in\N}\subset\R_+^N$ that converges to $\ub_\alpha$ as $k\rightarrow\infty$. Since $\ub_\alpha^k\in\calU_0^c$, we have $g(\ub_\alpha^k)\geq \nu(\ub_\alpha^k)$ for all $k\in\N$. Taking limit $k\rightarrow\infty$ on both sides of the inequality, we obtain $g(\ub_\alpha)\geq \nu(\ub_\alpha)$ by continuity of $g$ and $\nu$. This shows that $g\geq\nu$.

    \item \textit{Step 2}. Similar to step 1, consider $\ub\in\calU_0^c$. As shown in step 1, we have $\rho(\ub)>0$ and $g(\ub)>0$. Note that for any permutation matrix $P$, we have $g(P\ub)=\nu(P\ub)/\rho(P\ub)=\nu(\ub)/\rho(\ub)=g(\ub)$ since $\rho(P\ub)=\rho(\ub)$ by Axiom~\ref{axiom:symmetry} and $\nu(P\ub)=\nu(\ub)$ by symmetry of $\nu$. It follows that $g$ is symmetric. Finally, if $\ub\in\calU_0$, we can apply a similar limiting argument in step 1.

    \item \textit{Step 3}. Similar to step 1, consider $\ub\in\calU_0^c$. As shown in step 1, we have $\rho(\ub)>0$ and $g(\ub)>0$. Note that for any $\alpha>0$, we have $g(\alpha\ub)=\nu(\alpha\ub)/\rho(\alpha\ub)=\alpha\nu(\ub)/\rho(\ub)=\alpha g(\ub)$ since $\rho(\alpha\ub)=\rho(\ub)$ by Axiom~\ref{axiom:scale_invariance} and $\nu(\alpha\ub)=\alpha\nu(\ub)$ by positive homogeneity of $\nu$. It follows that $g$ is positive homogeneous. Finally, if $\ub\in\calU_0$, we can apply a similar limiting argument in step 1.

\end{itemize}

This completes the proof. 

\subsection{Proof of Theorem~\ref{thm:sufficient_rel_CFM}}
First, note that $g$ is symmetric and positive homogeneous. In addition, we claim that $g(\ub)\geq \nu(\ub)$ for all $\ub\in\R^N_+$. To prove this claim, consider the following optimization problem:
\begin{equation} \label{eqn_pf:thm_sufficient_rel_CFM_1}
    \max_{\ub\in\R^N_+} \Bigg\{ \nu(\ub)=\sup_{\wb\in\calW_\nu} \sum_{i=1}^N w_iu_{(i)} \,\Bigg|\, \one^\tp\ub=\gamma\Bigg\}
\end{equation}
for any $\gamma > 0$. It is easy to show that $\ub^\star_\gamma = (0,\dots,0,\gamma)^\tp$ is an optimal solution to \eqref{eqn_pf:thm_sufficient_rel_CFM_1} with objective value $\gamma \sup_{\wb\in\calW_\nu}\norms{\wb_+}_\infty$. Thus, for any $\gamma > 0$ and $\ub\in\R^N_+$ such that $\one^\tp\ub= \gamma$, we have
$$\nu(\ub)\leq \gamma \sup_{\wb\in\calW_\nu}\norms{\wb_+}_\infty=N\ubar \sup_{\wb\in\calW_\nu}\norms{\wb_+}_\infty\leq CN\ubar=g(\ub),$$
showing that $g(\ub)\geq \nu(\ub)$ for all $\ub\in\R^N_+$. Thus, by Theorem~\ref{thm:necessary_rel_CFM}, $\rho$ satisfies Axioms~\ref{axiom:normalization_rel}, \ref{axiom:symmetry}, and \ref{axiom:scale_invariance}.
    
Finally, we show that $g$ satisfies Axiom~\ref{axiom:Schur_convex}. Note that a function $h:\R_+^N\rightarrow\R$ is Schur convex if and only if $h$ is symmetric and $h$ satisfies the following condition: $h(\ub^\lambda_{12})\leq h(\ub)$ for any $\ub\in\R_+^N$ and $\lambda\in(0,1)$, where $\ub^\lambda_{12}=\lambda\ub + (1-\lambda)\ub_{12}$ and $\ub_{12}=(u_2,u_1,u_3,\dots,u_N)^\tp$ (see Theorem 2.4 of \citealp{Stepniak:2007}). Since $\rho$ is symmetric, it suffices to show that $\rho(\ub^\lambda_{12})\leq \rho(\ub)$ for any $\ub\in\R^N_+$. Indeed, since $\one^\tp\ub^\lambda_{12}=\one^\tp\ub=\one^\tp\ub_{12}$, letting $\mu=\one^\tp\ub/N$, we have
$$\rho(\ub^\lambda_{12})=\frac{\nu(\ub^\lambda_{12})}{CN\mu}\leq \frac{\lambda\nu(\ub)+(1-\lambda)\nu(\ub_{12})}{CN\mu}=\frac{\nu(\ub)}{CN\mu}=\frac{\nu(\ub)}{g(\ub)}=\rho(\ub),$$
where the inequality follows from the convexity of $\nu$, and the second equality follows from $\nu(\ub_{12})=\nu(\ub)$ (Axiom~\ref{axiom:symmetry}). Hence, $\rho$ is Schur convex.  

\subsection{Proof of Theorem~\ref{thm:nceessary_sufficient_rel_CFM_mean}}
Note that we have proved the sufficiency part in Theorem \ref{thm:sufficient_rel_CFM}, and hence, we only need to prove the necessity part. From Theorem \ref{thm:necessary_rel_CFM}, we know that $g$ is symmetric and positive homogeneous with $g\geq\nu$. This implies that $\gt:\R_+\rightarrow\R$ is also a (one-dimensional) positive homogeneous function. Theorem 2.2.1 of \cite{Castillo_Ruiz-Cobo:1992} shows that the only class of solutions of this homogeneous functional equation is of the form $\gt(\ubar)=c\ubar$ for some constant $c$. Finally, by Axiom~\ref{axiom:normalization_rel}, we have $\rho(\ub)\leq 1$, which implies $\nu(\ub)\leq c\ubar$. Recall, from the proof of Theorem~\ref{thm:sufficient_rel_CFM}, for any $\gamma>0$ and $\ub\in\R^N_+$ such that $\ub^\tp\one = \gamma$, we have $\nu(\ub)\leq \gamma \sup_{\wb\in\calW_\nu}\norms{\wb_+}_{\infty} =N\sup_{\wb\in\calW_\nu}\norms{\wb_+}_{\infty}\cdot\ubar$. Thus, to ensure that $\rho(\ub)\leq 1$ for all $\ub\in\R^N_+$, we must have $c\geq N\sup_{\wb\in\calW_\nu}\norms{\wb_+}_{\infty}$.

\subsection{Proof of Proposition~\ref{prop:perfect_inequality_rel_CFM}}
Assume that $\ub\in\calU^\star$. Let $\gamma=\one^\tp\ub$. Since $\ub\in\calU^\star$, we have $\ub\in\argmax\{\nu(\ub')\mid\one^\tp\ub'=\gamma\}$, which implies that
$$\nu(\ub) \geq \nu(0,\dots,0,\gamma)=\sup_{\wb\in\calW_\nu}\norms{\wb_+}_\infty \gamma = \sup_{\wb\in\calW_\nu}\norms{\wb_+}_\infty (\one^\tp\ub) = g(\ub) > 0. $$
Thus, it follows that $\rho(\ub)=\nu(\ub)/g(\ub)\geq 1$. This shows that $\rho(\ub)=1$.

Next, assume that $\rho(\ub)=1$. From Theorem~\ref{thm:nceessary_sufficient_rel_CFM_mean}, we know that $\rho(\ub)\leq 1$ (Axiom~\ref{axiom:normalization_rel}). This implies that the set of $\ub\in\R^N_+$ such that $\rho(\ub)=1$ is the set of the maximizers of $\rho(\ub)$, i.e., $\ub\in\argmax_{\ub'\in\R^N_+}\{\rho(\ub')\}$. Finally, note that
\begin{align}
\argmax_{\ub'\in\R^N_+}\Bigg\{\rho(\ub')=\frac{\nu(\ub')}{g(\ub')}\Bigg\}&=\argmax_{\ub'\in\R^N_+}\Bigg\{\frac{\nu(\ub')}{\sup_{\wb\in\calW_\nu}\norms{\wb_+}_\infty (\one^\tp\ub')}\Bigg\} \label{eqn_pf:them_perfect_inequality_rel_CFM_1}  \\
&=\bigcup_{\gamma>0}\argmax_{\ub'\in\R^N_+}\big\{\nu(\ub')\mid \one^\tp\ub'=\gamma\big\} \label{eqn_pf:them_perfect_inequality_rel_CFM_2} \\
&=\bigcup_{\gamma>0}\calU^\star_\gamma = \calU^\star. \nonumber
\end{align}
Here, note that the objective function in the maximization problem \eqref{eqn_pf:them_perfect_inequality_rel_CFM_1} is a fraction with both nominator and denominator being positive homogeneous. By classical fractional programming results, this is equivalent to maximizing the nominator with a fixed value of the denominator (see Chapter~4.3 of \citealp{Stancu-Minasian:1997} for details), which leads to \eqref{eqn_pf:them_perfect_inequality_rel_CFM_2}.  

\subsection{Proof of Corollary~\ref{coro:perfect_inequity}}
By Proposition~\ref{prop:perfect_inequality_rel_CFM}, it suffices to show that $\calU^\star_\gamma=\{P(0,\dots,0,\gamma)^\tp\mid P\in\calP\}$ for all $\gamma>0$. That is, $\ub=(0,\dots,0,\gamma)^\tp$ is the unique maximizer to the optimization problem $\max_{\ub'\in\R^N_+}\{\nu(\ub')\mid \one^\tp\ub'=\gamma,\,\ub'\in\Rup^N\}$ in which we seek to find the impact vector $\ub'$ that maximizes the value of the fairness measure $\nu$. Suppose, on the contrary, that there exists $\ub\in\Rup^N$ and $i\in[N-1]$ such that $\rho(\ub)=1$ with $u_j>0$ for all $j\in[i,N]_\Z$ and $u_j=0$ otherwise, i.e., $\ub$ does not take the form $(0,\dots,\gamma)^\tp$. Let $\gamma=\one^\tp\ub$. Define $\ub^{1}=(0,\dots,0,0,\gamma)^\tp$ and $\ub^{2}=(0,\dots,0,\gamma/2,\gamma/2)^\tp$. Consider the following two cases.
\begin{itemize}[topsep=2mm,itemsep=0mm]
    \item If $u_N \leq \gamma/2$, we must have $\ub\preceq\ub^{2}$. By Schur convexity of $\nu$ and our assumption that $\nu(\ub^{1})>\nu(\ub^{2})$, we have  $\nu(\ub)\leq\nu(\ub^{2})<\nu(\ub^{1})$. Since $\one^\tp\ub=\one^\tp\ub^{1}=\gamma$, we arrive at the contradiction that $1=\rho(\ub)<\rho(\ub^{1})=1$.

    \item If $u_N>\gamma/2$, we define the vector $\ubt\in\Rup^N$ by $\ut_j=0$ for $j\in[N-2]$, $\ut_{N-1}=\sum_{j=i}^{N-1} u_j$, and $\ut_N=u_N$. By construction, we have $\ub\preceq\ubt$ and $\ubt=\lambda\ub^{2}+(1-\lambda)\ub^{1}$, where $\lambda=2(1-u_N/\gamma)\in(0,1)$. Thus, by Schur convexity of $\nu$, convexity of $\nu$ and our assumption that $\nu(\ub^{1})>\nu(\ub^{2})$, we have $\nu(\ub)\leq \nu(\ubt)\leq \lambda \nu(\ub^{2}) + (1-\lambda)\nu(\ub^{1})<\nu(\ub^{1})$. Again, since $\one^\tp\ub=\one^\tp\ub^{1}=\gamma$, we arrive at the contradiction that $1=\rho(\ub)<\rho(\ub^{1})=1$.
\end{itemize}
Thus, this shows that $\rho(\ub)=1$ if and only if $\ub=P(0,\dots,0,\gamma)^\tp$ for any $\gamma>0$ and permutation matrix $P$.

\section{Solution Approaches with Convex Fairness Measures in Constraint Form} \label{appdx:sol_approach_CFM_constraint}

In this section, we propose solution approaches for optimization problem of the form \eqref{prob:FM_eff_min_constraint} with our proposed convex fairness measures.

\subsection{Absolute Convex Fairness Measure} \label{appdx:sol_approach_CFM_constraint_abs}

Using the dual representation of absolute convex fairness measures (see Section~\ref{sec:convex_fairness_measures}), we can reformulate \eqref{prob:FM_eff_min_constraint} into
\begin{equation} \label{eqn:FM_eff_min_constraint_general}
    \min_{\xb,\,\ub} \Big\{ f(\ub) \,\Big|\, \nu_{\wb}(\ub)\leq \eta,\,\forall \wb\in\calW_\nu,\, \ub = U(\xb),\, \xb\in\calX\Big\}.
\end{equation}
Note that decision-makers typically specify the upper bound $\eta$ on the absolute convex fairness measure in \eqref{eqn:FM_eff_min_constraint_general}. If they chose a very small $\eta$, then \eqref{eqn:FM_eff_min_constraint_general} could be infeasible.

Let us first consider the case when $\nu$ is an order-based fairness measure, i.e., $\calW_\nu=\{\wb\}$ is a singleton. Using the same proof techniques of Theorem~\ref{thm:FM_min_reformulation_order_based}, we can reformulate \eqref{eqn:FM_eff_min_constraint_general} as
\allowdisplaybreaks
\begin{subequations}  \label{eqn:FM_min_constraint_reformulation_order_based}
\begin{align}
\underset{\xb,\,\ub,\,\lambdab\in\R^N,\, \thetab\in\R^N}{\textup{minimize}\,} \quad
&  f(\ub)  \\   
\textup{subject to} \hspace{9mm}
& \one^\tp(\lambdab+\thetab) \leq \eta, \label{eqn:FM_min_constraint_reformulation_direct_con1} \\
&  \lambda_i + \theta_{i'} \geq u_i w_{i'}, \quad\forall i\in[N],\, i'\in[N],  \label{eqn:FM_min_constraint_reformulation_direct_con2}\\
& \ub = U(\xb),\, \xb\in\calX.
\end{align}%
\end{subequations}

Second, when $\nu$ is an absolute convex fairness measure, we propose a  C\&CG algorithm similar to Algorithm~\ref{algo:decomposition_method} to solve \eqref{eqn:FM_eff_min_constraint_general}. Algorithm~\ref{algo:decomposition_method_constraint} summarizes the steps of this algorithm. In this C\&CG, we solve a master problem and a subproblem at each iteration. Specifically, at iteration $j$, we solve the following master problem:
\begin{equation} \label{eqn:C&CG_constraint_master_original}
    \min_{\xb,\,\ub} \Big\{ f(\ub) \,\Big|\, \nu_{\wb}(\ub)\leq \eta,\,\forall \wb\in\{\wb^0,\dots,\wb^{j-1}\},\, \ub = U(\xb),\, \xb\in\calX\Big\},
\end{equation}
where $\{\wb^0,\dots,\wb^{j-1}\}\subseteq\calW_\nu$. Following the same techniques in the proof Proposition~\ref{prop:FM_min_convex_FM_CCG_master}, we derive the following equivalent solvable reformulation of \eqref{eqn:C&CG_constraint_master_original}:
\allowdisplaybreaks
\begin{subequations} \label{eqn:C&CG_constraint_master}
\begin{align}
\underset{\xb,\,\ub,\,\lambdab\in\R^N,\, \thetab\in\R^N}{\textup{minimize}\,} \quad
&  f(\ub)  \\   
\textup{subject to} \hspace{9mm}
& \one^\tp(\lambdab^k+\thetab^k) \leq \eta,\quad\forall k\in[0,j-1]_\Z, \\
&  \lambda^k_i + \theta^k_{i'} \geq u_i w^k_{i'}, \quad\forall i\in[N],\, i'\in[N],\, k\in[0,j-1]_\Z,  \\
& \ub = U(\xb),\, \xb\in\calX.
\end{align}%
\end{subequations} 
Since only a set of weight vectors in $\calW_\nu$ is considered, the master problem \eqref{eqn:C&CG_constraint_master} is a relaxation of the original problem \eqref{eqn:FM_eff_min_constraint_general}, i.e., the feasible region of \eqref{eqn:FM_eff_min_constraint_general} is a subset of the feasible region of \eqref{eqn:C&CG_constraint_master}. Thus, if \eqref{eqn:C&CG_constraint_master} is infeasible, then we can conclude that the original problem \eqref{eqn:FM_eff_min_constraint_general} is also infeasible. Otherwise, with the optimal solution $\ub^j$ from the master problem, we solve the same subproblem \eqref{prob:C&CG_Subproblem} and record the optimal solution $\wb^j\in\calW_\nu$ and optimal value $D^j$. Note that $D^j=\nu(\ub^j)$ is the value of the absolute fairness measure evaluated at $\ub^j$. Thus, if $D^j > \eta$, the current solution $(\xb^j,\ub^j)$ is infeasible to the original problem \eqref{eqn:FM_eff_min_constraint_general}, and thus we enlarge the set of weight vectors and proceed to the next iteration. Otherwise, we terminate and conclude that $(\xb^j,\ub^j)$ is an optimal solution. 
\IncMargin{0em}
\begin{algorithm}[t!] 
\SetKwInOut{Initialization}{Initialization}
\Initialization{Set $LB=0$, $UB=\infty$, $\varepsilon>0$, $j=1$, $\wb^0\in\calW_\nu$.}
\textbf{1. Master problem.} \\
\hspace{5.8mm}Solve master problem \eqref{eqn:C&CG_constraint_master} with weights $\{\wb^0,\dots,\wb^{j-1}\}$.\\
\hspace{5.8mm}If the master problem \eqref{eqn:C&CG_constraint_master} is infeasible, terminate and return the infeasibility of \eqref{eqn:FM_eff_min_constraint_general}.\\
\hspace{5.8mm}Otherwise, record the optimal solution $(\xb^j,\ub^j)$. \\
%
\textbf{2. Subproblem.} Solve subproblem \eqref{prob:C&CG_Subproblem} for fixed $\ub=\ub^j$.\\
\hspace{5.8mm}Record the optimal solution $\wb^j$ and value $D^j$. \\ 
\hspace{5.8mm}If $D^j\leq \eta$, terminate and return the optimal solution $(\xb^j,\ub^j)$. \\
\textbf{3. Scenario set enlargement.}\\
\hspace{5.8mm}Update $j \leftarrow j+1$ and go back to step 1.
\BlankLine
\caption{A column-and-constraint generation (C\&CG) method to solve \eqref{eqn:FM_eff_min_constraint_general}} \label{algo:decomposition_method_constraint}
\end{algorithm}\DecMargin{1em}

\subsection{Relative Convex Fairness Measure}  \label{appdx:sol_approach_CFM_constraint_rel}

Recall that, in Section~\ref{appdx:relative_convex_fairness_measures}, we define a relative convex fairness measure by $\rho(\ub)=\nu(\ub)/g(\ub)$ for some absolute convex fairness measure $\nu$ and continuous normalization function $g$. Thus, the corresponding equity-promoting optimization model is
\begin{equation} \label{eqn:FM_eff_min_constraint_general_rel}
    \min_{\xb,\,\ub} \Big\{ f(\ub) \,\Big|\, \rho(\ub)=\nu_{\wb}(\ub)/g(\ub)\leq \eta,\,\forall \wb\in\calW_\nu,\, \ub = U(\xb),\, \xb\in\calX\Big\}.
\end{equation}
for some $\eta>0$. From Theorem~\ref{thm:necessary_rel_CFM}, the normalization function $g$ satisfies $g\geq\nu\geq 0$. In particular, if $g(\ub)=0$, we must have $\nu(\ub)=0$. Therefore, we can reformulate \eqref{eqn:FM_eff_min_constraint_general_rel} into
\begin{equation} \label{eqn:FM_eff_min_constraint_general_rel_2}
    \min_{\xb,\,\ub} \Big\{ f(\ub) \,\Big|\, \nu_{\wb}(\ub)\leq \eta g(\ub),\,\forall \wb\in\calW_\nu,\, \ub = U(\xb),\, \xb\in\calX\Big\},
\end{equation}
which takes the same form of \eqref{eqn:FM_eff_min_constraint_general}. Thus, we can apply similar solution approaches as in \ref{appdx:sol_approach_CFM_constraint_abs} to solve \eqref{eqn:FM_eff_min_constraint_general_rel_2}.

\section{Additional Computational Results for the Fair Facility Location Problem} \label{appdx:add_comp_results}


In this section, we provide additional results comparing the computational performance of our proposed unified reformulations and solution methods and traditional techniques using a set of larger instances of the fair facility location problem introduced in Section~\ref{subsec:expt_facility_location}. Specifically, we follow the experimental settings discussed in Section~\ref{subsec:expt_facility_location} to generate five random instances of each of following combinations of problem parameters: (a) $(|I|,|J|)=(60,40)$ with $p\in\{13, 10, 8\}$, (b) $(|I|,|J|)=(80,20)$ with $p\in\{7, 5, 4\}$, (c) $(|I|,|J|)=(80,35)$ with $p\in\{7, 5, 4\}$. These instances reflect realistic scenarios where the number of customer locations is larger than that of potential facility locations, i.e., $|I| > |J|$.

Tables \ref{table:FLP_comp_time_MAD_add}--\ref{table:FLP_comp_time_Gini_deviation_add} present the minimum (min), average (avg), and maximum (max) solution time (in seconds) of the generated instances solved using our formulations and traditional ones for the problem with the MAD and Gini deviation, respectively. The observations are similar to those in Section~\ref{subsec:expt_facility_location}.  First, using our unified reformulation and the proposed  C\&CG, we can solve all the generated instances significantly faster than the classical reformulation techniques.   Second, solution time increases as $\gamma$ decreases, i.e., when more emphasis is placed on fairness. This suggests that the problem becomes more challenging when seeking fairer decisions. Notably, the solution times of traditional approaches are substantially longer than our solution times when $\gamma$ is smaller. Consider, for example, the largest instances with $(|I|,|J|)=(80,35)$. We can solve all the instances using our unified reformulations and C\&CG method within two hours. Specifically, the average solution using our approach ranges from $9$ minutes to $13$ minutes for the MAD with $\gamma=0.2$ and $12$ minutes to $62$ minutes for the Gini deviation and $\gamma=0.3$. In contrast, we are unable to solve a large number of these generated instances within two hours using traditional reformulations \eqref{eqn:FLP_direct_reformulation_MAD} and \eqref{eqn:FLP_direct_reformulation_Gini_deviation}.  For instances that were not solved by traditional formulations within two hours, the relative optimality gap reported by the solver at termination ranges from $2.4\%$ to $15.5\%$. Third, the gap in solution time between the two approaches generally widens as $p$ decreases, i.e., as the problem becomes more constrained. These results further underscore the computational efficiency of our approach compared with traditional methods.

\begin{table}[t]
\footnotesize
\center
\renewcommand{\arraystretch}{0.8}
\caption{Solution time (in seconds) over five randomly generated instances with $\gamma\in\{0.3,0.2\}$ using the absolute deviation from mean (MAD). \textit{Note:} Solution times with `$>$' indicate that one or more instances cannot be solved within the imposed two-hour time limit. } \label{table:FLP_comp_time_MAD_add} 
\begin{tabular}{clrrrrrrrrrrrrrr}
\Xhline{1.0pt}
$|I|=40$,   $|J|=60$ &  &  &  & \multicolumn{3}{c}{$\gamma = 0.3$} &  & \multicolumn{3}{c}{$\gamma = 0.2$} \\ \cline{5-11}
 &  &  &  & min & avg & max &  & min & avg & max \\ \cline{5-7} \cline{9-11}
$p = 13$ & \multicolumn{2}{l}{C\&CG Method} &  & 1& 7 & 11 &  & 2 & 567 & 2340 \\
 & \multicolumn{2}{l}{Traditional Reformulation \eqref{eqn:FLP_direct_reformulation_MAD}} &  & 1 & 40 & 63 &  & 6 & \textgreater{}5802 & \textgreater{}7200 \\ \Xhline{0.5pt}
 &  &  &  & min & avg & max &  & min & avg & max \\ \cline{5-7} \cline{9-11}
$p =    10$ & \multicolumn{2}{l}{C\&CG Method} &  & 4 & 24 & 50 &  & 86 & 867 & 3263 \\
 & \multicolumn{2}{l}{Traditional Reformulation \eqref{eqn:FLP_direct_reformulation_MAD}} &  & 6 & \textgreater{}1490 & \textgreater{}7200 &  & 405 & \textgreater{}4755 & \textgreater{}7200 \\ \Xhline{0.5pt}
 &  &  &  & min & avg & max &  & min & avg & max \\ \cline{5-7} \cline{9-11}
$p =    8$ & \multicolumn{2}{l}{C\&CG Method} &  & 4 & 18 & 35 &  & 15 & 853 & 3377 \\
 & \multicolumn{2}{l}{Traditional Reformulation \eqref{eqn:FLP_direct_reformulation_MAD}} &  & 4 & \textgreater{}1477 & \textgreater{}7200 &  & 153 & \textgreater{}4667 & \textgreater{}7200 \\ \Xhline{1.0pt}
$|I|=80$, $|J|=20$ &  &  &  & \multicolumn{3}{c}{$\gamma = 0.3$} &  & \multicolumn{3}{c}{$\gamma = 0.2$} \\
 &  &  &  & min & avg & max &  & min & avg & max \\ \cline{5-7} \cline{9-11}
$p = 7$ & \multicolumn{2}{l}{C\&CG Method} &  & 1 & 3 & 8 &  & 7 & 66 & 162 \\
 & \multicolumn{2}{l}{Traditional Reformulation \eqref{eqn:FLP_direct_reformulation_MAD}} &  & 1 & 5 & 15 &  & 4 & \textgreater{}1613 & \textgreater{}7200 \\ \Xhline{0.5pt}
 &  &  &  & min & avg & max &  & min & avg & max \\ \cline{5-7} \cline{9-11}
$p =    5$ & \multicolumn{2}{l}{C\&CG Method} &  & 1 & 2 & 5 &  & 16 & 30 & 50 \\
 & \multicolumn{2}{l}{Traditional Reformulation \eqref{eqn:FLP_direct_reformulation_MAD}} &  & 1 & 5 & 14 &  & 17 & 110 & 216 \\ \Xhline{0.5pt}
 &  &  &  & min & avg & max &  & min & avg & max \\ \cline{5-7} \cline{9-11}
$p =    4$ & \multicolumn{2}{l}{C\&CG Method} &  & 1 & 2 & 5 &  & 13 & 22 & 34 \\
 & \multicolumn{2}{l}{Traditional Reformulation \eqref{eqn:FLP_direct_reformulation_MAD}} &  & 2 & 4 & 12 &  & 15 & 114 & 251 \\ \Xhline{1.0pt}
$|I|=80$, $|J|=35$ &  &  &  & \multicolumn{3}{c}{$\gamma = 0.3$} &  & \multicolumn{3}{c}{$\gamma = 0.2$} \\  \cline{5-11}
 &  &  &  & min & avg & max &  & min & avg & max \\ \cline{5-7} \cline{9-11}
$p = 7$ & \multicolumn{2}{l}{C\&CG Method} &  & 2 & 9 & 30 &  & 28 & 756 & 2055 \\
 & \multicolumn{2}{l}{Traditional Reformulation \eqref{eqn:FLP_direct_reformulation_MAD}} &  & 1 & 14 & 44 &  & 15 & \textgreater{}3832 & \textgreater{}7200 \\ \Xhline{0.5pt}
 &  &  &  & min & avg & max &  & min & avg & max \\ \cline{5-7} \cline{9-11}
$p =    5$ & \multicolumn{2}{l}{C\&CG Method} &  & 1 & 5 & 14 &  & 14 & 551 & 1458 \\
 & \multicolumn{2}{l}{Traditional Reformulation \eqref{eqn:FLP_direct_reformulation_MAD}} &  & 1 & 45 & 144 &  & 218 & \textgreater{}3673 & \textgreater{}7200 \\ \Xhline{0.5pt}
 &  &  &  & min & avg & max &  & min & avg & max \\ \cline{5-7} \cline{9-11}
$p =    4$ & \multicolumn{2}{l}{C\&CG Method} &  & 1 & 9 & 25 &  & 15 & 654 & 2064 \\
 & \multicolumn{2}{l}{Traditional Reformulation \eqref{eqn:FLP_direct_reformulation_MAD}} &  & 1 & 11 & 36 &  & 182 & \textgreater{}3468 & \textgreater{}7200 \\             
\Xhline{1.0pt}
\end{tabular}
\end{table}

\begin{table}[t] 
\footnotesize
\center
\renewcommand{\arraystretch}{0.8}
\caption{Solution time (in seconds) over five randomly generated instances with $\gamma\in\{0.4,0.3\}$ using the Gini deviation. \textit{Note:} Solution times with `$>$' indicate that one or more instances cannot be solved within the imposed two-hour time limit.} \label{table:FLP_comp_time_Gini_deviation_add} 
\begin{tabular}{clrrrrrrrrrrrrrr}
\Xhline{1.0pt}
$|I|=40$,   $|J|=60$ &  &  &  & \multicolumn{3}{c}{$\gamma = 0.4$} &  & \multicolumn{3}{c}{$\gamma = 0.3$} \\ \cline{5-11}
 &  &  &  & min & avg & max &  & min & avg & max \\ \cline{5-7} \cline{9-11}
$p = 13$ & \multicolumn{2}{l}{Our Reformulation \eqref{eqn:FLP_dual_reformulation_Gini_deviation}} &  & 39 & 49 & 62 &  & 79 & 584 & 1461 \\
 & \multicolumn{2}{l}{Traditional Reformulation \eqref{eqn:FLP_direct_reformulation_Gini_deviation}} &  & 4 & 72 & 271 &  & 53 & 808 & 1804 \\ \Xhline{0.5pt}
 &  &  &  & min & avg & max &  & min & avg & max \\ \cline{5-7} \cline{9-11}
$p =    10$ & \multicolumn{2}{l}{Our Reformulation \eqref{eqn:FLP_dual_reformulation_Gini_deviation}} &  & 39 & 64 & 103 &  & 81 & 1569 & 4149 \\
 & \multicolumn{2}{l}{Traditional Reformulation \eqref{eqn:FLP_direct_reformulation_Gini_deviation}} &  & 14 & 137 & 269 &  & 292 & \multicolumn{1}{r}{\textgreater{}3020} & \multicolumn{1}{r}{\textgreater{}7200} \\ \Xhline{0.5pt}
 &  &  &  & min & avg & max &  & min & avg & max \\ \cline{5-7} \cline{9-11}
$p =    8$ & \multicolumn{2}{l}{Our Reformulation \eqref{eqn:FLP_dual_reformulation_Gini_deviation}} &  & 17 & 108 & 207 &  & 72 & 936 & 1829 \\
 & \multicolumn{2}{l}{Traditional Reformulation \eqref{eqn:FLP_direct_reformulation_Gini_deviation}} &  & 16 & 309 & 820 &  & 278 & \multicolumn{1}{r}{\textgreater{}4183} & \multicolumn{1}{r}{\textgreater{}7200} \\ \Xhline{1.0pt}
$|I|=80$, $|J|=20$ &  &  &  & \multicolumn{3}{c}{$\gamma = 0.4$} &  & \multicolumn{3}{c}{$\gamma = 0.3$} \\ \cline{5-11}
 &  &  &  & min & avg & max &  & min & avg & max \\ \cline{5-7} \cline{9-11}
$p = 7$ & \multicolumn{2}{l}{Our Reformulation \eqref{eqn:FLP_dual_reformulation_Gini_deviation}} &  & 59 & 130 & 269 &  & 795 & 1800 & 3557 \\
 & \multicolumn{2}{l}{Traditional Reformulation} &  & 14 & 189 & 464 &  & 784 & \multicolumn{1}{r}{\textgreater{}2860} & \multicolumn{1}{r}{\textgreater{}7200} \\ \Xhline{0.5pt}
 &  &  &  & min & avg & max &  & min & avg & max \\ \cline{5-7} \cline{9-11}
$p =    5$ & \multicolumn{2}{l}{Our Reformulation \eqref{eqn:FLP_dual_reformulation_Gini_deviation}} &  & 56 & 225 & 742 &  & 406 & 1259 & 1982 \\
 & \multicolumn{2}{l}{Traditional Reformulation \eqref{eqn:FLP_direct_reformulation_Gini_deviation}} &  & 38 & 162 & 559 &  & 1051 & 1858 & 2916 \\ \Xhline{0.5pt}
 &  &  &  & min & avg & max &  & min & avg & max \\ \cline{5-7} \cline{9-11}
$p =    4$ & \multicolumn{2}{l}{Our Reformulation \eqref{eqn:FLP_dual_reformulation_Gini_deviation}} &  & 52 & 87 & 134 &  & 127 & 336 & 745 \\
 & \multicolumn{2}{l}{Traditional Reformulation \eqref{eqn:FLP_direct_reformulation_Gini_deviation}} &  & 45 & 113 & 367 &  & 602 & 1076 & 2249 \\ \Xhline{1.0pt}
$|I|=80$, $|J|=35$ &  &  &  & \multicolumn{3}{c}{$\gamma = 0.4$} &  & \multicolumn{3}{c}{$\gamma = 0.3$} \\ \cline{5-11}
 &  &  &  & min & avg & max &  & min & avg & max \\ \cline{5-7} \cline{9-11}
$p = 7$ & \multicolumn{2}{l}{Our Reformulation \eqref{eqn:FLP_dual_reformulation_Gini_deviation}} &  & 78 & 353 & 1282 &  & 1988 & 3694 & 5918 \\
 & \multicolumn{2}{l}{Traditional Reformulation \eqref{eqn:FLP_direct_reformulation_Gini_deviation}} &  & 65 & 418 & 1790 &  & 5450 & \multicolumn{1}{r}{\textgreater{}6683} & \multicolumn{1}{r}{\textgreater{}7200} \\ \Xhline{0.5pt}
 &  &  &  & min & avg & max &  & min & avg & max \\ \cline{5-7} \cline{9-11}
$p =    5$ & \multicolumn{2}{l}{Our Reformulation  \eqref{eqn:FLP_dual_reformulation_Gini_deviation}} &  & 70 & 156 & 319 &  & 222 & 1863 & 5426 \\
 & \multicolumn{2}{l}{Traditional Reformulation \eqref{eqn:FLP_direct_reformulation_Gini_deviation}} &  & 44 & 648 & 1951 &  & 1107 & \multicolumn{1}{r}{\textgreater{}3970} & \multicolumn{1}{r}{\textgreater{}7200} \\ \Xhline{0.5pt}
 &  &  &  & min & avg & max &  & min & avg & max \\ \cline{5-7} \cline{9-11}
$p =    4$ & \multicolumn{2}{l}{Our Reformulation  \eqref{eqn:FLP_dual_reformulation_Gini_deviation}} &  & 70 & 126 & 217 &  & 143 & 740 & 1423 \\
 & \multicolumn{2}{l}{Traditional Reformulation} &  & 41 & 303 & 1106 &  & 490 & \multicolumn{1}{r}{\textgreater{}4558} & \multicolumn{1}{r}{\textgreater{}7200} \\  
\Xhline{1.0pt}
\end{tabular}
\end{table}

\section{Comparison with \texorpdfstring{\protect\cite{Lan_et_al:2010}}{ }} \label{appdx:comparison_with_Lan_et_al}

In this section, we compare our proposed axiomatic approach for the class of convex fairness measures (see Section~\ref{sec:convex_fairness_measures}) with \cite{Lan_et_al:2010}'s axiomatic approach for a class of fairness measures in resource allocation. Specifically, \cite{Lan_et_al:2010} considered resource allocation settings where one wants to quantify the degree of fairness associated with a given allocation vector $\ub\in\R_+^N$ ($u_i$ is the resource allocated to subject $i$). They construct a class of fairness measures  $\{h_N\}_{N\in\N}$ for this setting based on a set of five axioms and a generator function $g$, where $g$ must be an increasing and continuous function that leads to a well-defined mean function. To facilitate the discussion, we first summarize the axiomatic characterization of \cite{Lan_et_al:2010}'s class of fairness measures in Theorem~\ref{thm:Lan_et_al_rel_axiom} (we slightly modified \cite{Lan_et_al:2010}'s notation to align it with our notation for consistency).

\begin{theorem}[Theorems~1 and 2 of \cite{Lan_et_al:2010}] \label{thm:Lan_et_al_rel_axiom}
There exists a unique set of fairness measures $\{h_N:\R^N_+\rightarrow\R\}_{N\in\N}$ defined using a generator function $g$ such that $\{h_N\}_{N\in\N}$ satisfy the following axioms:
\begin{enumerate}[topsep=2mm,itemsep=0mm]
    \item Continuity: $h_N$ is continuous on $\R^N_+$;
    \item Homogeneity: $h_N(\alpha\ub)=h_N(\ub)$ for all $\alpha>0$ with $|h_1(u)|=1$ for all $u>0$;
    \item Asymptotic Saturation: $\lim_{N\rightarrow\infty} h_{N+1}(\one)/h_N(\one)=1$;
    \item Irrelevance of Partition: for any $\ub=(\ub^1,\ub^2)\in\R^N$ with $\ub^1\in\R^{N_1}$ and $\ub^2\in\R^{N_2}$, 
    $$h_N(\ub)=h_2(\one^\tp\ub^1,\,\one^\tp\ub^2)\cdot g^{-1} \Big( s_1\cdot g\big(h_{N_1}(\ub_1)\big) + s_2\cdot g\big(h_{N_2}(\ub_2)\big) \Big) $$
    for some $\{s_1,s_2\}\subset\R$ with $s_1+s_2=1$, and continuous and strictly monotonic function $g$ such that $m= g^{-1} \Big( s_1\cdot g\big(h_{N_1}(\ub_1)\big) + s_2\cdot g\big(h_{N_2}(\ub_2)\big) \Big) $ is a mean function of $\{h_{N_1}(\ub^1),h_{N_2}(\ub^2)\}$;
    \item Monotonicity: $h_2(u_1,u_2)$ is monotonically decreasing with $|u_1-u_2|$.
\end{enumerate}
\end{theorem}

As discussed in \cite{Lan_et_al:2010}, from any function $g(\cdot)$ satisfying Axiom~4, Axioms 1--5 generate a unique set $\{h_N\}_{N\in\N}$. Such $h_N$ is a well-defined fairness measure if it also satisfies Axioms 1--5. Hence, to derive a closed-form mathematical expression of $h_N$ that can be used in optimization models, one has to choose a suitable generator function $g$ and specify the values of $\{s_1,s_2\}$ in Axiom 4 such that $h_N$ satisfies Axioms 1--5. \cite{Lan_et_al:2010} focused on the special case where $g(\cdot)$  is a power function and derived the mathematical expression of $h_N$ as shown in Theorem~\ref{thm:Lan_et_al_rel_axiom_power}.

\begin{theorem}[Theorem~4 of \cite{Lan_et_al:2010}] \label{thm:Lan_et_al_rel_axiom_power}
If the generation function $g$ is chosen as the power function $g(y)=|y|^\beta$ with parameter $\beta$ and coefficients $s_i=(\one^\tp\ub^i)^\rho/[(\one^\tp\ub^1)^\rho+(\one^\tp\ub^2)^\rho]$ for $\rho \ne 0$ and $i\in\{1,2\}$, then Axioms~1--5 define a unique family of fairness measures as follows:
\begin{equation*}
    h_N(\ub)=\sgn(1-\beta r) \Bigg[\sum_{i=1}^N \Bigg(\frac{u_i}{\sum_{j=1}^N u_j}\Bigg)^{1-\beta r}\Bigg]^{\frac{1}{\beta}},
\end{equation*}
where $r=(1-\rho)/\beta$ and $\sgn(\cdot)$ is the sign function, i.e., $\sgn(a)=1$ if $a>0$ and $\sgn(a)=-1$ if $a<0$. In particular, if $r$ is chosen to $1$, then
\begin{equation*}
    h_N(\ub)=\sgn(1-\beta) \Bigg[\sum_{i=1}^N \Bigg(\frac{u_i}{\sum_{j=1}^N u_j}\Bigg)^{1-\beta}\Bigg]^{\frac{1}{\beta}}.
\end{equation*}
\end{theorem}

We are now ready to discuss the differences between our work and that of \cite{Lan_et_al:2010}. First, we emphasize that our proposed unified framework for convex fairness measures and \cite{Lan_et_al:2010}'s axiomatic approach to fairness measures consider different classes of fairness measures and different settings, and thus, there is no direct relationship between the two. As mentioned earlier, \cite{Lan_et_al:2010}'s axiomatic approach is based on the specific context of network resource allocation. Specifically, the outcome vector  $\ub\in\R_+^N$ in \cite{Lan_et_al:2010} can be interpreted as the non-negative resource allocation decision. In contrast, we propose a unified framework for a class of convex fairness measures suitable for various optimization contexts, and thus, it applies to a broader range of applications. Moreover, while \cite{Lan_et_al:2010}'s approach is limited to settings where the outcome vector is non-negative, our unified framework for convex fairness measures is applicable to general settings where the entries of the outcome vector can be negative. 

Second, \cite{Lan_et_al:2010} considers fairness in a relative sense, i.e., their class of fairness measures consists of relative measures (more on this below). In contrast, we consider convex fairness measures in an absolute sense and also introduce their relative counterpart. Third, we observe the following differences in the mathematical properties and interpretations between our convex fairness measure $\nu$ and the  fairness measure $h_N$:
\begin{itemize}[topsep=2mm,itemsep=2mm]
    \item The fairness measure $h_N$ may assume negative value depending on the choice of the generator function (see Theorem~\ref{thm:Lan_et_al_rel_axiom_power}).  In contrast, the convex fairness measure $\nu$ satisfies Axiom~\ref{axiom:normalization}, and thus, is always non-negative. As discussed in Section 3, Axiom~\ref{axiom:normalization} has the following intuitive interpretation: $\nu$ equals zero implies perfect fairness or equality.

    \item A larger value of $h_N$ corresponds to a fairer distribution of outcomes. In contrast, a smaller value of $\nu$ means a fairer distribution of outcomes. In particular, while the equal distribution $\ub=\one$ maximizes $h_N$, i.e., $h_N(\one)=\max_{\ub\in\R^N_+}h_N(\ub)$ (see Corollary~2 in \citealp{Lan_et_al:2010}), it minimizes $\nu$, i.e., $\nu(\one)=\min_{\ub\in\R^N}\nu(\ub)=0$, both implying perfect fairness. Hence, by construction, the fairness measure $h_N$ is Schur concave (see Theorem~3 in \citealp{Lan_et_al:2010}), and our convex fairness measure $\nu$ is Schur convex (Axiom~\ref{axiom:Schur_convex}).
\end{itemize}

Finally, \cite{Lan_et_al:2010} considers a slightly different set of axioms to define their class of relative fairness measures. In the following, we discuss similarities and differences between \cite{Lan_et_al:2010}'s axioms and some of those we use to define the class of convex fairness measures. (Recall that the class of convex fairness measures is defined by Axioms~\ref{axiom:continuity}, \ref{axiom:normalization}, \ref{axiom:symmetry}, \ref{axiom:trans_invariance}, \ref{axiom:pos_homo}, and \ref{axiom:convexity}, and its relative counterpart is defined by Axioms~\ref{axiom:normalization_rel}, \ref{axiom:symmetry}, \ref{axiom:Schur_convex}, and \ref{axiom:scale_invariance}.)
\begin{enumerate}[topsep=2mm,itemsep=0mm]
    \item \textit{Continuity.} This axiom is the same as Axiom~\ref{axiom:continuity}.
    
    \item \textit{Homogeneity.} This axiom is equivalent to scale invariance, i.e., Axiom~\ref{axiom:scale_invariance}, used in the definition of the relative convex fairness measure (see Definition~\ref{def:relative_convex_fairness_measure}).  As discussed in \ref{appdx:relative_convex_fairness_measures}, this axiom characterizes relative fairness measures and thus implies that \cite{Lan_et_al:2010}'s class of fairness measures $\{h_N\}_{N\in\N}$ is defined in the relative sense (see, e.g., Theorem~\ref{thm:Lan_et_al_rel_axiom_power}). In contrast, we define the class of convex fairness measures in the absolute sense and also derive its relative counterpart (which satisfies Axiom~\ref{axiom:scale_invariance}) in \ref{appdx:relative_convex_fairness_measures}. 
    
    \item \textit{Asymptotic Saturation.} As pointed out in \cite{Lan_et_al:2010}, this is a technical condition and not a known axiom for fairness measures. They adopt this condition to guarantee the uniqueness of the set of fairness measures defined in Theorem~\ref{thm:Lan_et_al_rel_axiom}. 
    
    \item \textit{Irrelevance of Partition.} This is another axiom specific for \cite{Lan_et_al:2010}'s class of fairness measures. It allows one to compute the value of the fairness measure of a high-dimensional outcome vector (i.e., with a large number of subjects) recursively from that of low-dimensional outcome vectors (i.e., with smaller numbers of subjects). As \cite{Chen_Hooker:2023} pointed out, it is still unclear ``how one might assess whether the rather abstract axiom of partition is appropriate for a particular practical application.''

    \item \textit{Monotonicity.} This axiom applies exclusively to the case where there are two subjects, i.e., $N=2$. Specifically, it indicates that $h_2$ decreases as the absolute difference $|u_1 - u_2|$ increases. In other words, a higher value of $|u_1 - u_2|$ implies a more unfair distribution of outcomes. Note that when there are two subjects, our convex fairness measure $\nu$ can be represented as $c|u_1 - u_2|$, where $c$ is a positive constant (see discussions after Theorem~\ref{thm:characterization_of_equivalence_CFM}). Hence, $\nu:\R^2\rightarrow\R_+$ is monotonoically increasing with $|u_1-u_2|$. It follows that $\nu(\ub)=c|u_1-u_2|$ satisfies the notion of monotonicity. 
\end{enumerate}

\section{Envy-Based Fairness Measures} \label{appdx:envy_based}

In this section, we analyze a general class of envy-based fairness measures and derive conditions under which a subset of them belongs to the class of convex fairness measures. Without loss of generality, we consider $u_i$ as the positive impact on subject $i \in [N]$, i.e., a larger value of $u_i$ is preferred. 

We first introduce a general and classical class of envy-based fairness measures widely adopted in the literature \citep{Aleksandrov_et_al:2019, Boiney:1995, Feldman_Kirman:1974, Tan_et_al:2023}.  Let $V_i:\R\rightarrow\R$ be the utility function of individual $i\in[N]$. That is, $V_i(y)\in\R$ is the utility or value that individual $i$ receives from a given impact $y\in\R$. We define the envy that  $i$ has towards $j$ for a given impact vector $\ub \in \R^N$ as
$$e_{ij}(\ub):=\big(V_i(u_j)-V_i(u_i)\big)_+=\begin{cases} 0, &\text{ if } V_i(u_i)\geq V_i(u_j), \\ V_i(u_j)-V_i(u_i), &\text{ if } V_i(u_i) < V_i(u_j). \end{cases}$$
Then, we define the envy-based fairness measure $E:\R^N\rightarrow\R$ as the total envy:
\begin{equation} \label{eqn:total_envy}
    E(\ub)=\sum_{i=1}^N \sum_{j=1}^N e_{ij}(\ub)=\sum_{i=1}^N \sum_{j=1}^N \big(V_i(u_j)-V_i(u_i)\big)_+.
\end{equation}

In Theorem~\ref{thm:envy_CFM}, we identify a set of envy-based fairness measures of the form \eqref{eqn:total_envy} that belong to the class of convex fairness measures.

\begin{theorem} \label{thm:envy_CFM}
Assume that the utility function $V_i:\R\rightarrow\R$ is continuous and non-decreasing for all $i\in[N]$. The total envy $E:\R^N\rightarrow\R$ defined in \eqref{eqn:total_envy} is a convex fairness measure if and only if $V_i(y)=cy+V_i(0)$ for some $c>0$ and for all  $i\in[N]$, i.e., 
\begin{equation} \label{eqn:total_envy_CFM}
E(\ub)=c\sum_{i=1}^N \sum_{j=1}^N \big(u_j-u_i\big)_+ = c\sum_{i=1}^N \sum_{j=i+1}^N \big|u_j-u_i\big|.
\end{equation}
\end{theorem}

\begin{proof}
First, suppose that $V_i(y)=cy+V(0)$. We show that $E(\ub)$ defined in \eqref{eqn:total_envy_CFM} is equivalent to the Gini deviation $\phi^{\mbox{\tiny Gini}}(\ub)=\sum_{i=1}^N \sum_{j=1}^N \big|u_j-u_i\big|$ and thus is a convex fairness measure. From the definition of $E(\ub)$, we have
\begin{align}
  E(\ub)=c\sum_{i=1}^N \sum_{j=1}^N \big(u_j-u_i\big)_+ &=c\sum_{i=1}^N \sum_{j=1}^{i-1} \big(u_j-u_i\big)_+  + c\sum_{i=1}^N \sum_{j=i+1}^{N} \big(u_j-u_i\big)_+  \nonumber \\
  &=c\sum_{j=1}^N \sum_{i=j+1}^{N} \big(u_j-u_i\big)_+  + c\sum_{i=1}^N \sum_{j=i+1}^{N} \big(u_j-u_i\big)_+  \nonumber \\
  &= c\sum_{i=1}^N \sum_{j=i+1}^N  \Big[\big(u_i-u_j\big)_+ + \big(u_j-u_i\big)_+ \Big] \nonumber\\
  &= c\sum_{i=1}^N \sum_{j=i+1}^N \big|u_j-u_i\big|  \label{eqn_pf:thm:envy_CFM_0} \\
  &= \frac{c}{2} \sum_{i=1}^N \sum_{j=1}^N \big|u_j-u_i\big| = \frac{c}{2}\phi^{\mbox{\tiny Gini}}(\ub), \nonumber 
\end{align}
where \eqref{eqn_pf:thm:envy_CFM_0} follows from the equality $(a)_++(-a)_+=|a|$ for any $a\in\R$. It follows that $E(\ub)$ in \eqref{eqn:total_envy_CFM} is a convex fairness measure.

Now, suppose that $E$ defined in \eqref{eqn:total_envy} is a convex fairness measure. We will show that $V_i(y)=cy+V(0)$ for some $c\in\R$ and for all $i\in[N]$ in two steps.

\textit{Step 1.} Consider any $(a,b)\in\R^2$. Let $\ub=(a,b,b,\dots,b)^\tp\in\R^N$ and its permutation $\ub'=(b,a,b,\dots,b)^\tp\in\R^N$. By Axiom~\ref{axiom:symmetry}, we have $E(\ub)=E(\ub')$. This implies that
\begin{equation} \label{eqn_pf:thm:envy_CFM_1}
    (N-1)\cdot \Big(V_1(b)-V_1(a)\Big)_+ +  \Big(V_2(a)-V_2(b)\Big)_+ = (N-1)\cdot \Big(V_2(b)-V_2(a)\Big)_+ +  \Big(V_1(a)-V_1(b)\Big)_+ 
\end{equation}
for any $(a,b)\in\R^2$. Consider the following three cases.
\begin{itemize}
    \item Suppose that $V_1(b)-V_1(a)>0$. Note that if $V_2(a)-V_2(b)>0$, the left-hand-side of \eqref{eqn_pf:thm:envy_CFM_1} is strictly greater than zero while the right-hand side of \eqref{eqn_pf:thm:envy_CFM_1} equals zero, which is a contradiction. Therefore, we must have $V_2(a)-V_2(b)\leq 0$. It then follows from \eqref{eqn_pf:thm:envy_CFM_1} that $V_1(b)-V_1(a)=V_2(b)-V_2(a)$.

    \item Suppose that $V_1(b)-V_1(a)<0$. Note that if $V_2(a)-V_2(b)<0$, the left-hand-side of \eqref{eqn_pf:thm:envy_CFM_1} equals zero while the right-hand side of \eqref{eqn_pf:thm:envy_CFM_1} is strictly greater than zero, which is a contradiction. Therefore, we must have $V_2(a)-V_2(b)\geq 0$. It then follows from \eqref{eqn_pf:thm:envy_CFM_1} that $V_2(a)-V_2(b)=V_1(a)-V_1(b)$.

    \item Suppose that $V_1(b)-V_1(a)=0$. It then follows from \eqref{eqn_pf:thm:envy_CFM_1} that $ \big(V_2(a)-V_2(b)\big)_+ = (N-1)\cdot \big(V_2(b)-V_2(a)\big)_+$, which implies $V_2(a)-V_2(b)=0$. Thus, we also have $V_2(a)-V_2(b)=V_1(a)-V_1(b)$.
\end{itemize}
Combining the three cases, we have that \eqref{eqn_pf:thm:envy_CFM_1} implies $V_1(b)-V_1(a)=V_2(b)-V_2(a)$, or equivalently, $\big(V_1-V_2\big)(b)=\big(V_1-V_2\big)(a)$ for any $(a,b)\in\R^2$. This shows that $V_1-V_2$ is a constant function, i.e., $V_1-V_2\equiv \tau_{12}$ for some $\tau_{12}\in\R$. Repeating the same argument, one can show that for any $i\in[N]$ and $j\in[N]\setminus\{i\}$, we have $V_i-V_j\equiv\tau_{ij}$ for some $\tau_{ij}\in\R$. That is, we have $V_i\equiv V + \tau_i$ for some function $V:\R\rightarrow\R$ and $\tau_i\in\R$ for all $i\in[N]$. Note that since $V_i$ is non-decreasing by assumption, $V$ is also non-decreasing. We can now write the fairness measure $E$ in \eqref{eqn:total_envy} as
\begin{align}
    E(\ub)&=\sum_{i=1}^N\sum_{j=1}^N\big(V(u_j)-V(u_i)\big)_+ \nonumber\\
    &=\sum_{i=1}^N\sum_{j=1}^{i-1}\big(V(u_j)-V(u_i)\big)_+ + \sum_{i=1}^N\sum_{j=i+1}^N\big(V(u_j)-V(u_i)\big)_+ \label{eqn_pf:thm:envy_CFM_2}\\
    &=\sum_{i=1}^N\sum_{j=i+1}^N \big|V(u_i)-V(u_j)\big|, \label{eqn_pf:thm:envy_CFM_3}
\end{align}
where \eqref{eqn_pf:thm:envy_CFM_2} follows from the equality $(a)_++(-a)_+=|a|$ for any $a\in\R$.

\textit{Step 2.} Consider any $(a,b)\in\R^2$. Let $\ub=(a,b,b,\dots,b)^\tp\in\R^N$. From Axiom~\ref{axiom:trans_invariance}, we have $E(\ub+\alpha\one)=E(\ub)$ for any $\alpha\in\R$. It follows from \eqref{eqn_pf:thm:envy_CFM_3} that
\begin{equation} \label{eqn_pf:thm:envy_CFM_5}
    \big|V(a+\alpha)-V(b+\alpha)\big| = \big|V(a)-V(b)\big| 
\end{equation}
for any $(a,b)\in\R^2$ and $\alpha\in\R$. By setting $\alpha=-b$, \eqref{eqn_pf:thm:envy_CFM_5} implies that
\begin{equation} \label{eqn_pf:thm:envy_CFM_6}
    \big|V(a-b)-V(0)\big| = \big|V(a)-V(b)\big| 
\end{equation}
for any $(a,b)\in\R^2$. Consider the following two cases.
\begin{itemize}
    \item Suppose that $a\leq b$, or equivalently, $a-b\leq 0$. Since $V$ is non-decreasing, we have $V(a)\leq V(b)$ and $V(a-b)\leq V(0)$. Therefore, \eqref{eqn_pf:thm:envy_CFM_6} reduces to $V(0)-V(a-b)=V(b)-V(a)$.

    \item Suppose that $a>b$, or equivalently, $a-b>0$. Since $V$ is non-decreasing, we have $V(a)\geq V(b)$ and $V(a-b)\geq V(0)$. Therefore, \eqref{eqn_pf:thm:envy_CFM_6} reduces to $V(a-b)-V(0)=V(a)-V(b)$.
\end{itemize}
Combining the two cases, we have that \eqref{eqn_pf:thm:envy_CFM_6} implies $V(a-b)=V(a)-V(b)+V(0)$ for any $(a,b)\in\R^2$. Note that this is a one-dimensional generalized Cauchy functional equation. Since $V$ is continuous, it follows from Corollary~2.4.1 of \cite{Castillo_Ruiz-Cobo:1992} that $V(y)=cy+V(0)$ for some $c\in\R$, and thus $V_i(y)=V(y)+\tau_i=cy+V(0)+\tau_i$ for all $i\in[N]$. This shows that $V_i(y)=cy+V_i(0)$ for all $i\in[N]$. Finally, since $V_i$ is non-decreasing, we have $c\geq 0$. Plugging $V_i(y)=cy+V_i(0)$ for all $i\in[N]$ in \eqref{eqn:total_envy}, we obtain $E(\ub)=c\sum_{i=1}^N \sum_{j=i+1}^N \big|u_j-u_i\big|$ as in \eqref{eqn:total_envy_CFM}. By Axiom~\ref{axiom:normalization}, $E(\ub)=0$ if and only if $\ub=\alpha\one$ for any $\alpha\in\R$, which implies that the parameter $c$ should be strictly greater than zero. This completes the proof.
\end{proof}

We remark the following on the utility function $V_i$. First, the assumption that $V_i$ is non-decreasing is mild. In particular, it is reasonable and common to expect that the utility individual $i$ receives, denoted as $V_i(y)$, is non-decreasing with the (positive) impact $y$, i.e., $V_i(y_1)\geq V_i(y_2)$ if $y_1\geq y_2$. Second, several studies have adopted affine utility functions in optimization models of the form~\eqref{prob:eff_min} (see, e.g., \citealp{Blanco_et_al:2024, Chanta_et_al:2011, Chanta_et_al:2014, Espejo_et_al:2009}). In such settings, the function $U$ can be interpreted as the utility function, i.e., $\ub=U(\xb)$ computes the utility of the subjects. 
Finally, we note that for the envy-based fairness measure, $E(\ub)$, to be a convex fairness measure and satisfy Axiom~\ref{axiom:symmetry}---which stipulates that the perceived degree of fairness should not depend on the identity of individuals---the proportional constant $c$ in the utility function $V_i(y)=cy+V_i(0)$ must be identical for all $i\in[N]$. That is, the additional utility or value an individual receives due to a unit increase in $y$ is the same for all individuals.

\section{Fairness-Promoting Optimization Models Based on the Rawlsian Principle} \label{appdx:Rawlsian}

In this section, we discuss how our unified framework can be leveraged to analyze fairness-promoting problems incorporating the Rawlsian principle \citep{Chen_Hooker:2023, Karsu_Morton:2015, Shehadeh_Snyder:2021}. Let us consider the following Rawlsian-based fairness-promoting optimization model:     
\begin{equation} \label{eqn:Rawlsian}
    \max_{\xb,\,\ub}\bigg\{ \min_{i\in[N]} \{u_i\} \,\,\bigg|\,\, \ub=U(\xb),\,\xb\in\calX\bigg\}.
\end{equation}%
Formulation \eqref{eqn:Rawlsian} finds $\xb$ and $\ub$ that maximize the minimum (positive) outcome across the $N$ subjects. Note that we can equivalently rewrite \eqref{eqn:Rawlsian} as follows:
\begin{equation} \label{eqn:Rawlsian_min}
    \min_{\xb,\,\ub}\bigg\{ -\min_{i\in[N]} \{u_i\} \,\,\bigg|\,\, \ub=U(\xb),\,\xb\in\calX\bigg\}=\min_{\xb,\,\ub}\bigg\{ \max_{i\in[N]} \{-u_i\} \,\,\bigg|\,\, \ub=U(\xb),\,\xb\in\calX\bigg\}.
\end{equation}%
Letting $\ubar=N^{-1} (\one^\tp\ub)$ be the mean of $\ub$, we can rewrite the objective function in \eqref{eqn:Rawlsian_min} as
\begin{equation*} 
    \max_{i\in[N]} \{-u_i\} =  - \ubar + \max_{i\in[N]} \{-u_i+\ubar\} = f(\ub) + \nu(\ub),
\end{equation*}%
where $f(\ub)=-\ubar$ and 
\begin{equation} \label{eqn:Rawlsian_fairness_measure}
     \nu(\ub)=\max_{i\in[N]} \{-u_i+\ubar\}=\ubar - \min_{i\in[N]}\{u_i\}.
\end{equation}%
Note that minimizing $f(\ub)=-\ubar$ is equivalent to maximizing the average outcome. Thus, $f(\ub)$ represents an efficiency measure. It is straightforward to verify that $\nu$ in \eqref{eqn:Rawlsian_fairness_measure} satisfies Axioms C, N, S, TI, PH, and CV.  Thus, $\nu$ is a convex fairness measure (see Definition~2). Therefore, we can rewrite the objective of the Rawlsian optimization model in the form \eqref{eqn:Rawlsian_min} as the sum of an efficiency measure $f$ and a convex fairness measure $\nu$. Hence, one can use our proposed unified framework for solving and analyzing the Rawlsian optimization problem of the form \eqref{eqn:Rawlsian}.

\newpage
\bibliographystyle{elsarticle-harv}
\bibliography{references}


\end{document}